\documentclass[amstex,12pt,russian,amssymb]{article}

\usepackage{mathtext}
\usepackage[cp1251]{inputenc}
\usepackage[T2A]{fontenc}
\usepackage[russian]{babel}
\usepackage[dvips]{graphicx}
\usepackage{amsmath}
\usepackage{amssymb}
\usepackage{amsxtra}
\usepackage{latexsym}
\usepackage{ifthen}

\begin{document}

\newcounter{lemma}
\newcommand{\lemma}{\par \refstepcounter{lemma}%
{\bf Лемма \arabic{lemma}.}}

\newcounter{corollary}
\newcommand{\corollary}{\par \refstepcounter{corollary}%
{\bf Следствие \arabic{corollary}.}}

\newcounter{remark}
\newcommand{\remark}{\par \refstepcounter{remark}%
{\bf Замечание \arabic{remark}.}}

\newcounter{theorem}
\newcommand{\theorem}{\par \refstepcounter{theorem}%
{\bf Теорема \arabic{theorem}.}}

\newcounter{proposition}
\newcommand{\proposition}{\par \refstepcounter{proposition}%
{\bf Предложение \arabic{proposition}.}}

\renewcommand{\refname}{\centerline{\bf Список литературы}}

\newcommand{\proof}{{\it Доказательство.\,\,}}

\noindent УДК 517.5

{\bf Д.П.~Ильютко} (МГУ имени М.\,В.\,Ломоносова),

{\bf Е.А.~Севостьянов} (Житомирский государственный университет им.\
И.~Франко)

\medskip
{\bf D.P.~Ilyutko} (M.\,V.\,Lomonosov Moscow State University),

{\bf E.A.~Sevost'yanov} (Zhytomyr Ivan Franko State University)

\medskip
{\bf О граничном поведении отображений на римановых многообразиях в
терминах простых концов}

\medskip
{\bf On boundary behavior of mappings on Riemannian manifolds in
terms of prime ends}

\medskip\medskip
Изучается граничное поведение классов кольцевых отображений на
римановых многообразиях, являющихся обобщением квазиконформных
отображений по Герингу. В терминах простых концов регулярных
областей получены теоремы о непрерывном продолжении указанных
классов на границу области. В этих же терминах доказаны результаты о
равностепенной непрерывности классов этих отображений в замыкании
заданной области.

\medskip
A boundary behavior of ring mappings on Riemannian manifolds, which
are generalization of quasiconformal mappings by Gehring, is
investigated. In terms of prime ends, we obtain theorems about
continuous extension to a boundary of classes mentioned above. In
the terms mentioned above, we prove results about equicontinuity of
these classes in the closure of the domain.

\newpage
{\bf 1. Введение.} Настоящая работа посвящена изучению отображений с
ограниченным и конечным искажением, активно изучаемых в последнее
время в ряде работ отечественных и зарубежных авторов, см., напр.,
\cite{ABBS}--\cite{Vu}. Отдельного внимания заслуживают работы, в
которых изложены результаты, относящиеся к изучению классов
Орлича-Соболева в окрестности границы заданной области в терминах
простых концов (см. \cite{KR} и \cite{GRY}). Здесь же упомянем
публикации, в которых исследовано граничное поведение этих классов в
случае локально связных границ (см., напр., \cite{MRSY},
\cite{KRSS}--\cite{Sev$_3$}).

\medskip
Остановимся более подробно на работах \cite{KR} и \cite{GRY}, в
которых речь идёт о граничном поведении гомеоморфизмов,
удовлетворяющих определённым геометрическим условиям. Главная
особенность этих двух работ состоит в том, что граничное
соответствие между областями, установленное в них, понимается в
терминах так называемых простых концов. Указанная особенность
исключительно важна с точки зрения степени общности рассматриваемых
объектов, так как даже для конформных отображений единичного круга
на некоторую область плоскости никакого поточечного соответствия
между границами областей может не быть. Подчеркнём также, что
результаты этих работ получены только для случая евклидового
$n$-мерного пространства и случай римановых многообразий здесь не
учтён. С другой стороны, случай последних рассмотрен в статьях
\cite{ARS}--\cite{IS}, однако, здесь речь идёт только о поточечном
соответствии и <<хороших>> границах. Естественно может быть
поставлен вопрос о справедливости соответствующих результатов на
римановых многообразиях в ситуации, когда граничное соответствие
устанавливается только по простым концам. Как мы уже заметили, эта
ситуация эквивалентна рассмотрению границ с весьма большой
<<степенью испорченности>>, однако, и для случая многих <<хороших>>
границ соответствие по простым концам вполне может иметь место.
Последнее относится, например, к случаю так называемых локально
квазиконформных границ, о которых неоднократно упоминается далее.

\medskip В настоящей статье речь идёт исключительно о граничном
соответствии между областями, понимаемом в смысле простых концов.
Рукопись является в некотором смысле итоговой работой, которая
подводит черту под исследованиями публикаций \cite{KR}, \cite{GRY},
\cite{ARS}--\cite{IS}. Основная её цель -- изложить наиболее важные
результаты, касающиеся граничного поведения отображений между
римановыми многообразиями в случае плохих границ. Рассмотрение
произвольных отображений не входит в наши ближайшие планы, так как
это потребовало бы наличия неких универсальных методов исследования,
существование которых не установлено. В силу этого, мы
ограничиваемся ситуацией квазиконформных отображений и их обобщений,
которые могут быть исследованы методом модулей. Пионерской работой в
этом направлении является известная статья Някки \cite{Na}, в
которой впервые был решён вопрос о граничном продолжении
квазиконформных отображений пространства ${\Bbb R}^n$ в терминах
простых концов. Насколько нам известно, квазиконформный случай так и
не был перенесен на римановы многообразия, что, впрочем, сделано
нами в настоящем тексте.

\medskip
Приведём теперь необходимые для изложения сведения. Всюду ниже
${\Bbb M}^n$ и ${\Bbb M}^n_*$~--- римановы многообразия размерности
$n\geqslant 2$ с геодезическими расстояниями $d$ и $d_*,$
соответственно, $D, D^{\,\prime}$ -- области, принадлежащие ${\Bbb
M}^n$ и ${\Bbb M}_*^n,$ соответственно, а $Q:{\Bbb M}^n\rightarrow
[0, \infty]$ -- измеримая относительно меры объёма функция, равная
нулю вне заданной области $D$ (при этом, при всех $x\in D$ мы
предполагаем, что $0<Q(x)<\infty$).  Мы считаем далее известными
понятия римановой метрики, геодезического расстояния, объёма и длины
на многообразии (см. \cite{IS}). По умолчанию замыкание
$\overline{A}$ и граница $\partial A$ множества $A\subset {\Bbb
M}^n$ понимаются в смысле геодезического расстояния $d(x, y)$ на
${\Bbb M}^n.$ Мы также считаем известными понятия кривых и модуля
семейств кривых (поверхностей), которые для многообразий также могут
быть найдены в работе \cite{IS} (см. также классическую работу
Фугледе по этому поводу \cite{Fu}).

\medskip
Всюду далее, если не оговорено противное,
$|x|=\sqrt{x_1^2+\ldots+x_n^2},$ где $x=(x_1, \ldots, x_n).$ Для
удобства положим
$${\Bbb B}^n_+:=\{x=(x_1, \ldots,
x_n)\in {\Bbb R}^n: |x|<1, x_n>0\}\,,$$
$${\Bbb
B}^{n-1}:=\{x=(x_1, \ldots, x_n)\in {\Bbb R}^n: |x|<1, x_n=0\}\,,$$
$$S_+(x_0, r)=\{x\in {\Bbb R}^n: x=(x_1,\ldots, x_n), |x-x_0|=r, x_n>0\}\,,$$
$$B_+(x_0, r)=\{x\in {\Bbb R}^n: x=(x_1,\ldots, x_n), |x-x_0|<r, x_n>0\}\,.$$
Всюду далее
$$B(x_0, r)=\{x\in {\Bbb M}^n: d(x, x_0)<r\}\,,$$
$$S(x_0, r)=\{x\in {\Bbb M}^n: d(x, x_0)=r\}\,,$$
$$A=A(x_0,
r_1, r_2)=\{x\in {\Bbb M}^n\,|\,r_1<d(x, x_0)<r_2\}\,.$$
Для евклидовых шаров и сфер в пространстве ${\Bbb R}^n$ мы также
используем обозначения $B(x_0, r),$ $S(x_0, r),$ что отдельно не
оговаривается за исключениям случаев, допускающих двоякое
толкование. Отображение $f:D\rightarrow D^{\,\prime}$ между
областями $D\subset {\Bbb M}^n$ и $D^{\,\prime}\subset {\Bbb M}_*^n$
будем называть {\it квазиконформным}, если для каждого семейства
кривых $\Gamma$ в области $D$ и некоторой постоянной $K\geqslant 1$
мы имеем
$$(1/K)\cdot M(\Gamma)\leqslant M(f(\Gamma))\leqslant
K\cdot M(f(\Gamma))\,,$$
где, как обычно, $M(\Gamma)$ обозначает модуль семейства кривых
$\Gamma$ (см., напр., \cite[раздел~1.8]{IS}). Следующие определения
могут быть найдены в работе \cite{KR}. Пусть $\omega$ -- открытое
множество в ${\Bbb R}^k$, $k=1,\ldots,n-1$. Непрерывное отображение
$\sigma:\omega\rightarrow{\Bbb M}^n$ называется {\it $k$-мерной
поверхностью} в ${\Bbb M}^n$. {\it Поверхностью} будет называться
произвольная $(n-1)$-мерная поверхность $\sigma$ в ${\Bbb M}^n.$
Поверхность $\sigma:\omega\rightarrow D$ называется {\it жордановой
поверхностью} в $D$, если $\sigma(z_1)\ne\sigma(z_2)$ при $z_1\ne
z_2$. Далее мы иногда будем использовать $\sigma$ для обозначения
всего образа $\sigma(\omega)\subset {\Bbb M}^n$ при отображении
$\sigma$, $\overline{\sigma}$ вместо $\overline{\sigma(\omega)}$ в
${\Bbb M}^n$ и $\partial\sigma$ вместо
$\overline{\sigma(\omega)}\setminus\sigma(\omega)$. Жорданова
поверхность $\sigma$ в $D$ называется {\it разрезом} области $D$,
если $\sigma$ разделяет $D$, т.\,е. $D\setminus \sigma$ имеет больше
одной компоненты, $\partial\sigma\cap D=\varnothing$ и
$\partial\sigma\cap\partial D\ne\varnothing$.

Последовательность $\sigma_1,\sigma_2,\ldots,\sigma_m,\ldots$
разрезов области $D$ называется {\it цепью}, если:

\medskip

%

(i) множество $\sigma_{m+1}$ содержится в точности в одной
компоненте $d_m$ множества $D\setminus \sigma_m$, при этом,
$\sigma_{m-1}\subset D\setminus (\sigma_m\cup d_m)$;

\medskip

(ii) $\cap\,d_m=\varnothing$, где $d_m$ -- компонента $D\setminus
\sigma_m$, содержащая $\sigma_{m+1}$.

\medskip
Согласно определению, цепь разрезов $\{\sigma_m\}$ определяет цепь
областей $d_m\subset D$, таких, что $\partial\,d_m\cap
D\subset\sigma_m$ и $d_1\supset d_2\supset\ldots\supset
d_m\supset\ldots$. Две цепи разрезов $\{\sigma_m\}$ и
$\{\sigma_k^{\,\prime}\}$ называются {\it эквивалентными}, если для
каждого $m=1,2,\ldots$ область $d_m$ содержит все области
$d_k^{\,\prime}$ за исключением конечного числа, и для каждого
$k=1,2,\ldots$ область $d_k^{\,\prime}$ также содержит все области
$d_m$ за исключением конечного числа. {\it Конец} области $D$ -- это
класс эквивалентных цепей разрезов $D$.

Пусть $K$ -- конец области $D$ в ${\Bbb M}^n$, $\{\sigma_m\}$ и
$\{\sigma_m^{\,\prime}\}$ -- две цепи в $K$, $d_m$ и
$d_m^{\,\prime}$ -- области, соответствующие $\sigma_m$ и
$\sigma_m^{\,\prime}$. Тогда
$$\bigcap\limits_{m=1}\limits^{\infty}\overline{d_m}\subset
\bigcap\limits_{m=1}\limits^{\infty}\overline{d_m^{\,\prime}}\subset
\bigcap\limits_{m=1}\limits^{\infty}\overline{d_m}\ ,$$ и, таким
образом,
$$\bigcap\limits_{m=1}\limits^{\infty}\overline{d_m}=
\bigcap\limits_{m=1}\limits^{\infty}\overline{d_m^{\,\prime}}\ ,$$
т.\,е. множество
%
$$I(K)=\bigcap\limits_{m=1}\limits^{\infty}\overline{d_m}$$
%
зависит только от $K$ и не зависит от выбора цепи разрезов
$\{\sigma_m\}$. Множество $I(K)$ называется {\it телом конца} $K$.

Хорошо известно, что $I(K)$ является континуумом, т.\,е. связным
компактным множеством, см., напр., \cite[I(9.12)]{Wh}. Кроме того,
ввиду условий (i) и (ii),  имеем, что
$$I(K)=\bigcap\limits_{m=1}\limits^{\infty}(\partial d_m\cap\partial D)=
\partial D\ \cap\ \bigcap\limits_{m=1}\limits^{\infty}\partial d_m\ .$$
Таким образом, получаем следующее утверждение.

\begin{proposition}\label{thabc1} {\sl Для каждого конца $K$ области $D$
в ${\Bbb M}^n$
$$I(K)\subset\partial D.$$}
\end{proposition}
Всюду далее, как обычно, $\Gamma(E, F, D)$ обозначает семейство всех
кривых $\gamma:[a, b]\rightarrow D$ таких, что $\gamma(a)\in E$ и
$\gamma(b)\in F.$ Следуя \cite{Na}, будем говорить, что конец $K$
является {\it простым концом}, если $K$ содержит цепь разрезов
$\{\sigma_m\}$, такую, что
\begin{equation}\label{eq5}
M(\Gamma(\sigma_m, \sigma_{m+1}, D))<\infty \quad\forall \quad m\in
{\Bbb N}
\end{equation}
и
\begin{equation}\label{eqSIMPLE}
\lim\limits_{m\rightarrow\infty}M(\Gamma(C, \sigma_m, D))=0
\end{equation}
для произвольного континуума $C$ в $D.$ В дальнейшем используются
следующие обозначения: множество простых концов, соответствующих
области $D,$ обозначается символом $E_D,$ а пополнение области $D$
её простыми концами обозначается $\overline{D}_P.$

\medskip
Будем говорить, что граница области $D$ в ${\Bbb M}^n$ является {\it
локально квазиконформной}, если каждая точка $x_0\in\partial D$
имеет окрестность $U$, которая может быть отображена квазиконформным
отображением $\varphi$ на единичный шар ${\Bbb B}^n\subset{\Bbb
R}^n$ так, что $\varphi(\partial D\cap U)$ является пересечением
${\Bbb B}^n$ с координатной гиперплоскостью. Рассмотрим также
следующее определение (см. \cite{KR}). Будем называть цепь разрезов
$\{\sigma_m\}$ {\it регулярной}, если $d(\sigma_{m})\rightarrow 0$
при $m\rightarrow\infty.$ Если конец $K$ содержит по крайней мере
одну регулярную цепь, то $K$ будем называть {\it регулярным}.
Говорим, что ограниченная область $D$ в ${\Bbb M}^n$ {\it
регулярна}, если $D$ может быть квазиконформно отображена на область
с локально квазиконформной границей и, кроме того, каждый простой
конец $P\subset E_D$ является регулярным. Заметим, что в
пространстве ${\Bbb R}^n$ все простые концы регулярной области
регулярны, и наоборот (см. \cite[теорема~5.1]{Na}).

\medskip
Говорят, что семейство кривых $\Gamma_1$ {\it минорируется}
семейством $\Gamma_2,$ пишем $\Gamma_1\,>\,\Gamma_2,$
если для каждой кривой $\gamma\,\in\,\Gamma_1$ существует подкривая,
которая принадлежит семейству $\Gamma_2.$
В этом случае,
\begin{equation}\label{eq32*A}
\Gamma_1
> \Gamma_2 \quad \Rightarrow \quad M_p(\Gamma_1)\leqslant M_p(\Gamma_2)
\end{equation} (см. \cite[теорема~1(c)]{Fu}).

\medskip
{\bf 2. Аналог лемм Някки для многообразий}. Для дальнейшего
изложения нам необходимы вспомогательные утверждения о соответствии
простых концов между областями, одна из которых является
квазиконформным образом области с локально квазиконформной границей.
Для пространства ${\Bbb R}^n$ такие утверждения известны и доказаны
Някки в его работе \cite[теорема~4.1]{Na}. Поскольку справедливость
этих результатов на многообразиях нам неизвестна, следует установить
эти результаты путём прямого доказательства. Как и прежде, {\it
геодезическим расстоянием} между множествами $F, F^{\,*}\subset
D\subset {\Bbb M}^n$ будем называть величину
$$d(F, F^{\,*}):=\inf\limits_{x\in F, y\in F^{\,*}}d(x, y)\,,$$
где $d(x, y)$ -- геодезическое расстояние. Также геодезическим
диаметром $F$ будем называть величину
$$d(F):=\sup\limits_{x, y\in F}d(x, y)\,.$$
Как обычно, $M_p$ обозначает $p$-модуль семейств кривых на
многообразии (см., напр., \cite[раздел~1.3]{IS$_1$}). Для
дальнейшего изложения будет полезным следующее определение (см.
\cite[разд.~13.3]{MRSY}). Будем говорить, что граница $\partial D$
области $D$ {\it сильно достижима в точке $x_0\in
\partial D$ относительно $p$-модуля}, если для любой окрестности $U$ точки $x_0$ найдется
компакт $E\subset D,$ окрестность $V\subset U$ точки $x_0$ и число
$\delta
>0$ такие, что
\begin{equation}\label{eq1}
M_p(\Gamma(E,F, D))\geqslant \delta
\end{equation}
для любого континуума $F$ в $D,$ пересекающего $\partial U$ и
$\partial V.$ Если $p=n,$ приставка <<$p$-модуля>> по отношению к
(\ref{eq1}), как правило, опускается.

Смысл условия (\ref{eq1}) состоит в том, что при приближении
континуума $F$ фиксированного диаметра к точке границы области,
модуль семейств кривых, соединяющих этот континуум с некоторым
фиксированным компактом, не стремится к нулю. Указанное свойство
имеет место, в частности, для <<хороших>> областей в ${\Bbb R}^n:$
единичного шара, всего пространства ${\Bbb R}^n,$ а также любой
ограниченной выпуклой области. Несложно также привести примеры
областей, в которых данное условие нарушается. Приведём один из
таких примеров.

{\bf Пример 1.} Для простоты рассмотрим случай плоскости ${\Bbb
R}^2,$ и пусть $p=n.$ Рассмотрим квадрат
$$\prod=(0,1)\times (0,1)=\{z=(x, y)\in {\Bbb R}^2: x\in (0, 1), y\in
(0,1)\}\,,$$
на основе котрого построим некоторую область $D$ следующим образом.
Пусть $I_k$ -- отрезок
$$I_k=\{z=(x, y): x=1/k, 0\leqslant y\leqslant 1/2\}\,,$$ где для каждого
$I_k$ число $k=2,3,\ldots$ -- фиксировано. Полагаем
$$D:=\prod\setminus \bigcup\limits_{k=2}^{\infty}I_k\,,$$
(см. рисунок 1).
\begin{figure}[h]
\centerline{\includegraphics[scale=0.9]{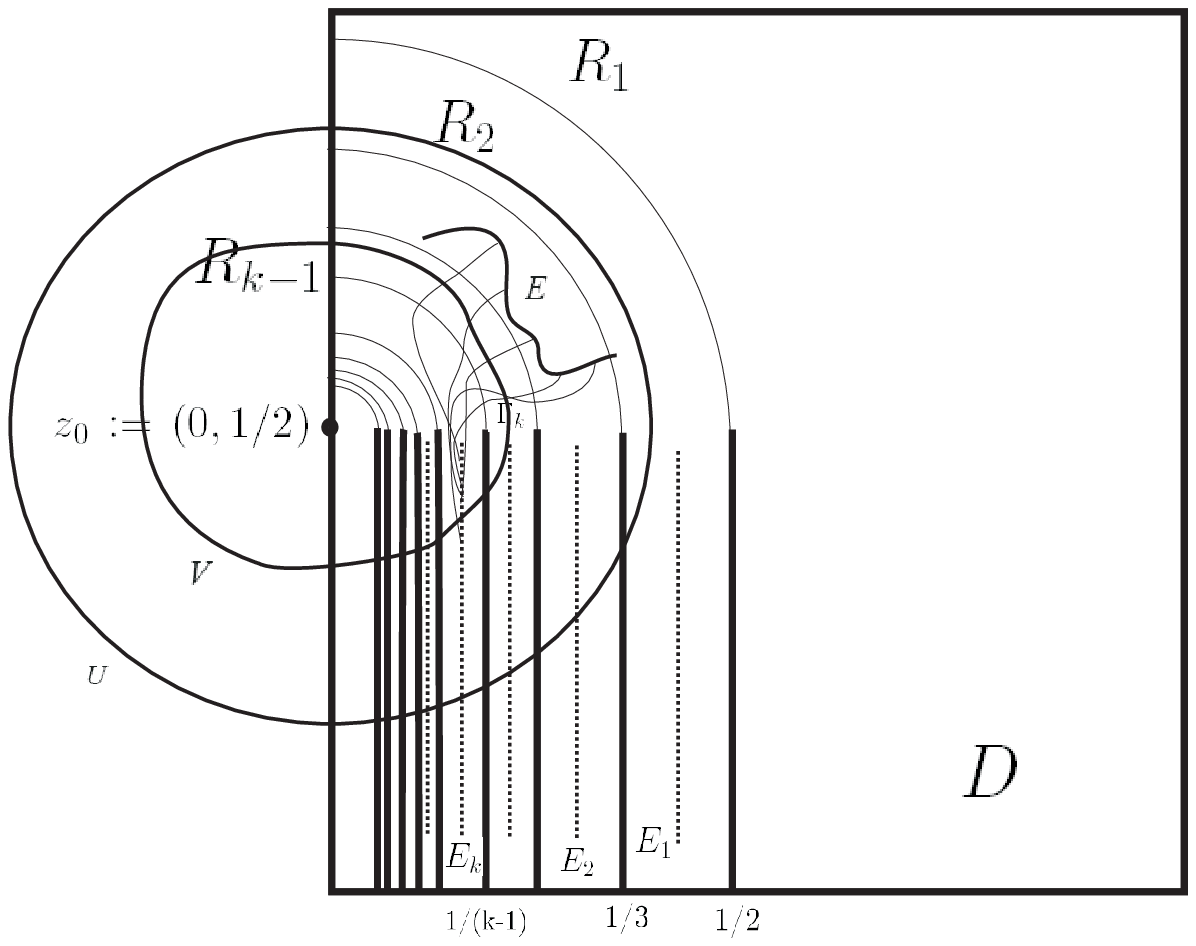}} \centerline{Рис. 1.
Пример области, не имеющей сильно достижимой границы}
\end{figure}
Покажем, что область $D$ не имеет сильно достижимой границы,
например, в точке $z_0=(0, 1/2).$ Согласно определению, нам следует
показать, что найдётся хотя бы одна окрестность $U$ точки $z_0$
такая, что для любой окрестности $V\subset U,$ произвольного
компакта $E\subset D$ и числа $\delta
>0$ найдётся континуум $F,$ пересекающий $\partial U$ и $\partial
V,$ для которого выполнено неравенство
\begin{equation}\label{eq1G}
M(\Gamma(E,F, D))<\delta\,.
\end{equation}
Выберем в качестве $U:=B(z_0, 1/3),$ где, как прежде, $z_0=(0,
1/2).$ Зафиксируем окрестность $V\subset U,$ компакт $E\subset D$ и
число $\delta>0.$ Рассмотрим последовательность $R_k=1/k,$
$k=1,2,\ldots .$ Поскольку $E$ -- компакт в $D,$ найдётся $k_0>0$
такое, что $E\cap B(z_0, R_{k_0-1})=\varnothing.$ Мы можем также
считать, что все точки $z=(x, y)\in E$ удовлетворяют условию:
$x>1/(k_0-1).$ Рассмотрим также последовательность континуумов
$E_k,$ определённых следующим образом:
$$E_k=\left\{z=(x, y): x=\frac{1}{2}\left(\frac{1}{k}+\frac{1}{(k-1)}\right), \frac18\leqslant y\leqslant v_0\right\}\,,$$
где $v_0$ -- произвольное положительное число, удовлетворяющее
условиям: $v_0<1/2,$ $1/2-v_0<{\rm dist}\,(z_0,
\partial V).$
Ясно, что $E_k$ -- континуумы в $D.$ Покажем, что при некотором
$k_1\in {\Bbb N},$ $k_1>k_0,$ и всех натуральных $k\in {\Bbb N}$
континуумы $E_k$ удовлетворяют условиям: $E_k\cap \partial
U\ne\varnothing$ и $E_k\cap \partial V\ne\varnothing.$ В самом деле,
пусть
$$x_k=\left(\frac{1}{2}\left(\frac{1}{k}+\frac{1}{(k-1)}\right),
\frac18\right),
y_k=\left(\frac{1}{2}\left(\frac{1}{k}+\frac{1}{(k-1)}\right),
v_0\right)$$ -- <<самая нижняя>> и <<самая верхняя>> точки
континуума $E_k,$ соответственно. Тогда
$$|z_0-x_k|\rightarrow 3/8, \quad |z_0-y_k|\rightarrow 1/2-v_0, \quad k\rightarrow\infty\,.$$
Следовательно, $x_k\not\in B(z_0, 1/3)=U$ и $y_k\in V$ при
достаточно больших $k>k_0,$ тем самым, при тех же $k$ имеем:
$E_k\cap (D\setminus U)\ne \varnothing$ и $E_k\cap V\ne
\varnothing.$ В таком случае, ввиду \cite[теорема~1.I.5,
$\S\,46$]{Ku} выполнены соотношения $E_k\cap \partial
U\ne\varnothing$ и $E_k\cap \partial V\ne\varnothing.$ Пусть $k_1$
-- наименьшее из чисел $k\geqslant k_0,$ при которых указанные
соотношения имеют место.

Пусть $L_k$ -- сектор круга $B(z_0, R_{k-1}),$ определённый
следующим образом: $$L_k=\{z=(x, y)=z_0+re^{i\varphi}: 0<r<R_{k-1},
\varphi\in [0, \pi/2)\}\,.$$ Пусть также $$P_k=\{z=(x, y): z\in D:
0<x<1/(k-1), 0<y\leqslant 1/2\}\,.$$ Положим $D_k:=L_k\cup P_k.$
Ясно, что $D_k$ является областью, при этом, $E_k\in D_k.$ Кроме
того, построению, $E\in D\setminus D_k$ при всех $k>k_0.$

Пусть $\Gamma_k$ -- семейство кривых, соединяющих $E_k$ и $E$ в
области $D.$ В дальнейшем, как обычно, $|\gamma|$ -- носитель кривой
$\gamma.$ Если $\gamma\in \Gamma_k,$ то $|\gamma|\cap
D_k\ne\varnothing\ne |\gamma|\cap (D\setminus D_k),$ поэтому ввиду
\cite[теорема~1.I.5, $\S\,46$]{Ku} это возможно лишь в случае
$|\gamma|\cap \partial D_k\ne \varnothing.$ Так как кривая $\gamma$
принадлежит области $D,$ то из последнего условия вытекает наличие
подкривой $\gamma_1<\gamma,$ соединяющей $S(z_0, R_{k-1})$ c $E,$
$k>k_0.$ Путём аналогичных рассуждений заключаем, что у кривой
$\gamma$ имеется подкривая, соединяющая множества $S(z_0, R_{k-1})$
и $S(z_0, R_{k_0-1})$ в ${\Bbb R}^2.$ Таким образом,
\begin{equation}\label{eq19}
\Gamma(E, E_k, D)>\Gamma(S(z_0, R_{k-1}), S(z_0, R_{k_0-1}), {\Bbb
R}^2)\,,\qquad k>k_0.
\end{equation}
Хорошо известно (см. \cite[теорема~7.5]{Va}), что
\begin{equation}\label{eq20}
M(\Gamma(S(x_0, R_{k-1}), S(x_0, R_{k_0-1}), {\Bbb
R}^2))=\frac{2\pi}{\log\frac{R_{k_0-1}}{R_{k-1}}}\rightarrow
0\,,\quad k\rightarrow \infty\,.
\end{equation}
Таким образом, из (\ref{eq19}) и (\ref{eq20}) ввиду минорирования
модуля вытекает, что
$$M(\Gamma(E, E_k, D))\rightarrow 0\,, \quad k\rightarrow\infty.$$
Из последнего соотношения вытекает существование натурального числа
$k_2>k_1,$ такого что $M(\Gamma(E, E_k, D))<\delta.$ Теперь положим
$E=E_{k_2}.$ В силу установленных выше свойств множеств $E_k,$
$k>k_1,$ континуум $E$ удовлетворяет условиям $E\cap
\partial U\ne\varnothing$ и $E\cap
\partial V\ne\varnothing,$ кроме того, $E$ удовлетворяет условию
(\ref{eq1G}). Таким образом, то, что граница области $D$ не является
сильно достижимой в точке $z_0,$ установлено.

\medskip
Будем говорить, что граница $\partial D$ области $D$
является {\it слабой плоской в точке $x_0\in
\partial D,$} если для любой окрестности $U$ точки $x_0$ и для каждого $P>0$ найдется
окрестность $V\subset U$ точки $x_0$ такая, что для любых двух
континуумов $F$ и $G,$ пересекающих $\partial U$ и $\partial V,$
выполняется неравенство
\begin{equation}\label{eq3}
M(\Gamma(E,F, D))\geqslant P\,.
\end{equation}
Граница $\partial D$ области $D$ будет называться {\it слабой
плоской}, если она является слабо плоской в каждой точке $x_0\in D.$
Непосредственно из определения вытекает, что области со слабо
плоскими границами являются также областями с сильно достижимыми
границами, поэтому область из примера 1, в частности, не является
слабо плоской.

Смысл свойства слабой плоскости границы заключается в том, что
модуль семейства кривых, соединяющих континуумы с не стремящимися к
нулю диаметрами вблизи граничной точки, близок к бесконечности.
Опять-таки, данный факт хорошо известен для <<хороших>> областей
(единичного круга, ограниченных выпуклых областей, всего
пространства ${\Bbb R}^n$ и т.п.). <<Слабая плоскость>> имеет и
довольно ясную физическую интерпретацию: если представить себе $E$ и
$F$ как пару заряженных пластин, сближающихся друг с другом, то
ёмкость соответствующего <<хорошего>> конденсатора обязана быть
сколь угодно большой. (Не лишним будет напомнить, что ёмкость
конденсатора численно равна модулю соответствующего семейства
кривых. Разница между модулем и ёмкостью является сугубо
терминологической). Наличие <<диэлектрика>> между пластинами --
разреза в заданной области -- {\it существенно меняет ситуацию,} что
можно явно продемонстрировать на следующем примере.

\medskip
{\bf Пример 2.} Рассмотрим на плоскости единичный круг ${\Bbb B}^n$
с разрезом
$$I=[0, 1]=\{z=(x, y)\in {\Bbb R}^2: 0\leqslant x\leqslant 1,
y=0\}\,.$$
Иными словами, положим $D:={\Bbb B}^n\setminus I.$ Покажем, что
заданная область не является слабо плоской в точке $z_0=(1, 0).$ По
определению слабо плоской границы, мы должны показать, что хотя бы
для одной окрестности $U$ этой точки и некоторого числа $P>0$
выполнено следующее условие: для любой окрестности $V\subset U$
найдутся континуумы $E$ и $F$ в $D,$ для которых
\begin{equation}\label{eq21}
E\cap \partial U\ne \varnothing\ne E\cap
\partial V,\quad F\cap
\partial U\ne \varnothing\ne F\cap \partial V
\end{equation}
и, одновременно,
\begin{equation}\label{eq22}
M(\Gamma(E, F, D))<P\,.
\end{equation}
Выберем в качестве $U=B(z_0, 1/2).$ Пусть $P:=\pi$ и пусть $V$ --
произвольная окрестность точки $z_0,$ содержащаяся в $U$ (см.
рисунок 2).
\begin{figure}[h]
\centerline{\includegraphics[scale=0.45]{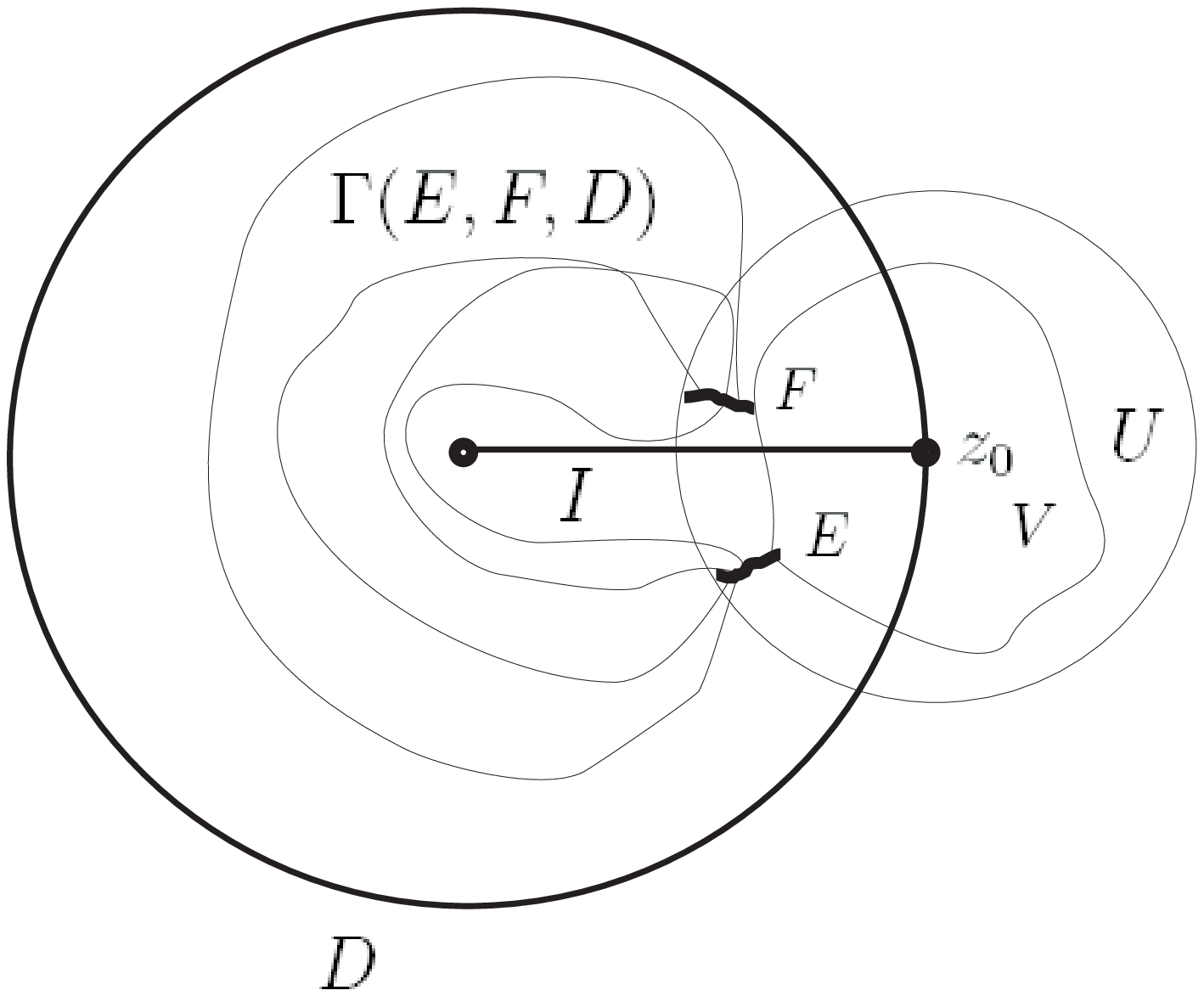}} \centerline{Рис.
2. Пример области, не имеющей слабо плоской границы}
\end{figure}
В качестве континуумов $F$ и $G$ выберем произвольным образом пару
кривых, соединяющих по разные стороны от разреза $I$ сферу $S(z_0,
1/2)$ с множеством $\partial V$ внутри шара $B(z_0, 1/2).$ По
построению $F$ и $G$ удовлетворяют соотношениям вида (\ref{eq21}).
Пусть $\gamma\in \Gamma(E, F, D),$ тогда геометрически очевидно, что
длина любой такой кривой $\gamma$ не меньше длины разреза $I,$ т.е.
не меньше 1. В таком случае, функция $\rho(z)=1$ при $z\in D$ и
$\rho(z)=0$ при $z\not\in D$ является допустимой для семейства
$\Gamma(E, F, D),$ поскольку
$$\int\limits_{\gamma}\rho(z)|dz|=\int\limits_{\gamma}|dz|=l(\gamma)\geqslant 1\,,$$
где $l(\gamma)$ обозначает длину кривой $\gamma.$ В таком случае, по
определению модуля семейств кривых как точной нижней грани
соответствующих интегралов, мы имеем:
$$M(\Gamma(E, F, D))\leqslant \int\limits_D1^2dm(z)=m(D)=\pi\,,$$
откуда вытекает соотношение (\ref{eq22}) при $P=\pi.$ Отсутствие
свойства слабой плоскости границы области $D$ в точке $z_0$
установлено.

\medskip
Следующее утверждение содержит расшифровку понятия локально
квазиконформной границы в терминах соотношений вида (\ref{eq1}) и
(\ref{eq3}). Его доказательство дословно повторяет доказательство
\cite[теорема~17.10]{Va}, и потому опускается.

\medskip
\begin{lemma}\label{lem2}
{\sl Пусть $D\subset {\Bbb M}^n$ -- область с локально
квазиконформной границей, тогда граница этой области является слабо
плоской и, в частности, является сильно достижимой. Более того,
окрестность $U$ в определении локально квазиконформной границы может
быть взята сколь угодно малой, при этом, в этом определении можно
считать $\varphi(x_0)=0.$}
\end{lemma}

\medskip
Справедлива также следующая лемма, обобщающая \cite[лемма~3.5]{Na}.

\medskip
\begin{lemma}\label{lem1}
{\sl Предположим, $D\subset {\Bbb M}^n$ -- область с локально
квазиконформной границей такая, что $\overline{D}$ является
компактом в ${\Bbb M}^n.$ Тогда тело $I(P)$ простого конца $P\subset
E_D$ состоит из одной точки $p\in \partial D$ и $d(\sigma_k,
\sigma_{k+1})>0$.}
\end{lemma}

\medskip
\begin{proof}
По предложению \ref{thabc1} имеем: $I(P)\subset \partial D.$
Покажем, что $I(P)$ -- одноточечное множество. Предположим
противное, то есть, существуют, по крайней мере, две точки $x, y\in
I(P).$ Тогда $d(x, y)=r>0.$ Пусть $D_m,$ $m=1,2,\ldots,$ --
последовательность областей в ${\Bbb M}^n,$ соответствующих простому
концу $P,$ тогда, согласно определению,
$I(P)=\bigcap\limits_{m=1}^{\infty}\overline{D_m}.$ В таком случае,
найдутся последовательности $x_m, y_m\in D_m$ такие, что
$x_m\rightarrow x$ и $y_m\rightarrow y$ при $m\rightarrow\infty.$ В
силу неравенства треугольника, $d(x_m, y_m)\geqslant r/2$ при
достаточно больших $m\geqslant m_0\in {\Bbb N}.$ Соединим точки
$x_m$ и $y_m$ кривыми $C_m$ в области $D_m.$ По построению
$d(C_m)\geqslant r/2$ при $m\geqslant m_0\in {\Bbb N}.$

Пусть $U_0$ -- произвольная окрестность точки $x,$ не содержащая
точки $y.$ По лемме \ref{lem2} область $D$ имеет слабо плоскую
границу, поэтому найдётся окрестность $V_0,$ такая, что для всяких
континуумов, $F$ и $G,$ пересекающих $\partial U_0$ и $\partial
V_0,$ выполняется условие
\begin{equation}\label{eq4}
M(\Gamma(E, F, D))\geqslant 1/2.
\end{equation}
Поскольку последовательность $x_m$ сходится к $x,$ то при всех
$m\geqslant m_1,$ $m_1\geqslant m_0,$ $m_1\in {\Bbb N},$ все точки
$x_m$ принадлежат окрестности $V_0.$ Таким образом, континуум $C_m$
пересекает $\partial U_0$ и $\partial V_0$ ввиду
\cite[теорема~1.I.5, $\S\,46$]{Ku}. Рассмотрим произвольную кривую
$C,$ соединяющую $\partial U_0\cap D$ и $\partial V_0\cap D.$ Тогда
ввиду (\ref{eq4}) мы будем иметь, что
$$M(\Gamma(C_m, C, D))\geqslant 1/2.$$
С другой стороны, очевидно, при больших $m\geqslant m_2,$ $m_2\in
{\Bbb N},$ выполнено соотношение $\Gamma(C_m, C, D)>\Gamma(\sigma_m,
C, D),$ откуда в силу минорирования модуля следует, что
$$M(\Gamma(\sigma_m, C, D))\geqslant M(\Gamma(C_m, C, D))\geqslant 1/2\,,$$
что противоречит соотношению (\ref{eqSIMPLE}). Полученное
противоречие указывает на неверность предположения о наличии не
менее двух точек во множестве $I(P).$

Осталось показать, что $d(\sigma_k, \sigma_{k+1})>0$. Предположим
противное, а именно, пусть при некотором $k\in {\Bbb N}$ выполнено
$d(\sigma_k, \sigma_{k+1})=0.$ Так как любое замкнутое подмножество
компакта -- компакт (см. \cite[теорема~2.II.4]{Ku}),
$\overline{\sigma_k}$ и $\overline{\sigma_{k+1}}$ --
непересекающиеся компактные подмножества $\overline{D}.$ Заметим,
что по определению
$$0=d(\sigma_k, \sigma_{k+1})=\inf\limits_{x\in \sigma_k, y\in \sigma_{k+1}}d(x, y)\,.$$
В силу определения точной нижней грани найдутся последовательности
$x_l\in \sigma_k, y_l\in\sigma_{k+1},$ такие, что $d(x_l,
y_l)\rightarrow d(\sigma_k, \sigma_{k+1})$ при $l\rightarrow\infty.$
Поскольку  $\overline{\sigma_k}$ и $\overline{\sigma_{k+1}}$ --
компакты, без ограничения общности рассуждений мы можем считать, что
обе последовательности $x_l$ и $y_l$ сходятся к точкам $x_0\in
\sigma_k$ и $y_0\in\sigma_{k+1},$ соответственно. Тогда в силу
неравенства треугольника
$$d(x_0, y_0)\leqslant d(x_0, x_l)+d(x_l, y_l)+d(y_l, y_0)\rightarrow 0,\quad l\rightarrow\infty\,,$$
откуда следует $x_0=y_0.$ Таким образом,
$\overline{\sigma_k}\cap\overline{\sigma_{k+1}}\ne \varnothing,$ то
есть, найдётся точка $p_0\in
\overline{\sigma_k}\cap\overline{\sigma_{k+1}}.$ Заметим, что
$p_0\in \partial D.$ Выберем произвольным образом окрестность $U$
точки $p_0,$ такую, что $\partial U\cap \sigma_k\ne\varnothing\ne
\sigma_{k+1}\cap \partial U.$ Ввиду леммы \ref{lem2} для каждого
$P>0$ существует окрестность $V\subset U$ этой же точки $p_0$ такая,
что $M(\Gamma(E, F, D))>P$ как только $E$ и $F$ пересекают $\partial
U$ и $\partial V.$ Заметим, что ввиду условия $p_0\in
\overline{\sigma_k}\cap\overline{\sigma_{k+1}}$ мы можем считать,
что условия $\partial V\cap \sigma_k\ne\varnothing\ne
\sigma_{k+1}\cap \partial V$ выполнены. Тогда $M(\Gamma(\sigma_k,
\sigma_{k+1}, D))=\infty$ ввиду произвольности $P>0.$ Последнее
противоречит свойству (\ref{eq5}), входящего в определение простого
конца. Полученное противоречие указывает на неверность предположения
$d(\sigma_k, \sigma_{k+1})=0.$ Лемма доказана.~$\Box$

\medskip
Следующее утверждение для пространства ${\Bbb R}^n$ и областей с
локально квазиконформными границами также доказано в
\cite[лемма~3.5]{Na}.

\medskip
\begin{lemma}\label{lem3}
{\sl Предположим, $D\subset {\Bbb M}^n$ -- область с локально
квазиконформной границей такая, что $\overline{D}$ является
компактом в ${\Bbb M}^n.$ Тогда для каждой точки $x_0\in\partial D$
найдётся простой конец $P,$ для которого $I(P)=\{x_0\}.$ }
\end{lemma}

\medskip
Пусть $x_0\in
\partial D$ и $\varphi$ -- квазиконформное отображение из
определения локально квазиконформной границы. Ввиду леммы \ref{lem2}
мы можем считать, что $\varphi(x_0)=0.$ Следуя началу доказательства
этой леммы, заключаем, что найдётся последовательность сфер $S(0,
1/2^k),$ $k=1,2,\ldots,$ и убывающая последовательность окрестностей
$V_k$ точки $x_0,$ для которых $\varphi(V_k)=B(0, 1/2^k),$
$\varphi(\partial V_k\cap D)=S(0, 1/2^k)\cap {\Bbb B}^n_+,$ где
$\varphi$ -- квазиконформное отображение, соответствующее
определению локально квазиконформной границы. Заметим, что
последовательность областей $V_k$ соответствует простому концу $P$ с
требуемыми свойствами, где $\sigma_k:=\partial V_k\cap D.$

Для доказательства этого заметим, прежде всего, что $\sigma_k,$
действительно, является разрезом, поскольку $V_k$ и $D\setminus
\overline{V_k},$ действительно, являются различными компонентами
связности $D\setminus \sigma_k,$ при этом, $\sigma_{k+1}\in V_k.$
Условия $\partial\sigma_k\cap D=\varnothing$ и
$\partial\sigma_k\cap\partial D\ne\varnothing,$ участвующие в
определении разреза, выполняются (как мы отметили ранее, $\sigma_k$
отождествляется с поверхностью $\sigma_k(w)=(\varphi^{\,-1}\circ
S_k)(w),$ где $S_k:\omega\rightarrow {\Bbb R}^n$ обозначает
некоторую параметризацию полусферы $S(0, 1/2^k),$ $\omega$ --
соответствующая этой параметризации область в пространстве ${\Bbb
R}^{n-1}$ и $w\in \omega$).

\medskip
Проверим теперь условия (i)--(ii) из определения цепи и требования
(\ref{eq5})--(\ref{eqSIMPLE}).  Как уже было отмечено выше,
$\sigma_{k+1}$ содержится в $V_k,$ кроме того, $\sigma_{k-1}\in
D\setminus V_k$ по построению. Наконец,
$\bigcap\limits_{k=1}^{\infty}V_k=\varnothing,$ поскольку, в
противном случае, мы имели бы точку $p_0\in
\bigcap\limits_{k=1}^{\infty}V_k,$ однако, тогда также
$\varphi(p_0)\in \bigcap\limits_{k=1}^{\infty}B_+(0, 1/2^k),$ что не
имеет места. Условие (ii), таким образом, также выполняется.

Осталось убедиться в выполнении условий
(\ref{eq5})--(\ref{eqSIMPLE}). Действительно, так как
$\overline{\sigma_k}$ и $\overline{\sigma_{k+1}}$ не пересекаются,
то $r:={\rm dist}\,(\sigma_k, \sigma_{k+1})>0.$ Тогда функция
$\rho(x),$ равная $1/r$ при $x\in D$ и $\rho(x)=0$ при $x\not\in D,$
допустима для семейства $\Gamma(\sigma_k, \sigma_{k+1}, D).$ Так как
$\overline{D}$ -- компакт, множество $D$ имеет  конечный объём
$v(D),$ поскольку $\overline{D}$ можно покрыть конечным числом
окрестностей конечного объёма. Значит,
$$M(\Gamma(\sigma_k, \sigma_{k+1}, D))\leqslant \int\limits_{{\Bbb M}^n}\frac{dv(x)}{r^n}
\leqslant \frac{v(D)}{r^n}<\infty\,.$$
Чтобы проверить условие (\ref{eqSIMPLE}), выберем произвольный
континуум $C\subset D.$ Заметим, что $C\subset D\setminus V_k$ при
некотором достаточно большом $k\in {\Bbb N}.$ Тогда
\begin{equation}\label{eq6}
\Gamma(C, \sigma_m, D)>\Gamma (\partial V_k\cap D, \sigma_m, V_k\cap
D)
\end{equation}
при всех $m>k.$ Кроме того, заметим, что
\begin{equation}\label{eq7}
\varphi(\Gamma (\partial V_k\cap D, \sigma_m, V_k\cap
D))=\Gamma(S_+(0, 1/2^k), S_+(0, 1/2^m), B_+(0, 1/2^k))
\end{equation}
и что согласно \cite[разд.~7.5]{Va}
$$M(\Gamma(S_+(0, 1/2^k), S_+(0,
1/2^m), B_+(0, 1/2^k))\leqslant$$
\begin{equation}\label{eq8}\leqslant M(\Gamma(S(0, 1/2^k), S(0, 1/2^m), B(0,
1/2^k)\setminus \overline{B(0, 1/2^m)}))=\end{equation}
$$=\frac{\omega_{n-1}}{\left(\log\frac{2^m}{2^k}\right)^{n-1}}\rightarrow
0,\quad m\rightarrow\infty\,.$$
Окончательно, из (\ref{eq6}), (\ref{eq7}) и (\ref{eq8}) ввиду
свойства минорирования модуля вытекает, что
$$M(\varphi(\Gamma(C, \sigma_m, D)))\leqslant \frac{\omega_{n-1}}{(\log\frac{2^m}{2^k})^{n-1}}\rightarrow
0,\quad m\rightarrow\infty\,.$$
Однако, так как $\varphi$ -- квазиконформное отображение, то из
последнего соотношения также вытекает, что $M(\Gamma(C, \sigma_m,
D))\rightarrow 0$ при $m\rightarrow\infty,$ что и завершает
доказательство леммы.~$\Box$
\end{proof}

\medskip
Следующее фундаментальное утверждение также доказано Някки в случае
${\Bbb R}^n$ (см. \cite[теорема~4.1]{Na}).

\medskip
\begin{theorem}\label{th1}
{\sl Пусть $D_0, D$ -- области с компактными замыканиями на
римановых многообразиях ${\Bbb M}^n$ и ${\Bbb M}^n_*,$
соответственно, и пусть $D_0$ -- область с локально квазиконформной
границей. Предположим, $f$ -- квазиконформное отображение области
$D_0$ на $D.$ Тогда существует взаимно однозначное соответствие
между точками границы области $D_0$ и простыми концами области $D.$}
\end{theorem}

\medskip
\begin{proof}
Прежде всего, установим, что между простыми концами областей $D$ и
$D_0$ имеется взаимно однозначное соответствие. Действительно, пусть
$P$ -- простой конец в $D$ и $\sigma_k,$ $k=1,2,\ldots ,$ --
соответствующая ему цепь разрезов. Заметим, прежде всего, что
последовательность $f(\sigma_k),$ $k=1,2,\ldots ,$ также образует
цепь разрезов. В самом деле, если $D\setminus\sigma_k$ состоит из
двух и более компонент, то $f(D)\setminus f(\sigma_k)$ также не
может быть связным множеством. Кроме того, если
$\partial\sigma_k\cap D=\varnothing$ и $\partial\sigma_k\cap\partial
D\ne\varnothing,$ то ввиду гомеоморфности отображения $f$ также и
$\partial f(\sigma_k)\cap f(D)=\varnothing$ и $\partial
f(\sigma_k)\cap\partial f(D)\ne\varnothing.$ Заметим также, что
выполнены условия (i)-(ii) из определения цепи разрезов: (i)
множество $f(\sigma_{m+1})$ содержится в точности в одной компоненте
$f(d_m)$ множества $f(D)\setminus f(\sigma_m)$, при этом,
$f(\sigma_{m-1})\subset f(D)\setminus (f(\sigma_m)\cup f(d_m))$;
(ii) $\cap\,f(d_m)=\varnothing$, где $f(d_m)$ -- компонента
$f(D)\setminus f(\sigma_m)$, содержащая $f(\sigma_{m+1})$. Наконец,
условия вида (\ref{eq5})--(\ref{eqSIMPLE}) выполнены для
последовательности $f(\sigma_m),$ $m=1,2,\ldots ,$ ввиду
квазиконформности $f.$ Таким образом, отображение $f$ может быть
продолжено до отображения $f:\overline{D}_P\rightarrow
\overline{D_0}_P,$ которое сюрьективно и инъективно.

\medskip
Таким образом, для доказательства утверждения теоремы \ref{th1}
достаточно установить взаимно однозначное соответствие между
$E_{D_0}$ и $\partial D_0.$ Пусть далее $h$ -- тождественное
отображение области $D_0$ на $D_0.$ Будем следовать схеме
доказательства \cite[теорема~4.1]{Na}. Пусть $P\in E_{D_0},$ тогда
положим
%
$$h(P)=I(P)\,.$$
%
Ввиду леммы \ref{lem1} множество $I(P)$ состоит из единственной
граничной точки $b\in \partial D_0$, а по лемме \ref{lem3} указанное
соответствие является сюрьективным отображением $E_{D_0}$ на
$\partial D_0.$ Покажем, что $h$ является также и инъективным
отображением на множестве $E_{D_0}.$ Предположим противное, а
именно, что найдётся точка $b\in\partial D_0$ и два различных
простых конца $P_1\ne P_2,$ $P_1, P_2\in E_{D_0},$ такие, что
$I(P_1)=I(P_2)=b.$ Предположим, $D_i$ -- последовательность
областей, соответствующая простому концу $P_1.$ Согласно определению
\begin{equation}\label{eq9}
\bigcap\limits_{i=1}^{\infty}\overline{D_i}=b\,.
\end{equation}
Пусть $G_i,$ $i=1,2,\ldots,$ -- последовательность областей,
соответствующая простому концу $P_2,$ тогда также
\begin{equation}\label{eq10}
\bigcap\limits_{i=1}^{\infty}\overline{G_i}=b\,.
\end{equation}
Так как по предположению $P_1\ne P_2,$ то соответствующие им цепи
разрезов не эквивалентны, т.е., либо область $D_i$ (при некотором
$i\in {\Bbb N}$) не содержит все области $G_k,$ кроме конечного
числа, либо область $G_m$ (при некотором $m\in {\Bbb N}$) не
содержит все области $G_s,$ кроме конечного числа. Другими словами,
выполнено одно из двух: 1) либо найдутся $i\in {\Bbb N},$
возрастающая последовательность элементов $k_l,$ $l=1,2, \ldots,$ и
элементы $a_{k_l}\in G_{k_l}$ такие, что $a_{k_l}\in D_0\setminus
D_i;$ 2) либо найдутся $m\in {\Bbb N},$ возрастающая
последовательность элементов $r_l,$ $l=1,2, \ldots,$ и элементы
$c_{r_l}\in D_{r_l}$ такие, что $c_{r_l}\in D_0\setminus G_m.$ Так
как $\overline{D_0}$ -- компакт, то мы можем считать, что и в
первом, и во втором случае последовательности $b_{k_l}$ либо
$c_{r_l}$ являются сходящимися, причём ввиду (\ref{eq9}) и
(\ref{eq10}) они могут сходиться только к точке $b.$

\medskip
В любом из этих двух случаев мы имеем последовательность элементов
$b_l\in D,$ $l=1,2,\ldots,$ сходящуюся при $l\rightarrow\infty$ к
$b$ и лежащую в $D_0\setminus D_i$ (либо в $D_0\setminus G_m$) при
всех $l\in {\Bbb N}.$ Пусть для определённости указанная
последовательность $b_l$ лежит в $D_0\setminus D_i$ при всех $l=1,2,
\ldots ,$ и пусть $\sigma_i$ -- цепь разрезов, соответствующих
последовательности областей $D_i.$

\medskip
Докажем, что при сделанных предположениях $b\in \overline{\sigma_k}$
при всех $k\geqslant i.$ Если $b\in \overline{\sigma_k}$ хотя бы при
одном $k\geqslant i,$ то найдётся окрестность $U$ точки $b,$ такая
что $U\cap \sigma_k=\varnothing,$ при этом, для некоторого
квазиконформного отображения $\varphi:U\rightarrow {\Bbb R}^n$
выполнялись бы условия $\varphi(U)={\Bbb B}^n$ и $\varphi(U\cap
D_0)={\Bbb B}^n_+.$ Таким образом, множество $U\cap D_0$ является
связным и, значит, оно принадлежит только одной из связных компонент
$D_0\setminus \sigma_k,$ а именно, либо $U\cap D_0\subset D_k,$ либо
$U\cap D_0\subset D_0\setminus \overline{D_k}.$ Так как
последовательность $b_l$ сходится при $l\rightarrow\infty$ к точке
$b,$ то $b_l\in U\cap D_0$ при больших $l\geqslant l_0,$ поэтому
случай $U\cap D_0\subset D_k$ невозможен, поскольку по предположению
$b_l$ лежит в $D_0\setminus D_i$ при всех $l=1,2, \ldots .$ В таком
случае, $U\cap D_0\subset D_0\setminus \overline{D_k},$ что также не
может иметь места, так как ввиду соотношения (\ref{eq9}) мы можем
найти последовательность элементов $a_m\in D_m,$ $m=1,2,\ldots,$
сходящуюся к $b$ при $m\rightarrow\infty,$ т.е., $U\cap D_0\cap
D_m\ne \varnothing$ при больших $m\geqslant k$ и, в частности,
$U\cap D_0\cap D_k\ne \varnothing.$ Полученное противоречие говорит
о том, что  $b\in \overline{\sigma_k}$ при всех $k\geqslant i.$
Тогда $d(\sigma_k, \sigma_{k+1})=0,$ что противоречит утверждению
леммы \ref{lem1}. Указанное противоречие говорит о том, что исходное
предположение о наличии различных простых концов $P_1\ne P_2,$ $P_1,
P_2\in E_{D_0},$ таких что $I(P_1)=I(P_2)=b,$ было неверным. Теорема
доказана.~$\Box$
\end{proof}

\medskip
\begin{corollary}\label{cor1}
{\sl Пусть $D_0, D$ -- области с компактными замыканиями на
римановых многообразиях ${\Bbb M}^n$ и ${\Bbb M}^n_*,$
соответственно, имеющие локально квазиконформную границу.
Предположим, $f:D_0\rightarrow D$ -- квазиконформное отображение
области $D_0$ на $D.$ Тогда $f$ продолжается до гомеоморфизма
$\overline{D_0}$ на $\overline{D}.$}
\end{corollary}

\medskip
\begin{proof}
Пусть $x_m\in D_0,$ $x_m\stackrel{d}{\rightarrow} x_0\in
\overline{D_0}$ при $m\rightarrow\infty$ -- произвольная
последовательность. Нужно показать, что существует
$\lim\limits_{m\rightarrow\infty}f(x_m)$ в метрике $d_*.$ Если $x_0$
-- внутренняя точка $D_0,$ доказывать нечего. Пусть $x_0\in \partial
D_0.$ По теореме \ref{th1} найдётся единственный простой конец
$P_0\in E_{D_0}$ такой, что $x_0=I(P_0).$ По этой же теореме
простому концу $P$ соответствует единственный простой конец области
$D,$ а именно простой конец $f(P_0),$ более того, найдётся точка
$y_0\in D$ такая, что $y_0=I(f(P_0)).$ Пусть $\varphi$ --
квазиконформное отображение из определения локально квазиконформной
границы, соответствующее точке $x_0.$  Как уже было установлено при
доказательстве леммы \ref{lem2}, мы можем считать, что
$\varphi(x_0)=0.$ Следуя началу доказательства этой леммы,
заключаем, что найдётся последовательность сфер $S(0, 1/2^k),$
$k=1,2,\ldots,$ и убывающая последовательность окрестностей $V_k$
точки $x_0,$ для которых $\varphi(V_k)=B(0, 1/2^k),$
$\varphi(\partial V_k\cap D)=S(0, 1/2^k)\cap {\Bbb B}^n_+,$ где
$\varphi$ -- квазиконформное отображение, соответствующее
определению локально квазиконформной границы. Заметим, что
последовательность областей $V_k$ соответствует простому концу $P$ с
требуемыми свойствами, где $\sigma_k:=\partial V_k\cap D$ (этот факт
был установлен при доказательстве леммы \ref{lem3}).

Отсюда следует, что $x_m\in V_k $ при каждом фиксированном $k\in
{\Bbb N}$ и всех $m\geqslant m_0(k),$ где $m_0\in {\Bbb N}.$

Выберем произвольно $\varepsilon>0.$ Так как $y_0=I(f(P_0)),$
найдётся номер $k_0(\varepsilon)\in {\Bbb N}:$ $f(V_k)\subset B(y_0,
\varepsilon)$ при всех $k\geqslant k_0.$ Положим
$M(\varepsilon):=m_0(k_0(\varepsilon)).$ Тогда при $m\geqslant
M(\varepsilon)$ имеем $|f(x_m)-y_0|<\varepsilon,$ поскольку $x_m\in
V_{k_0},$ а $f(V_{k_0})\in B(y_0, \varepsilon).$ Отсюда следует, что
$f(x_m)\stackrel{d_*}{\rightarrow} y_0,$ что и доказывает
непрерывность отображения $f:\overline{D_0}\rightarrow
\overline{D}.$

Осталось установить, что $f(\overline{D_0})=\overline{D}.$ Очевидно,
$f(\overline{D_0})\subset\overline{D}.$ Покажем обратное включение.
Пусть $y_0\in \overline{D}.$ Если $y_0\in D,$ то, очевидно, $y_0\in
f(D_0).$ Пусть теперь $y_0\in\partial D.$ По теореме \ref{th1}
найдутся единственные простые концы $P_0\in E_{D_0}$ и $f(P_0)\in
E_D$ такие, что $y_0=I(f(P_0))$ и, кроме того, найдётся $x_0\in
\partial D_0$ такая, что $x_0=I(P_0).$ Следовательно, найдётся также
последовательность $x_k\in D_0,$ такая что
$x_k\stackrel{d}{\rightarrow}x_0.$ По доказанному выше $f(x_0)=y_0.$
Следствие доказано.
\end{proof}~$\Box$

\medskip
\begin{remark}\label{rem1}
Обозначим $\overline{D_0}_P:=D_0\cup E_{D_0},$ где $E_{D_0}$ --
множество всех простых концов области $D_0.$ Пусть $D_0, D$ --
области с компактными замыканиями на римановых многообразиях ${\Bbb
M}^n$ и ${\Bbb M}^n_*,$ соответственно, и пусть $D_0$ -- область с
локально квазиконформной границей. Руководствуясь теоремой
\ref{th1}, положим
\begin{equation}\label{eq18}
h(x)=\left\{\begin{array}{rr} x,& x\in D_0,\\ I(x), & x\in
E_{D_0}\,,
\end{array} \right.
\end{equation}
где, как и прежде, $I(x)$ обозначает тело простого конца $x\in
E_{D_0}.$ Ввиду теоремы \ref{th1} отображение $h$ взаимнооднозначно
отображает $\overline{D_0}_P$ на $\overline{D_0};$ в частности, $h$
взаимнооднозначно отображает $E_{D_0}$ на $\partial D_0.$

Если $\overline{D}_P$ является пополнением регулярной области $D$ её
простыми концами и $g_0$ является квазиконформным отображением
области $D_0$ с локально квазиконформной границей на $D$, то оно
естественным образом определяет в $\overline{D}_P$ метрику
\begin{equation}\label{eq16} \rho_0(p_1,p_2)=d\left( h(g_0^{-1}(p_1)),
h(g_0^{-1}(p_2))\right)\,.
\end{equation}

Если $g_*$ является другим квазиконформным отображением некоторой
области $D_*$ с локально квазиконформной границей на область $D$, то
соответствующая метрика
\begin{equation}\label{eq13}
\rho_*(p_1, p_2)=d\left(h(g_*^{\,-1}(p_1)),
h(g_*^{\,-1}(p_2))\right)
\end{equation}
порождает ту же самую сходимость и, следовательно, ту же самую
топологию в $\overline{D}_P$ как и метрика $\rho_0$, поскольку
$g^{\,-1}_0\circ g_*$ является квазиконформным отображением между
областями $D_*$ и $D_0$, которое по теореме \ref{th1} продолжается
до гомеоморфизма между $\overline{D_*}$ и $\overline{D_0}$.

\medskip
В дальнейшем, будем называть данную топологию в пространстве
$\overline{D}_P$ {\it топологией простых концов} и понимать
непрерывность отображений
$F:\overline{D}_P\rightarrow\overline{D^{\,\prime}}_P$ как раз
относительно этой топологии.
\end{remark}

\medskip
\begin{remark}\label{rem2}
Пусть $D_0, D$ -- области с компактными замыканиями на римановых
многообразиях ${\Bbb M}^n$ и ${\Bbb M}^n_*,$ соответственно, и пусть
$D_0$ -- область с локально квазиконформной границей. Заметим, что
метрическое пространство $(\overline{D}_P, \rho_0)$ компактно. В
самом деле, пусть у нас есть последовательность элементов $x_k\in
\overline{D}_P,$ $k=1,2, \ldots,$ и $g_0$ является квазиконформным
отображением области $D_0$ с локально квазиконформной границей на
$D,$ которому соответствует метрика $\rho_0,$ определённая
соотношением (\ref{eq16}). Тогда $z_k:=h(g_0^{\,-1}(x_k))$ --
последовательность элементов в $\overline{D_0},$ где $h$ определено
соотношением (\ref{eq18}). Так как $\overline{D_0}$ предполагалось
компактным множеством, то из последовательности $z_k$ можно извлечь
сходящуюся подпоследовательность $z_{k_l},$ $l=1,2,\ldots,$ к
некоторой точке $z_0\in \overline{D_0}.$ Точке $z_0$ соответствует
некоторый простой конец $P_0\in E_{D_0}$ (точка $P_0\in D_0$),
которому, в свою очередь, соответствует простой конец $f(P_0)\in
E_D$ (точка $f(P_0)\in D$).
\end{remark}

\medskip
Из теоремы \ref{th1} с учётом замечания \ref{rem1} вытекает
следующее утверждение, обобщающее классический результат Някки для
пространства ${\Bbb R}^n$ (см. \cite[теорема~4.2]{Na}).

\medskip
\begin{theorem}\label{th2}
{\sl Пусть $D_0, D$ -- области с компактными замыканиями на
римановых многообразиях ${\Bbb M}^n$ и ${\Bbb M}^n_*,$
соответственно, и пусть $D_0$ -- область с локально квазиконформной
границей. Предположим, $f$ -- квазиконформное отображение области
$D_0$ на $D.$ Тогда $f$ продолжается до гомеоморфизма
$f:\overline{D_0}\rightarrow \overline{D}_P.$ }
\end{theorem}

\medskip
{\bf 3. Основная лемма о регулярных концах.} В настоящем разделе
рассматриваются области, содержащие регулярные цепи разрезов.
Следующее утверждение является аналогом \cite[лемма~1]{KR} на
римановых многообразиях.

\medskip
\begin{lemma}\label{thabc2} {\sl Каждый регулярный конец $K$
области $D\subset{\Bbb M}^n,$ имеющей компактное замыкание
$\overline{D}\subset{\Bbb M}^n$ содержит в себе цепь разрезов
$\sigma_m$, лежащих на сферах $S_m$ с центром в некоторой точке
$x_0\in\partial D$ и геодезическими радиусами $\rho_m\rightarrow 0$
при $m\rightarrow\infty.$ }\end{lemma}

\begin{proof} Пусть $\{\sigma_m\}$ -- цепь разрезов в конце $K$ и $x_m$ --
последовательность точек в $\sigma_m.$ Без ограничения общности
можем считать, что $x_m\rightarrow x_0\in\partial D$ при
$m\rightarrow\infty$, поскольку $\overline{D}$ -- компакт.
Положим
$$\rho^-_m:=d(x_0,\sigma_m)\,.$$
По неравенству треугольника $d(x_0, \sigma_m)\leqslant d(x_0,
x_m)+d(x_m, y)\leqslant d(x_0, x_m)+d(\sigma_m).$ Поскольку
$d(\sigma_m)\rightarrow 0$ при $m\rightarrow\infty,$ отсюда следует,
что
$$\rho^-_m\rightarrow 0,\quad m\rightarrow\infty\,.$$
Кроме того,
$$\rho^+_m:=H(x_0,\sigma_m)=\sup\limits_{x\in\sigma_m}d(x,x_0)=
\sup\limits_{x\in\overline{\sigma_m}}d(x,x_0)$$
-- хаусдорфово расстояние между компактными множествами $\{x_0\}$ и
$\overline{\sigma_m}$ в $\overline{D}.$ В силу всё того же
неравенства треугольника $d(x_0, x)\leqslant d(x_0, x_m)+d(x_m,
x)\leqslant d(x_0, x_m)+d(\sigma_m).$ Отсюда следует, что
$$\rho^+_m\rightarrow 0,\quad m\rightarrow\infty\,.$$
Согласно условию (i) в определении конца можем считать, без
ограничения общности, что $\rho^-_m>0.$ Кроме того, переходя, если
это нужно к подпоследовательности, мы можем считать, что
$\rho^+_{m+1}<\rho^-_m$ для всех $m=1,2,\ldots$.

Положим $$\delta_m=\Delta_m\setminus d_{m+1},$$ где
$\Delta_m=S_m\cap d_m$ и
$$S_m=\left\{x\in {\Bbb M}^n: d(x_0,x)=\frac{1}{2}\left(\rho^-_m+\rho^+_{m+1}\right)\right\}\,.$$
Очевидно, что $\Delta_m$ и $\delta_m$ относительно замкнуты в $d_m.$

Заметим, что $d_{m+1}$ содержится в одной из компонент связности
открытого множества $d_m\setminus\delta_m$. Действительно,
предположим, что пара точек $x_1$ и $x_2\in d_{m+1}$ находится в
различных компонентах $\Omega_1$ и $\Omega_2$ множества
$d_m\setminus\delta_m$. Поскольку на римановых многообразиях
открытые связные множества являются также и линейно связными (см.
\cite[предложение~13.1]{MRSY}), точки $x_1$ и $x_2$ могут быть
соединены кривой $\gamma:[0,1]\rightarrow d_{m+1}$. Однако, по
построению, $d_{m+1}$, а поэтому и $\gamma$, не пересекают
$\delta_m$, следовательно,
$[0,1]=\bigcup\limits_{k=1}^{\infty}\omega_k$, где
$\omega_k=\gamma^{-1}(\Omega_k)$, $\Omega_k$ -- перенумерация
компонент $d_m\setminus\delta_m$ (поскольку многообразие ${\Bbb
M}^n$ локально связно, все компоненты $\Omega_k$ множества
$d_m\setminus\delta_m$ открыты и их не более, чем счётно, см.
\cite[теоремы~4 и 6, разд.~6.49.II]{Ku}). Но $\omega_k$ является
открытым в $[0,1]$, поскольку $\Omega_k$ открыто и $\gamma$
непрерывна. Последнее противоречит связности $[0,1]$, так как
$\omega_1\ne\varnothing$ и $\omega_2\ne\varnothing$ и, кроме того,
$\omega_i$ и $\omega_j$ попарно не пересекаются при $i\ne j$.

Пусть $d^*_m$ компонента $d_m\setminus\delta_m$, содержащая
$d_{m+1}$. Тогда по построению $d_{m+1}\subset d^*_m\subset d_m$.
Покажем, что $\partial d^*_m\setminus\partial D\subset\delta_m$.
Во-первых, очевидно, что $\partial d^*_m\setminus\partial
D\subset\delta_m\cup\sigma_m.$ (Действительно, если бы нашлась точка
$x\in (\partial d^*_m\setminus\partial D)\setminus
(\delta_m\cup\sigma_m),$ то ввиду включений $d^*_m\subset d_m,$ и
$\partial d_m\cap D\subset \sigma_m,$ мы имели бы $x\in
(\overline{d_m}\setminus\partial D)\setminus
(\delta_m\cup\sigma_m)\subset (\sigma_m\cup d_m)\setminus
(\delta_m\cup\sigma_m)=d_m\setminus \delta_m.$ С другой стороны,
всякая точка в $d_m\setminus\delta_m$ принадлежит либо $d^*_m$, либо
другой компоненте $d_m\setminus\delta_m$, и поэтому не принадлежит
границе $d^*_m$, ввиду относительной замкнутости $\delta_m$ в $d_m.$
Полученное противоречие указывает на справедливость включения
$\partial d^*_m\setminus\partial D\subset\delta_m\cup\sigma_m$).
Таким образом, достаточно доказать, что $\sigma_m\cap\partial
d^*_m\setminus\partial D=\varnothing$.

Предположим, что существует точка $x_*\in\sigma_m$ в $\partial
d^*_m\setminus\partial D$. Покажем, что найдется точка $y_*\in
d^*_m$, достаточно близкая к $\sigma_m$, такая что
\begin{equation}\label{eq9A}
d(x_0,y_*)>\frac{1}{2}\left(\rho^-_m+\rho^+_{m+1}\right)\,.\end{equation}
В самом деле, по определению точной нижней грани $\rho^-_m\leqslant
d(x_0, x_*).$ Поскольку согласно сделанному выше предположению
$\partial d^*_m\setminus\partial D,$ найдётся последовательность
$x_k\in d^*_m\setminus\partial D,$ $k=1,2,\ldots ,$ такая что
$d(x_k, x_*)<1/k.$ По неравенству треугольника $d(x_0,
x_*)<1/k+d(x_k, x_*).$ Так как неравенство $\rho^+_{m+1}<\rho^-_m$
-- строгое, то из последнего неравенства при некотором достаточно
большом $k\in {\Bbb N}$ имеем
$$d(x_0, x_k)>d(x_0, x_*)-1/k\geqslant \rho^-_m-1/k>\frac{1}{2}\left(\rho^-_m+\rho^+_{m+1}\right)\,,$$
что совпадает с неравенством (\ref{eq9A}) при $y_*=x_k.$

На основании аналогичных рассуждений, найдется точка $z_*\in
d_{m+1}$, достаточно близкая к $\sigma_{m+1}$, такая, что
$$d(x_0,z_*)<\frac{1}{2}\left(\rho^-_m+\rho^+_{m+1}\right)\,.$$
Кроме того, точки $z_*$ и $y_*$  могут быть соединены непрерывной
кривой $\gamma:[0,1]\rightarrow d^*_{m+1}$. Заметим, что множества
$\gamma^{-1}(d^*_m\setminus\overline{d_{m+1}})$ состоят из счетного
набора открытых непересекающихся интервалов из $[0,1]$ и интервала
$(t_0,1]$ с $t_0\in(0,1)$, и $z_0=\gamma(t_0)\in\sigma_{m+1}$. Таким
образом,
\begin{equation}\label{eq11}
d(x_0,z_0)<\frac{1}{2}\left(\rho^-_m+\rho^+_{m+1}\right)\,,
\end{equation}
поскольку $d(x_0,z_0)\leqslant\rho^+_{m+1}$ и
$\rho^+_{m+1}<\rho^-_m$. Из (\ref{eq9A}) и (\ref{eq11}), в силу
непрерывности функции $\varphi(t)=h(x_0,\gamma(t)),$ вытекает
существование точки $\tau_0\in(t_0,1)$ такой, что
$$d(x_0,y_0)=\frac{1}{2}\left(\rho^-_m+\rho^+_{m+1}\right)\,,$$
где $y_0=\gamma(\tau_0)\in d^*_m$ в силу выбора $\gamma$. Полученное
противоречие показывает, что наше предположение неверно, так что
$\partial d^*_m\setminus\partial D\subset\delta_m$.

\medskip
В наших рассуждениях в качестве цепи разрезов следует взять
множества $\delta_m,$ а в качестве последовательности
соответствующих областей -- последовательность $d^*_m,$
$m=1,2,\ldots .$ Остаётся показать, что данные множества $\delta_m$
действительно образуют цепь разрезов в смысле свойств (i)--(ii),
приведенных в первой части работы.

\medskip
Заметим, прежде всего, что множества $\delta_m$ удовлетворяют
определению разреза, а именно, проверим следующие условия: 1)
множество $D\setminus \delta_m$ имеет больше одной компоненты, 2)
$\partial\delta_m\cap D=\varnothing$ и 3)
$\partial\delta_m\cap\partial D\ne\varnothing$. В самом деле, 1)
область $d_m^{\,*}$ является одной из компонент $D\setminus
\delta_m$ ввиду определения $d_m^{\,*},$ кроме того, если бы
$D\setminus \delta_m$ состояло бы из одной компоненты связности, то
любые две точки $x_1, x_2\in D\setminus \delta_m$ можно было бы
связать кривой $\gamma$ в $D\setminus \delta_m$ (так как открытое
связное множество на римановом многообразии является линейно связным
ввиду \cite[следствие~13.1]{MRSY}). Выберем $x_1\in d_m^{\,*},$
$x_2\in D\setminus d_m.$ Заметим, что $x_1$ и $x_2$ лежат в
$D\setminus \delta_m$ по построению. Поскольку $d_m^{\,*}\subset
d_m,$ то кривая $\gamma,$ соединяющая точки $x_1$ и $x_2,$ не лежит
целиком ни в $d_m^{*},$ ни в $D\setminus d_m^{*},$ поэтому эта
кривая ввиду \cite[теорема~1.I.5, $\S\,46$]{Ku}) пересекает
$\partial d_m^{\,*}\cap D\subset \delta_m,$ что противоречит
сделанному предположению. Значит, $D\setminus \delta_m$ имеет более
одной компоненты.

Осталось установить условия 2) и 3). Для этого установим сначала
соотношение
\begin{equation}\label{eq12}
\overline{\delta_m}\cap\partial D\ne\varnothing\,.
\end{equation}
 Заметим, что сфера
$S_m=\frac{1}{2}\left(\rho^-_m+\rho^+_{m+1}\right)$ при достаточно
больших $m$ лежит в нормальной окрестности точки $x_0.$ Таким
образом, $S_m$ является связным множеством на многообразии ${\Bbb
M}^n,$ так как в локальных координатах множество $S_m$ представляет
собой евклидову сферу (см. \cite[лемма~5.10 и следствие~6.11]{Lee}).
Тогда $(S_m\cap\overline{\delta_m})\cap \overline{S_m\setminus
\delta_m}\ne\varnothing$ ввиду связности $S_m,$
$\delta_m=\Delta_m\setminus d_{m+1}$ и $\Delta_m=S_m\cap d_m$ (см.
\cite[определение~5.I.46]{Ku}). Пусть $\zeta_0\in
(S_m\cap\overline{\delta_m})\cap \overline{S_m\setminus \delta_m},$
тогда, в частности,
\begin{equation}\label{eq2}
\zeta_0\in S_m\cap\overline{\delta_m}=S_m\cap\overline{S_m\cap
d_m\setminus d_{m+1}}\subset\overline{S_m\cap
d_m}\subset\overline{d_m}\,.
\end{equation}
Так как $\zeta_0\in \overline{S_m\setminus
\delta_m}=\overline{S_m\setminus ((d_m\cap S_m)\setminus d_{m+1})},$
то найдётся последовательность $\zeta_k\in S_m\setminus ((d_m\cap
S_m)\setminus d_{m+1})$ такая, что
$\zeta_0=\lim\limits_{k\rightarrow\infty}\zeta_k.$ Возможны две
ситуации: 1) когда бесконечное число элементов последовательности
$\zeta_k$ принадлежат множеству $S_m\setminus d_m;$ 2) данному
множеству принадлежат только конечное число элементов данной
последовательности.

В ситуации 1) мы имеем $\zeta_0\in \overline{S_m\setminus d_m},$ но
в силу (\ref{eq2}) мы также имеем, что $\zeta_0\in \overline{d_m}.$
Тогда $\zeta_0\in \partial d_m,$ что ввиду соотношения $(\partial
d_m\setminus \partial D)\cap S_m=\sigma_m\cap S_m=\varnothing$
(выполненного по построению сферы $S_m$) может быть возможно лишь в
ситуации $\zeta_0\in
\partial D.$ В ситуации 2) имеем $\zeta_0\in\overline{d_{m+1}}.$
Снова ввиду соотношения (\ref{eq2}) имеем
$\zeta_0\in\overline{S_m\setminus d_{m+1}},$ откуда вытекает, что
$\zeta_0\in\partial d_{m+1}.$ Так как по построению $(\partial
d_{m+1}\setminus
\partial D)\cap S_m=\sigma_{m+1}\cap S_m=\varnothing,$ то последнее
снова возможно лишь в случае $\zeta_0\in \partial D.$ Итак, в обеих
ситуациях 1) и 2) мы имеем точку $\zeta_0\in \partial D,$ причём
ввиду (\ref{eq2}) выполнено $\zeta_0\in\overline{\delta_m},$ что и
указывает на справедливость соотношения (\ref{eq12}).

\medskip
Покажем теперь справедливость условия 2) $\partial\delta_m\cap
D=\varnothing.$ В самом деле, если бы нашлась точка $\xi_0\in
\partial \delta_m=\partial (\Delta_m\setminus d_{m+1})=\overline{\Delta_m\setminus d_{m+1}}
\setminus (\Delta_m\setminus d_{m+1}),$ то это означало бы, что
нашлась бы последовательность $\xi_k,$ $k=1,2,\ldots,$ такая что
$\xi_k\in S_m\cap d_m\setminus d_{m+1}$ и $\xi_k\rightarrow \xi_0$
при $k\rightarrow\infty.$ В этом случае также либо $\xi_0\not\in
d_m,$ либо $\xi_0\in d_{m+1}.$ Тогда, соответственно, либо $\xi_0\in
\partial d_m\cap S_m,$ либо $\xi_0\in
\partial d_{m+1}\cap S_m.$ Так как по построению $(\partial
d_m\setminus \partial D)\cap S_m=\sigma_m\cap S_m=\varnothing$ и
$(\partial d_{m+1}\setminus \partial D)\cap S_m=\sigma_{m+1}\cap
S_m=\varnothing,$ каждый из этих двух случаев возможен лишь при
$\xi_0\in \partial D.$ Условие 2) $\partial\delta_m\cap
D=\varnothing$ установлено. Наконец, условие 3)
$\partial\delta_m\cap\partial D\ne\varnothing$ является следствием
условия 2) и соотношения (\ref{eq12}).

\medskip
Наконец, проверим условия цепи разрезов (i) множество $\delta_{m+1}$
содержится в точности в одной компоненте $d^{\,*}_m$ множества
$D\setminus \delta_m$, при этом, $\delta_{m-1}\subset D\setminus
(\delta_m\cup d^{\,*}_m)$; (ii) $\cap\,d^{\,\*}_m=\varnothing$, где
$d^{\,*}_m$ -- компонента $D\setminus \delta_m$, содержащая
$\delta_{m+1}$.

\medskip
Действительно, $\delta_{m+1}\subset d_{m+1}\subset d_m^{\,*}$ по
построению, причём $d_m^{\,*}$ -- некоторая компонента связности
множества $D\setminus \delta_m.$ Пусть, кроме того, $x\in
\delta_{m-1},$ тогда $x\not\in d_m^*,$ поскольку по построению
$d_m^*\subset d_{m-1}^*$ и $\partial d_{m-1}^*\subset \delta_{m-1}.$
В силу сказанного, $\delta_{m-1}\subset D\setminus (\delta_m\cup
d_m^*),$ т.е., выполнено условие (i). Наконец, пусть $y\in
\cap\,d^{\,\*}_m.$ Тогда также $y\in \cap\,d_m$ ввиду свойства
$d_{m+1}\subset d^{\,*}_m\subset d_m,$ $m=1,2,\ldots .$ Но последнее
невозможно, так как исходная последовательность областей $d_m$
образовывала пустое пересечение. Полученное противоречие указывает
на выполнение условия (ii). Лемма полностью доказана.~$\Box$
\end{proof}

\medskip
В дальнейшем, для заданной области $D$ в ${\Bbb M}^n$,
$n\geqslant2$, говорим, что последовательность точек $x_k\in D$,
$k=1,2,\ldots$, {\it сходится к концу} $K$, если для каждой цепи
$\{\sigma_m\}$ в $K$ и каждой области $d_m$ все точки $x_k$, за
исключением, быть может, конечного числа, принадлежат $d_m$. В этом
случае, мы пишем: $x_k\stackrel{\rho}{\rightarrow}P$ при
$k\rightarrow\infty,$ или даже $x_k\rightarrow P,$ если
недоразумение невозможно. Из определения метрики в пространстве
простых концов вытекает, что сходимость в указанном выше смысле
эквивалентна сходимости в пространстве $\overline{D}_P$ в смысле
соотношения (\ref{eq13}).

\medskip
{\bf 4. Граничное продолжение классов Олича--Соболева в терминах
простых концов.} Определение классов Орлича--Соболева $W_{loc}^{1,
\varphi}$ на римановых многообразиях, встречающееся ниже, может быть
найдено, напр., в работе \cite{IS}. Исследование этих классов в
случае гомеоморфизмов проведено, преимущественно, в статьях
\cite{ARS} и \cite{Af$_2$}. В случае отображений с ветвлением
некоторые важнейшие вопросы (локального и граничного поведения,
продолжения в изолированную точку границы и проч.) исследованы в
наших предыдущих статьях \cite{IS}--\cite{IS$_1$}. Ниже основной
акцент делается на граничном продолжении в терминах простых концов.

Напомним некоторые необходимые нам сведения. Прежде всего, определим
якобиан отображения $f$ в точке $x\in D$ как
$$J(x, f)=\limsup\limits_{r\rightarrow 0}\frac{v_*(f(B(x, r)))}{v(B(x, r))}\,,$$
где $v$ и $v_*$ --- объём в ${\Bbb M}^n$ и ${\Bbb M}^n_*,$
соответственно. Полагаем
$$L(x, f)=\limsup\limits_{y\rightarrow x}\frac{d_*(f(x), f(y))}{d(x,
y)}\,,\qquad l(x, f)=\liminf\limits_{y\rightarrow x}\frac{d_*(f(x),
f(y))}{d(x, y)}\,,$$
где $d$ и $d_*$ --- геодезические расстояния на ${\Bbb M}^n$ и
${\Bbb M}^n_*,$ соответственно. Для отображений с конечным
искажением и произвольного $p\geqslant 1$ корректно определена и
почти всюду конечна так называемая {\it внутренняя дилатация $K_{I,
p}(x,f)$ отображения $f$ порядка $p$ в точке $x$}, определяемая
равенствами
 \begin{equation}\label{eq0.1.1A}
K_{I, p}(x,f)\quad =\quad\left\{
\begin{array}{rr}
\frac{J(x,f)}{l^p(x, f)}, & J(x,f)\ne 0,\\
1,  &  f^{\,\prime}(x)=0, \\
\infty, & \text{в\,\,остальных\,\,случаях}.
\end{array}
\right.
 \end{equation}

\medskip
Пусть $X$ и $Y$ -- произвольные метрические пространства.
Отображение $f:X\rightarrow Y$ будем называть {\it дискретным,} если
для каждого $y\in Y$ множество $f^{\,-l}(y)$ состоит только из
изолированных точек. Отображение $f$ будем называть {\it открытым,}
если для каждого открытого множества $A\subset X$ множество $f(A)$
открыто в $Y.$ Для заданного $D\subset X,$ определим {\it предельное
множество} отображения $f:D\rightarrow Y$ в точке $b\in
\partial D,$ обозначаемое через $C(f, b),$ как множество всех точек $z\in Y$ для которых найдётся
последовательность $b_k\in D$ такая, что $b_k\rightarrow b$ и
$f(b_k)\rightarrow z$ при $k\rightarrow\infty.$ Пусть $E\subset
\partial D$ -- непустое множество, в таком случае, положим $C(f, E)=\cup C(f, b),$ где $b$
пробегает множество $E.$ Отображение $f:G\rightarrow Y$ будем
называть {\it замкнутым} в $G\subset X,$ если множество $f(A)$
замкнуто в $f(G)$ для каждого замкнутого множества $A\in G.$
Отображение $f$ будем называть {\it собственным,} если множество
$f^{\,-1}(K)$ компактно в $D,$ как только $K$ является компактным в
$f(D).$ Отображение $f$ будем называть {\it отображением,
сохраняющим границу,} если $C(f,
\partial D)\subset \partial f(D).$

Определение функций класса $FMO$ (конечного среднего колебания),
использующееся далее по тексту, также могут быть найдены в работах
\cite{IS}--\cite{IS$_1$}.

\medskip
Следующее утверждение в случае пространства ${\Bbb R}^n$ установлено
в \cite[теорема~3.3]{Vu$_1$}.

\medskip
\begin{proposition}\label{pr4}{\sl\, Пусть $n\geqslant 2,$  $D,
D^{\,\prime}$ -- области на римановых многообразиях ${\Bbb M}^n$ и
${\Bbb M}_*^n,$ соответственно, и пусть $f:D\rightarrow {\Bbb
M}_*^n$ -- открытое, дискретное и замкнутое отображение. Тогда $f$
также является сохраняющим границу и собственным отображением.  }
\end{proposition}

\medskip
\begin{proof}
Так как $f$ -- открыто, то  $D^{\,\prime}=f(D)$ является областью.
Покажем вначале, что $f$ сохраняет границу. Предположим противное,
тогда найдутся $x_0\in
\partial D$ и $y\in D^{\,\prime}$ такие, что $y\in C(f, x_0).$ Тогда найдётся
последовательность $x_k\in D,$ $x_k\stackrel{d}{\rightarrow} x_0$
при $k\rightarrow\infty,$ $x_k\in D,$ $k=1,2,\ldots ,$ такая что
$f(x_k)\stackrel{d_*}{\rightarrow} y$ при $k\rightarrow\infty.$

\medskip
Без ограничения общности, можно считать, что $f(x_k)\ne y$ при всех
$k=1,2\ldots .$ Действительно, возможны два случая: либо $f(x_k)\ne
y$ для сколь угодно больших номеров $k,$ либо $f(x_k)=y$ при всех
$k\geqslant k_0$ и некотором натуральном $k_0.$ В первом случае, по
определению, существует подпоследовательность $x_{k_l},$
$l=1,2,\ldots,$ сходящаяся к $x_0$ при $l\rightarrow\infty$ такая,
что $f(x_{k_l})\stackrel{d_*}{\rightarrow} y$ при
$l\rightarrow\infty$ и $f(k_l)\ne y,$ $l=1,2,\ldots .$ Во втором
случае, если $f(x_k)=y$ при всех $k\geqslant k_0$ и некотором
натуральном $k_0,$ воспользуемся непрерывностью $f:$ для каждого
$k\in {\Bbb N}$ найдётся $\delta_k>0$ такое, что
\begin{equation}\label{eq3A}
d_*(f(x), f(x_k))<1/k\quad \forall \,\,x\in B(x_k, \delta_k)\,.
\end{equation}
Можно считать, что $B(x_k, \delta_k)\subset D$ и $\delta_k<1/k.$ Так
как $f$ дискретно, найдётся точка $z_k\in B(x_k, \delta_k)$ такая,
что $f(z_k)\ne y.$ Тогда по неравенству треугольника,
$$d(z_k, x_0)\leqslant d(z_k, x_k)+d(x_k, x_0)\rightarrow 0\,,\quad
k\rightarrow\infty\,,$$ и одновременно, ввиду соотношения
(\ref{eq3A})
$$d_*(f(z_k), y)\leqslant d_*(f(z_k), f(x_k))+ d_*(f(x_k), y)<$$
$$<1/k+d_*(f(x_k), y)\rightarrow 0,\quad k\rightarrow\infty\,.$$

\medskip
Таким образом, $z_k\in D,$ $z_k\stackrel{d}{\rightarrow} x_0$ при
$k\rightarrow\infty,$ $f(z_k)\stackrel{d_*}{\rightarrow} y$ при
$k\rightarrow\infty,$ и $f(z_k)\ne y$ для каждого $k\in {\Bbb N}.$ С
другой стороны, заметим, что множество $\{x_k\}_{k=1}^{\infty}$
замкнуто в $D,$ однако, $\{f(x_k)\}_{k=1}^{\infty}$ не замкнуто в
$f(D),$ поскольку $y\not\in \{f(x_k)\}_{k=1}^{\infty}.$
Следовательно, $f$ не замкнуто в $D,$ что противоречит сделанному
предположению. Полученное противоречие указывает на то, что $f$
сохраняет границу.

\medskip
Осталось доказать собственность отображения $f,$ т.е., что множество
$f^{\,-1}(K)$ компактно при каждом компактном $K\subset
D^{\,\prime}.$ Если последнее не является верным, то найдётся
последовательность $x_k\in f^{\,-1}(K),$ такая что $x_k\rightarrow
x_0\in
\partial D.$ Как было показано выше, $f(x_k)\rightarrow y_0\in \partial D^{\,\prime},$
что противоречит компактности $K.$~$\Box$
\end{proof}

\medskip
Пусть $D$ --- область риманового многообразия ${\Bbb M}^n,$
$n\geqslant 2,$ $f\colon D \rightarrow {\Bbb M}_*^n$, $n\geqslant
2,$ --- отображение, $\beta\colon [a,\,b)\rightarrow {\Bbb M}_*^n$
--- некоторая кривая и $x\in\,f^{\,-1}\left(\beta(a)\right).$ Кривая
$\alpha\colon [a,\,c)\rightarrow D,$ $c\leqslant b,$ называется {\it
максимальным поднятием} кривой $\beta$ при отображении $f$ с началом
в точке $x,$ если $(1)\quad \alpha(a)=x;$ $(2)\quad
f\circ\alpha=\beta|_{[a,\,c)};$ $(3)$\quad если
$c<c^{\prime}\leqslant b,$ то не существует кривой
$\alpha^{\prime}\colon [a,\,c^{\prime})\rightarrow D,$ такой что
$\alpha=\alpha^{\prime}|_{[a,\,c)}$ и $f\circ
\alpha=\beta|_{[a,\,c^{\prime})}.$ Имеет место следующее утверждение
(см. также \cite[лемма~3.7]{Vu$_1$} для случая пространства ${\Bbb
R}^n$).

\medskip
 \begin{proposition}\label{pr7}
{\sl Пусть $n\geqslant 2,$ $D$~--- область в ${\Bbb M}^n,$ имеющая
компактное замыкание $\overline{D}\subset {\Bbb M}^n,$ $f\colon
D\rightarrow {\Bbb M}^n_*$~--- открытое дискретное и замкнутое
отображение\/{\em,} $\beta\colon [a,\,b)\rightarrow {\Bbb
M}_*^n$~--- кривая и точка $x\in\,f^{-1}\left(\beta(a)\right).$
Тогда кривая $\beta$ имеет максимальное поднятие $\alpha:[a,
c)\rightarrow D$ при отображении $f$ с началом в точке $x,$ при этом
$c=b.$ Кроме того, если $\beta$ продолжается до замкнутой кривой
$\beta:[a, b]\rightarrow {\Bbb M}^n_*,$ то и кривая $\alpha$
продолжается до замкнутой кривой $\alpha:[a, b]\rightarrow D,$
причём $f(\alpha(t))=\beta(t),$ $t\in [a, b].$ }
 \end{proposition}

\medskip
\begin{proof}
Существование максимального поднятия $\alpha:[a, c)\rightarrow D$
при отображении $f$ с началом в точке $x$ установлено ранее в
\cite[предложение~2.1]{IS}. Покажем, что в определении максимального
поднятия $c=b.$ Предположим противное: допустим, что $c\ne b.$

Прежде всего, заметим, что в этом случае ($c\ne b$) кривая $\alpha$
не может быть продолжена до замкнутой кривой $\alpha:[a,
c]\rightarrow D.$ Действительно, если бы последнее было возможно, то
мы смогли рассмотреть новое максимальное поднятие
$\alpha^{\,\prime}:[c, c^{\,\prime})\rightarrow D$ кривой
$\beta|_{[c, b)}$ при отображении $f$ с началом в точке $\alpha(c).$
Путём объединения между собой кривых $\alpha$ и $\alpha^{\,\prime},$
мы получили бы тогда новое <<максимальное поднятие>>
$\alpha^{\,\prime\prime}:[a, c^{\,\prime})\rightarrow D$ кривой
$\beta$ при отображении $f$ с началом в точке $x,$ что противоречило
бы <<максимальности>> кривой $\alpha$ (см. условие (3) из
определения максимального поднятия).

\medskip
Таким образом, предельное множество
$$C(\alpha, c)=\left\{x\in {\Bbb M}^n:\, x=\lim\limits_{k\rightarrow\,\infty}
\alpha(t_k) \right\}\,,\quad t_k\,\in\,[a,\,c)\,,\quad
\lim\limits_{k\rightarrow\infty}t_k=c\,,$$
либо пусто, либо содержит в себе не менее двух точек. Заметим, что
$C(\alpha, c)$ лежит в $D$ ввиду замкнутости отображения $f$ в $D,$
и что в определении $C(\alpha, c)$ мы можем ограничиться только
монотонными последовательностями $t_k.$ Кроме того, множество
$\overline{\alpha}$ является компактным, поскольку
$\overline{\alpha}$ -- замкнутое подмножество компактного
пространства $\overline{D}$ (см. \cite[теорема~2.II.4,
$\S\,41$]{Ku}). По условию Кантора на компакте $\overline{\alpha},$
ввиду монотонности последовательности связных множеств
$\alpha\left(\left[t_k,\,c\right)\right),$
$$C(\alpha, c)\,=\,\bigcap\limits_{k\,=\,1}^{\infty}\,\overline{\alpha\left(\left[t_k,\,c\right)\right)}
\ne\varnothing\,,
$$
%
см. \cite[1.II.4, $\S\,41$]{Ku}. Так как $\alpha$ не может быть
продолжено по непрерывности в точку $c$ и $C(\alpha,
c)\ne\varnothing,$ то множество $C(\alpha, c)$ содержит не менее
двух точек. Кроме того, ввиду \cite[теорема~5.II.5, $\S\,47$]{Ku}
множество $\overline{\alpha}$ является связным.

С другой стороны, при $x\in C(\alpha, c),$ ввиду непрерывности $f$
мы получим, что
$f\left(\alpha(t_k)\right)\rightarrow\,f(x)$ при
$k\rightarrow\infty,$ где $t_k\in[a,\,c),\,t_k\rightarrow c$ при
$k\rightarrow \infty.$ Однако,
$f\left(\alpha(t_k)\right)=\beta(t_k)\rightarrow\beta(c)$ при
$k\rightarrow\infty.$ Следовательно, отображение $f$ постоянно на
$C(\alpha, c)$ в ${\Bbb M}^n.$

\medskip
Итак, $G$ -- связное непустое множество в $D,$ содержащее не менее
двух точек, на котором $f$ постоянно, что противоречит дискретности
$f.$ Полученное противоречие указывает на равенство $c=b.$

\medskip
Предположим теперь, что кривая $\beta$ продолжается до замкнутой
кривой $\beta:[a, b]\rightarrow {\Bbb M}^n_*.$ Покажем, что кривая
$\alpha$ также продолжается до замкнутой кривой $\alpha:[a,
b]\rightarrow D,$ причём $f(\alpha(t))=\beta(t),$ $t\in [a, b].$
Рассуждая по аналогии и используя принятые выше обозначения, мы
получим, что $C(\alpha, b)$ -- континуум в $D,$ в точках которого
отображение $f$ принимает одно и то же значение. Последнее ввиду
дискретности $f$ возможно лишь в случае, когда $C(\alpha, b)$
одноточечное множество, т.е., кривая $\alpha$ продолжается по
непрерывности в точку $b.$ Равенство $f(\alpha(t))=\beta(t)$ при
$t\in [a, b)$ вытекает из определения максимального поднятия
$\alpha,$ а при $t=b$ это равенство имеет место ввиду непрерывности
отображения $f$ в $D.$
\end{proof}~$\Box$

\medskip
Основной результат настоящего раздела содержит в себе следующее
утверждение, доказанное для пространства ${\Bbb R}^n$ в работе
\cite{Sev$_2$} (см. также случай гомеоморфизмов евклидова
$n$-мерного пространства в себя, рассмотренный в работе \cite{KR}).

\medskip
\begin{theorem}\label{th3}
{\,\sl Предположим, ${\Bbb M}^n_*={\Bbb R}^n,$ области $D$ и
$D^{\,\prime}$ имеют компактные замыкания, область $D$ регулярна,
область $D^{\,\prime}\subset {\Bbb R}^n$ имеет локально
квазиконформную границу, кроме того, $\partial D^{\,\prime}$
является сильно достижимой относительно $\alpha$-модуля, $\alpha>1.$
Пусть также отображение $f:D\rightarrow D^{\,\prime},$
$D^{\,\prime}=f(D),$ принадлежащее классу $W_{loc}^{1, \varphi}(D),$
является открытым, дискретным и замкнутым. Тогда $f$ может быть
продолжено до непрерывного отображения $f:\overline{D}_P\rightarrow
\overline{D^{\,\prime}}_P,$
$f(\overline{D}_P)=\overline{D^{\,\prime}}_P,$ если выполнено
условие
\begin{equation}\label{eqOS3.0a}
\int\limits_{1}^{\infty}\left(\frac{t}{\varphi(t)}\right)^
{\frac{1}{n-2}}dt<\infty
\end{equation}
и, кроме того, найдётся измеримая по Лебегу функция $Q,$ такая что
$K_{I,\alpha} (x, f)\leqslant Q(x)$ при почти всех $x\in D,$ и
выполнено одно из следующих условий:

1) либо в каждой точке $x_0\in
\partial D$ при некотором $\varepsilon_0>0$ и всех $\varepsilon\in
(0, \varepsilon_0)$ выполнены следующие условия:
\begin{equation*}
\int\limits_{\varepsilon}^{\varepsilon_0}
\frac{dt}{t^{\frac{n-1}{\alpha-1}}q_{x_0}^{\,\frac{1}{\alpha-1}}(t)}<\infty\,,\qquad
\int\limits_{0}^{\varepsilon_0}
\frac{dt}{t^{\frac{n-1}{\alpha-1}}q_{x_0}^{\,\frac{1}{\alpha-1}}(t)}=\infty\,;
\end{equation*}

2) либо $Q\in FMO(x_0)$ в каждой точке $x_0\in \partial D.$
}
\end{theorem}
Здесь $q_{x_0}(r):=\frac{1}{r^{n-1}}\int\limits_{S(x,
x_0)=r}Q(x)\,d{\mathcal A}$ обозначает среднее интегральное значение
функции $Q$ над геодезической сферой $S(x_0, r)$ ($d{\mathcal A}$
означает элемент площади поверхности на многообразии).

\medskip
Мы докажем теорему \ref{th3} несколько позже, после того, как будут
сформулирован и доказан ряд вспомогательных утверждений. Во-первых,
напомним определение нижнего $Q$-отображения, играющего существенную
роль при исследовании классов Орлича--Соболева (см.
\cite[гл.~9]{MRSY}). Пусть $n\geqslant 2,$ $D$ и $D^{\,\prime}$~---
заданные области в ${\Bbb M}^n$ и ${\Bbb M}_*^n,$ соответственно,
$x_0\in\overline{D}$ и $Q\colon D\to(0,\infty)$~--- измеримая
функция относительно меры объёма $v.$ Пусть $U$~--- нормальная
окрестность, содержащая точку $x_0,$ $p\geqslant 1,$ тогда будем
говорить, что $f\colon D\to D^{\,\prime}$~--- {\it нижнее $
Q$-отображение в точке $x_0$ относительно $p$-модуля}, как только
\begin{equation}\label{eq1A}
M_p(f(\Sigma_{\varepsilon}))\geqslant \inf\limits_{\rho\in{\rm
ext\,adm}\,\Sigma_{\varepsilon}}\int\limits_{D\cap A(x_0,
\varepsilon, \varepsilon_0)}\frac{\rho^p(x)}{Q(x)}\,dv(x)
\end{equation}
для каждого кольца $A(x_0, \varepsilon, \varepsilon_0),$
$\varepsilon_0\in(0,d_0),$ $d_0=\sup\limits_{x\in U}d(x, x_0),$
где $\Sigma_{\varepsilon}$ обозначает семейство всех пересечений
геодезических сфер $S(x_0, r)$ с областью $D,$ $r\in (\varepsilon,
\varepsilon_0).$

\medskip
Следующее утверждение является обобщением теоремы 9.2 в~\cite{MRSY}
на случай отображений, заданных на многообразии и произвольный
порядок модуля $p>n-1.$

\medskip
 \begin{proposition}\label{lem4A}
{\,\sl Пусть $n\geqslant 2,$ $p>n-1,$ $D$ и $D^{\,\prime}$~---
заданные области в ${\Bbb M}^n$ и ${\Bbb M}_*^n,$
соответственно{\em,} $x_0\in\overline{D}$ и $Q\colon
D\to(0,\infty)$~--- измеримая функция. Если отображение $f\colon
D\to D^{\,\prime}$ является нижним $Q$-отображением в точке $x_0$
относительно $p$-модуля, то при произвольном $\varepsilon_0>0,$
таком{\em,} что $\overline{B(x_0,\varepsilon_0)}$ лежит в нормальной
окрестности $U$ точки $x_0$ и некоторой постоянной $C_1>0$ имеем
\begin{equation}\label{eq15}
M_p(f(\Sigma_{\varepsilon}))\geqslant
C_1\int\limits_{\varepsilon}^{\varepsilon_0}
\frac{dr}{\Vert\,Q\Vert_{s}(r)}\quad\forall\
\varepsilon\in(0,\varepsilon_0)\,,\ \varepsilon_0\in(0,d_0)\,,
\end{equation}
где $s=(n-1)/(p-n+1)$ и{\em,} как и выше{\em,}
$\Sigma_{\varepsilon}$ обозначает семейство всех пересечений сфер
$S(x_0, r)$ с областью $D,$ $r\in (\varepsilon, \varepsilon_0),$
 \begin{gather*}
\Vert
Q\Vert_{s}(r)=\left(\int\limits_{D(x_0,r)}Q^{s}(x)\,d{\mathcal{A}}\right)^{1/s}\,-\,L_{s}\hbox{-норма
функции}\, Q\,\hbox{над пересечением}\\
D \cap S(x_0,r)=D(x_0,r)=\{x\in D\,|\, d(x, x_0)=r\}.
 \end{gather*}

\medskip
Обратно{\em,} если соотношение~\eqref{eq15} выполнено при некотором
$\varepsilon_0>0$ и некоторой постоянной $C_1>0,$ то $f$ является
нижним $C_2 Q$-отображением в точке $x_0,$ где $C_2>0$~--- также
некоторая постоянная.}
 \end{proposition}

\medskip
{\it Доказательство} аналогично доказательству \cite[лемма~4.2]{IS}.

\medskip
В дальнейшем нам понадобится следующее вспомогательное утверждение
(см.~\cite[лемма~4.2]{ARS}), которое при $\alpha\ne n$ может быть
доказано по аналогии.

\medskip
\begin{proposition}\label{pr1A}
{\,\sl Пусть $\alpha\geqslant 1,$ $x_0 \in {\Bbb M}^n,$
$0<r_1<r_2<{\rm dist}\,(x_0,\partial U),$ $U$~--- некоторая
нормальная окрестность точки $x_0,$ $Q\colon{\Bbb M}^n\to [0,
\infty]$ измеримая функция{\em,} локально интегрируемая относительно
меры $v$ в $U.$ Полагаем
\begin{equation*}
\eta_0(r)=\frac{1}{Ir^{\frac{n-1}{\alpha-1}}q_{x_0}^{\frac{1}{\alpha-1}}(r)}\,,
\end{equation*}
где $I:=I=I(x_0,r_1,r_2)=\int\limits_{r_1}^{r_2}\
\frac{dr}{r^{\frac{n-1}{\alpha-1}}q_{x_0}^{\frac{1}{\alpha-1}}(r)}$
и
$q_{x_0}(r):=\frac{1}{r^{n-1}}\int\limits_{|x-x_0|=r}Q(x)\,d{\mathcal
H}^{n-1}.$ Тогда при некоторых постоянных $C>0$ и $C_1>0$
$$\frac{C_1}{I^{\alpha-1}}\leqslant\int\limits_{A} Q(x)
\eta_0^{\alpha}(d(x, x_0))\ dv(x)\leqslant C\int\limits_{A} Q(x)
\eta^{\alpha}(d(x, x_0))\ dv(x)\,,
$$
$A=A(x_0, r_1, r_2),$ для любой измеримой по Лебегу функции
$\eta\colon(r_1,r_2)\to [0,\infty],$ такой{\em,} что
$\int\limits_{r_1}^{r_2}\eta(r)dr=1.$ }
\end{proposition}

\medskip
Далее $C_p[G, C_{0}, C_{1}]$ означает $p$-ёмкость пары $C_{0},
C_{1}$ относительно замыкания $G$, а $\Sigma$ обозначает класс всех
множеств, разделяющих $C_{0}$ и $C_{1}$ в $R^{\,*}=R \cup C_{0}\cup
C_{1}$ (подробно об этих и других определениях см. в работе
\cite{IS$_1$}). Заметим, что согласно результату Цимера
\begin{equation}\label{eq3B}
\widetilde{M_{p^{\,\prime}}}(\Sigma)=C_p[G , C_{0} ,
C_{1}]^{\,-1/(p-1)}\,,
\end{equation}
см. \cite[теорема~3.13]{Zi} при $p=n$ и \cite[с.~50]{Zi$_1$} при
$1<p<\infty.$ Заметим также, что согласно результату В.А.~Шлык
 \begin{equation}\label{eq4A}
M_p(\Gamma(E, F, D))= C_p[D, E, F]\,,
 \end{equation}
см.~\cite[теорема~1]{Shl}.

\medskip
Всюду ниже пространство ${\Bbb R}^n$ интерпретируется как риманово
многообразие с единственной картой $\varphi(x)=x$ и $g_{ij}(x)\equiv
E,$ где $E$ -- единичная матрица в ${\Bbb R}^n.$ Установим теперь
справедливость следующей теоремы.

\medskip
\begin{theorem}\label{th2A}
{\,\sl Пусть $n\geqslant 2,$ $p>n-1,$ ${\Bbb M}^n_*={\Bbb R}^n,$
области $D$ и $D^{\,\prime}$ имеют компактные замыкания, область $D$
регулярна, область $D^{\,\prime}\subset {\Bbb R}^n$ имеет локально
квазиконформную границу, кроме того, $\partial D^{\,\prime}$
является сильно достижимой относительно $\alpha$-модуля, $\alpha>1.$
$\alpha:=\frac{p}{p-n+1}.$ Пусть также отображение $f:D\rightarrow
D^{\,\prime},$ $D^{\,\prime}=f(D),$ является нижним $Q$-отображением
в каждой точке $x_0\in \partial D$ относительно $p$-модуля, кроме
того, $f$ является открытым, дискретным и замкнутым. Тогда $f$
продолжается до непрерывного отображения
$f:\overline{D}_P\rightarrow \overline{D^{\,\prime}}_P,$
$f(\overline{D}_P)=\overline{D^{\,\prime}}_P,$ если выполнено одно
из следующих условий:

\medskip
1) либо в каждой точке $x_0\in \partial D$ при некотором
$\varepsilon_0=\varepsilon_0(x_0)>0$ и всех
$0<\varepsilon<\varepsilon_0$
\begin{equation}\label{eq10B}
\int\limits_{\varepsilon}^{\varepsilon_0}
\frac{dt}{t^{\frac{n-1}{\alpha-1}}\widetilde{q}_{x_0}^{\,\frac{1}{\alpha-1}}(t)}<\infty\,,\qquad
\int\limits_{0}^{\varepsilon_0}
\frac{dt}{t^{\frac{n-1}{\alpha-1}}\widetilde{q}_{x_0}^{\,\frac{1}{\alpha-1}}(t)}=\infty\,,
\end{equation}
где $\alpha=\frac{p}{p-n+1},$
$\widetilde{q}_{x_0}(r):=\frac{1}{r^{n-1}}\int\limits_{|x-x_0|=r}Q^{\frac{n-1}{p-n+1}}(x)\,d{\mathcal
A}$ обозначает среднее интегральное значение функции
$Q^{\frac{n-1}{p-n+1}}(x)$ над геодезической сферой $S(x_0, r);$

2) либо $Q^{\frac{n-1}{p-n+1}}\in FMO(\partial D).$ }
\end{theorem}

\begin{proof} Докажем вначале, что $f$ имеет непрерывное продолжение
$f:\overline{D}_P\rightarrow \overline{D^{\,\prime}}_P.$ Рассмотрим
прежде всего случай 1), т.е., когда имеют место соотношения
(\ref{eq10B}). Так как $D^{\,\prime}$ имеет локально квазиконформную
границу, то $\overline{D^{\,\prime}}_P=\overline{D^{\,\prime}}$
ввиду теоремы \ref{th1}. В силу метризуемости пространства
$\overline{D}_P$ достаточно доказать, что для каждого простого конца
$P$ области $D$ предельное множество
$$L=C(f, P):=\left\{y\in{{\Bbb
R}^n}:y=\lim\limits_{m\rightarrow\infty}f(x_m),x_m\rightarrow
P,x_m\in D\right\}$$ состоит из единственной точки $y_0\in\partial
D^{\,\prime}$.

Заметим, что $L\ne\varnothing$ в силу компактности множества
$\overline{D^{\,\prime}}$, и $L$ является подмножеством $\partial
D^{\,\prime}$ ввиду предложения \ref{pr4}. Предположим, что
существуют, по крайней мере, две точки $y_0$ и $z_0\in L$. Положим
$U=B(y_0,r_0)$, где $0<r_0<d_*(y_0, z_0)$.

В силу определения регулярной области и леммы \ref{thabc2} каждый
простой конец $P$ регулярной области $D$ в ${\Bbb M}^n$, $n\geqslant
2,$ содержит цепь разрезов $\sigma_m$, лежащую на сферах $S_m$ с
центром в некоторой точке $x_0\in\partial D$ и геодезическими
радиусами $r_m\rightarrow 0$ при $m\rightarrow\infty$. Пусть $D_m$
-- области, ассоциированные с разрезами $\sigma_m$, $m=1,2,\ldots$.
Тогда существуют точки $y_m$ и $z_m$ в областях
$D_{m}^{\,\prime}=f(D_{m})$, такие что $d_*(y_0, y_m)<r_0$ и
$d_*(y_0, z_m)>r_0$ и, кроме того, $d_*(y_m, y_0)\rightarrow 0$ и
$d_*(z_m, z_0)\rightarrow 0$ при $m\rightarrow\infty$.
Соответственно, найдутся $x_m$ и $x_m^{\,\prime}$ в области $D_m,$
такие что $f(x_m)=y_m$ и $f(x_m^{\,\prime})=z_m.$ Соединим точки
$x_m$ и $y_m$ кривой $\gamma_m,$ целиком лежащей в области $D_m.$
Пусть $C_m$ -- образ этой кривой при отображении $f$ в
$D^{\,\prime}.$ Заметим, что $\partial U\cap |C_m|\ne\varnothing$
ввиду \cite[теорема~1.I.5, $\S\,46$]{Ku} (как обычно, $|C_m|$
обозначает носитель кривой $C_m$).

В силу определения сильно достижимой границы относительно
$\alpha$-модуля существует компакт $E\subset D^{\,\prime}$ и число
$\delta>0$, такие, что
\begin{equation}\label{eq1B}
M_{\alpha}(\Gamma(E, |C_m|, D^{\,\prime}))\geqslant\delta
\end{equation}
для всех достаточно больших $m$.

Без ограничения общности можем считать, что последнее условие
выполнено для всех $m=1,2,\ldots$. Заметим, что $C=f^{-1}(E)$
является компактным подмножеством области $D$ ввиду замкнутости
отображения $f$ и предложения \ref{pr4}, поэтому, поскольку
$I(P)=\bigcap\limits_{m=1}\limits^{\infty}\overline{D_m}\subset
\partial D$ (см. предложение \ref{thabc1}), то не ограничивая общности рассуждений, можно считать, что
$C\cap\overline{D_m}=\varnothing$ для каждого $m\in {\Bbb N.}$
Положим $\delta_0:=d(x_0, C).$ Уменьшая $\varepsilon_0,$ если это
необходимо, можно считать, что $\varepsilon_0<\delta_0.$

Пусть $\Gamma_m$ -- семейство всех кривых в $D,$ соединяющих $C$ и
$\sigma_m$, $m=1,2,\ldots$. Заметим, что $\Gamma(|\gamma_m|, C,
D)>\Gamma_m$ ввиду \cite[теорема~1.I.5, $\S\,46$]{Ku}, так что
$f(\Gamma(|\gamma_m|, C, D))>f(\Gamma_m)$ и ввиду (\ref{eq32*A})
\begin{equation}\label{eq1D}
M_{\alpha}(f(\Gamma(|\gamma_m|, C, D)))\leqslant
M_{\alpha}(f(\Gamma_m))\,.
\end{equation}
Оценим $M_{\alpha}(f(\Gamma(|\gamma_m|, C, D)))$ в формуле
(\ref{eq1D}) снизу.
Пусть кривая $\beta:[0,1]\rightarrow D^{\,\prime}$ такова, что
$\beta(0)\in |C_m|$ и $\beta(1)=p\in E,$ где $p$ -- некоторый
фиксированный элемент множества $E.$ Тогда для кривой $\beta$ ввиду
предложения \ref{pr7} найдётся другая кривая $\gamma:[0,
1]\rightarrow D$ с началом в $|\gamma_m|,$ такая, что
$f\circ\gamma=\beta.$ Заметим, что по определению $\beta(1)\in E,$
так что $\gamma(1)\in C$ по определению множества $C.$ Значит,
$\gamma\in \Gamma(|\gamma_m|, C, D).$ Рассмотрим семейство
$\Gamma_m^*,$ состоящее из всех возможных таких кривых $\gamma,$
тогда $\Gamma_m^*\subset\Gamma(|\gamma_m|, C, D)$ и, одновременно,
$f(\Gamma_m^*)=\Gamma(E, |C_m|, D^{\,\prime}).$ Тогда
\begin{equation}\label{eq5B}
M_{\alpha}(\Gamma(E, |C_m|,
D^{\,\prime}))=M_{\alpha}(f(\Gamma^*_m))\leqslant
M_{\alpha}(f(\Gamma(|\gamma_m|, C, D)))\,.
\end{equation}
Из (\ref{eq1B}), (\ref{eq1D}) и (\ref{eq5B}) вытекает, что
\begin{equation}\label{eq7C}
M_{\alpha}(f(\Gamma_{m}))\geqslant\delta
\end{equation}
для всех $m=1,2,\ldots$. Заметим, что
$f(\Gamma_m)\subset\Gamma(f(\sigma_m), E, D^{\,\prime}),$ поэтому из
(\ref{eq7C}) вытекает, что
\begin{equation}\label{eq7B}
M_{\alpha}(\Gamma(f(\sigma_m), E,
D^{\,\prime}))\geqslant\delta\quad\forall\quad m=1,2,\ldots .
\end{equation}

\medskip
Оценим теперь величину $M_{\alpha}(\Gamma(f(\sigma_m), E,
D^{\,\prime}))$ сверху. Для этого подберём подходящим для нас
способом систему разделяющих множеств для $E$ и $f(\sigma_m)$ и
воспользуемся определением нижнего $Q$-отображения.

\medskip
Докажем, прежде всего, что множества $E$ и $\overline{f(B(x_0,
r)\cap D)}$ не пересекаются при любом $r\in (0, \varepsilon_0).$
Предположим противное, а именно, что найдётся $\zeta_0\in E\cap
\overline{f(B(x_0, r)\cap D)}.$ Тогда найдётся последовательность
$\zeta_k\in f(B(x_0, r)\cap D)$ такая, что
$\lim\limits_{k\rightarrow\infty} d_*(\zeta_k, \zeta_0)=0$ (где
$d_*$ -- евклидова метрика). В таком случае, $\zeta_k=f(\xi_k)$ при
некоторых значениях $\xi_k\in B(x_0, r)\cap D.$ Так как
$\overline{D}$ -- компакт, то из последовательности $\xi_k$ можно
выделить сходящуюся подпоследовательность
$\xi_{k_l}\stackrel{d}{\rightarrow} \xi_0\in \overline{B(x_0, r)\cap
D}.$ Случай $\xi_0\in \partial D$ невозможен, поскольку $f$ --
замкнутое отображение и, значит (ввиду предложения \ref{pr4}),
сохраняет границу: $C(f,
\partial D)\subset \partial f(D),$ но у нас $\zeta_0$ -- внутренняя
точка $D^{\,\prime}.$ Пусть $\xi_0$ -- внутренняя точка $D.$ По
непрерывности отображения $f$ имеем $f(\xi_0)=\zeta_0.$ Но тогда
одновременно $\xi_0\in B(x_0, \varepsilon_0)\cap D$ и $\xi_0\in
f^{\,-1}(E),$ что противоречит выбору $\varepsilon_0.$ Таким
образом, $E\cap \overline{f(B(x_0, r)\cap D)}=\varnothing$ и,
значит,
\begin{equation}\label{eq8B}
E\subset D^{\,\prime}\setminus \overline{f(B(x_0, r)\cap D)}\,, r\in
(0, \varepsilon_0)\,.
\end{equation}
Из (\ref{eq8B}), в частности, вытекает, что множества $E$ и
$f(\sigma_m)$ не пересекаются.

Заметим также, что при произвольном $r\in (r_m, \varepsilon_0)$
множество $A_r:=\partial (f(B(x_0, r)\cap D))\cap D^{\,\prime}$
отделяет $E$ и $f(\sigma_m)$ в $D^{\,\prime}.$ Действительно,
$$D^{\,\prime}=B_r\cup A_r\cup C_r\qquad\forall\quad
r\in (r_m, \varepsilon_0)\,,$$
где множества $B_r:=f(B(x_0, r)\cap D)$ и
$C_r:=D^{\,\prime}\setminus \overline{f(B(x_0, r)\cap D)}$ открыты в
$D^{\,\prime},$ $f(\sigma_m)\subset B_r,$ $E\subset C_r$ и $A_r$
замкнуто в $D^{\,\prime}.$

\medskip
Пусть $\Sigma_m$ -- семейство всех множеств, отделяющих
$f(\sigma_m)$ от $E$ в $D^{\,\prime}.$ Поскольку $f$ -- открытое
замкнутое отображение, мы получим, что
\begin{equation}\label{eq7A}
(\partial f(B(x_0, r)\cap D))\cap D^{\,\prime}\subset f(S(x_0,
r)\cap D), r>0.
\end{equation}
Действительно, пусть $\zeta_0\in (\partial f(B(x_0, r)\cap D))\cap
D^{\,\prime}.$ Тогда найдётся последовательность $\zeta_k\in
f(B(x_0, r)\cap D)$ такая, что $\zeta_k\stackrel{d_*}{\rightarrow}
\zeta_0$ при $k\rightarrow \infty,$ где $\zeta_k=f(\xi_k),$
$\xi_k\in B(x_0, r)\cap D$ и $d_*$ -- евклидова метрика. Не
ограничивая общности рассуждений, ввиду компактности $\overline{D},$
можно считать, что $\xi_k\stackrel{d}{\rightarrow} \xi_0$ при
$k\rightarrow\infty,$ где $d$ -- геодезическая метрика в ${\Bbb
M}^n.$ Заметим, что случай $\xi_0\in
\partial D$ невозможен, поскольку в этом случае $\zeta_0\in C(f,
\partial D),$ что противоречит замкнутости отображения $f.$ Тогда
$\xi_0\in D.$ Возможны две ситуации: 1) $\xi_0\in B(x_0 , r)\cap D$
и 2) $\xi_0\in S(x_0 , r)\cap D.$ Заметим, что случай 1) невозможен,
поскольку, в этом случае, $f(\xi_0)=\zeta_0$ и $\zeta_0$ --
внутренняя точка множества $f(B(x_0, r)\cap D),$ что противоречит
выбору $\zeta_0.$ Таким образом, включение (\ref{eq7A}) установлено.

\medskip
Здесь и далее объединения вида $\bigcup\limits_{r\in (r_1, r_2)}
\partial f(B(x_0, r)\cap D)\cap D^{\,\prime}$ понимаются как семейства множеств. Далее запись
$\rho\in \widetilde{\rm adm}\,\Sigma$ означает, что
$\rho$~--- неотрицательная борелевская функция в ${\Bbb R}^n$,
такая, что
\begin{equation} \label{eq13.4.13}
\int\limits_{\sigma \cap R}\rho d{\mathcal H}^{n-1} \geqslant
1\quad\forall\, \sigma \in \Sigma\,, \end{equation}
а запись $\rho\in{\rm adm}\,\Gamma$ означает, что борелевская
функция $\rho\colon{\Bbb M}^n\rightarrow\overline{{\Bbb R}^+}$
удовлетворяет условию
\begin{equation}\label{eq8.2.6}\int\limits_S\rho^k\,d{\mathcal{A}}=\int\limits_{{\Bbb
M}^n}\rho^k(y)\:N(S,y)\ d\mathcal{H}^ky\geqslant 1\end{equation} для
каждой поверхности $S\in\Gamma.$
Пусть $\rho^{n-1}\in \widetilde{{\rm adm}}\bigcup\limits_{r\in (r_m,
\varepsilon_0)}
\partial f(B(x_0, r)\cap D)\cap D^{\,\prime}$ в смысле соотношения (\ref{eq13.4.13}), тогда также
$\rho\in {\rm adm}\bigcup\limits_{r\in (r_m, \varepsilon_0)}
\partial f(B(x_0, r)\cap D)\cap D^{\,\prime}$ в смысле соотношения (\ref{eq8.2.6}) при
$k=n-1.$ Ввиду (\ref{eq7A}) мы получим, что $\rho\in {\rm
adm}\bigcup\limits_{r\in (r_m, \varepsilon_0)} f(S(x_0, r)\cap D)$
и, следовательно, так как $\widetilde{M}_q(\Sigma_m)\geqslant
M_{q(n-1)}(\Sigma_m)$ при произвольном $q\geqslant 1,$ то
$$\widetilde{M}_{p/(n-1)}(\Sigma_m)\geqslant$$
\begin{equation}\label{eq5A}\geqslant
\widetilde{M}_{p/(n-1)}\left(\bigcup\limits_{r\in (r_m,
\varepsilon_0)}
\partial f(B(x_0, r)\cap D)\cap D^{\,\prime}\right)\geqslant
\end{equation}
$$\geqslant\widetilde{M}_{p/(n-1)}
\left(\bigcup\limits_{r\in (r_m, \varepsilon_0)} f(S(x_0, r)\cap
D)\right)\geqslant M_p\left(\bigcup\limits_{r\in (r_m,
\varepsilon_0)} f(S(x_0, r)\cap D)\right)\,.$$

Однако, ввиду (\ref{eq3B}) и (\ref{eq4A}), учитывая, что $p>n-1,$
имеем
\begin{equation}\label{eq6A}
\widetilde{M}_{p/(n-1)}(\Sigma_m)=\frac{1}{(M_{\alpha}(\Gamma(f(\sigma_m),
E, D^{\,\prime})))^{1/(\alpha-1)}}\,.
\end{equation}
По предложению \ref{lem4A}
$$M_{p}\left(\bigcup\limits_{r\in (r_m, \varepsilon_0)} f(S(x_0,
r)\cap D)\right)\geqslant
$$
\begin{equation}\label{eq8A}
\geqslant C_1\cdot\int\limits_{r_m}^{\varepsilon_0}
\frac{dr}{\Vert\,Q\Vert_{s}(r)}=
C_1\cdot\int\limits_{r_m}^{\varepsilon_0}
\frac{dt}{t^{\frac{n-1}{\alpha-1}}\widetilde{q}_{x_0}^{\,\frac{1}{\alpha-1}}(t)}\quad\forall\,\,
m\in {\Bbb N}\,, s=\frac{n-1}{p-n+1}\,,\end{equation}
$\alpha=p/(p-n+1),$ где
$\Vert
Q\Vert_{s}(r)=\left(\int\limits_{D(x_0,r)}Q^{s}(x)\,d{\mathcal{A}}\right)^{\frac{1}{s}}$
-- $L_{s}$-норма функции $Q$ над сферой $D(x_0,r):=S(x_0,r)\cap D.$
Из условий (\ref{eq10B}) вытекает, что
$\int\limits_{r_m}^{\varepsilon_0}
\frac{dt}{t^{\frac{n-1}{\alpha-1}}\widetilde{q}_{x_0}^{\,\frac{1}{\alpha-1}}(t)}\rightarrow\infty$
при $m\rightarrow\infty.$

Из соотношений (\ref{eq5A}) и (\ref{eq8A}) следует, что
$\widetilde{M}_{p/(n-1)}(\Sigma_m)\rightarrow\infty$ при
$m\rightarrow\infty,$ однако, в таком случае, из (\ref{eq6A})
следует, что $M_{\alpha}(\Gamma(f(\sigma_m), E,
D^{\,\prime}))\rightarrow 0$ при $m\rightarrow\infty,$ что
противоречит неравенству (\ref{eq7B}). Полученное противоречие
опровергает предположение, что предельное множество $C(f, P)$
состоит более чем из одной точки.

\medskip
Рассмотрим теперь случай 2), а именно, пусть теперь $Q^s\in
FMO(\partial D),$ $s=(n-1)/(p-n+1).$ Покажем, что в этом случае
выполнено второе условие в (\ref{eq10B}). Для этой цели
воспользуемся предложением \ref{pr1A}. Полагаем
$0\,<\,\psi(t)\,=\,\frac
{1}{\left(t\,\log{\frac1t}\right)^{n/{\alpha}}}.$ Мы можем считать,
что $\varepsilon_0<1,$ где $\varepsilon_0$ -- положительное число из
условия теоремы. Тогда заметим, что при $t\in (r_m, \varepsilon_0)$
выполнено неравенство $\psi(t)\geqslant \frac
{1}{t\,\log{\frac1t}},$ так что
$I(r_m,
\varepsilon_0)\,:=\,\int\limits_{r_m}^{\varepsilon_0}\psi(t)\,dt\,\geqslant
\log{\frac{\log{\frac{1}{r_m}}}{\log{\frac{1}{r_0}}}}.$ Положим
$\eta(t):=\psi(t)/I(r_m, r_0).$ Тогда на основании
предложения~\ref{pr1A} и \cite[предложение~3]{Af$_2$} при некоторых
постоянных $C_2, C_3>0$
будем иметь, что %
\begin{equation}\label{eq1C}\frac{C_2}{\left(\int\limits_{r_m}^{\varepsilon_0}
\frac{dt}{t^{\frac{n-1}{\alpha-1}}\widetilde{q}_{x_0}^{\,\frac{1}{\alpha-1}}(t)}\right)^{\alpha-1}}
\leqslant\frac{1}{I^{\alpha}(r_m,
\varepsilon_0)}\int\limits_{r_m<d(x, x_0)<
\varepsilon_0}\frac{Q^{\frac{n-1}{p-n+1}}(x)\, dv(x)} {\left(d(x,
x_0) \log \frac{1}{d(x, x_0)}\right)^n} \leqslant\end{equation}
$$\leqslant
C_3\cdot\left(\log\log \frac{1}{r_m}\right)^{1-\alpha}\rightarrow
0$$
при  $m \to \infty.$ Отсюда получим, что интеграл слева
в~\eqref{eq1C} расходится при $m\rightarrow\infty.$ Снова из
предложения \ref{lem4A} непосредственно вытекает, что
$M_p^{\frac{n-1}{p-n+1}}(f(\Sigma_{r_m}))\to\infty$ при $m\to
\infty,$ однако, в таком случае, из (\ref{eq6A}) следует, что
$M_{\alpha}(\Gamma(f(\sigma_m), E, D^{\,\prime}))\rightarrow 0$ при
$m\rightarrow\infty,$ что противоречит неравенству (\ref{eq7B}).
Полученное противоречие опровергает предположение, что предельное
множество $C(f, P)$ состоит более чем из одной точки. Таким образом,
утверждение теоремы о возможности непрерывного продолжения
отображения до отображения $f:\overline{D}_P\rightarrow
\overline{D^{\,\prime}}_P$ в случае 2) также установлено.

\medskip
Для окончательного завершения доказательства необходимо показать,
что $f(\overline{D}_P)=\overline{D^{\,\prime}}.$ Очевидно,
$f(\overline{D}_P)\subset\overline{D^{\,\prime}}.$ Покажем обратное
включение. Пусть $\zeta_0\in \overline{D^{\,\prime}}.$ Если
$\zeta_0$ -- внутренняя точка области $D^{\,\prime},$ то, очевидно,
существует $\xi_0\in D$ так, что $f(\xi_0)=\zeta_0$ и, значит,
$\zeta_0\in f(D).$ Пусть теперь $\zeta_0\in
\partial D^{\,\prime},$ тогда найдётся последовательность $\zeta_m\in D^{\,\prime},$
$\zeta_m=f(\xi_m),$ $\xi_m\in D,$ такая, что
$\zeta_m\stackrel{d_*}{\rightarrow} \zeta_0$ при
$m\rightarrow\infty,$ где $d_*$ -- евклидова метрика. Поскольку
$\overline{D}_P$ -- компакт ввиду замечания \ref{rem2}, то можно
считать, что $\xi_m\rightarrow P_0,$ где $P_0$ -- некоторый простой
конец в $\overline{D}_P.$ Тогда также $\zeta_0\in
f(\overline{D}_P).$ Включение $\overline{D^{\,\prime}}\subset
f(\overline{D}_P)$ доказано и, значит,
$f(\overline{D}_P)=\overline{D^{\,\prime}}.$ Теорема
доказана.~$\Box$
\end{proof}

\medskip
Пусть $D$~--- подмножество ${\Bbb M}^n.$ Для отображения $f\colon
D\,\to\,{\Bbb M}_*^n,$ множества $E\subset D$ и $y\,\in\,{\Bbb
M}_*^n$  определим {\it функцию кратности $N(y, f, E)$} как число
прообразов точки $y$ в множестве $E,$ т.е.
\begin{equation}\label{eq1.7A}
N(y, f, E)\,=\,{\rm card}\,\left\{x\in E\,|\, f(x)=y\right\}\,,
\qquad N(f, E)\,=\,\sup\limits_{y\in{\Bbb M}_*^n}\,N(y, f, E)\,.
\end{equation}
Следующая лемма в ${\Bbb R}^n$ и для случая гомеоморфизмов
установлена Д.~Ковтонюком и В.~Рязановым
в~\cite[теорема~2.1]{KR$_1$}. Для случая римановых многообразий и
отображений с ветвлением данное утверждение установлено в
\cite[лемма~4.3]{IS$_1$}.

\medskip
 \begin{lemma}{}\label{thOS4.1}
{\sl Пусть $D$~--- область в ${\Bbb M}^n,$ $n\geqslant 3,$
$\varphi\colon(0,\infty)\to (0,\infty)$~--- неубывающая
функция{\em,} удовлетворяющая условию~\eqref{eqOS3.0a}.
Если $p>n-1,$ то каждое открытое дискретное отображение $f\colon
D\to {\Bbb M}_*^n$ с конечным искажением класса
$W^{1,\varphi}_{loc},$ такое{\em,} что $N(f, D)<\infty,$ является
нижним $Q$-отображением относительно $p$-модуля в каждой точке
$x_0\in\overline{D}$ при
 $$
Q(x)=N(f, D) K^{\frac{p-n+1}{n-1}}_{I, \alpha}(x, f),\quad
\alpha:=\frac{p}{p-n+1}\,,
 $$
где внутренняя дилатация $K_{I,\alpha}(x, f)$ отображения $f$ в
точке $x$ порядка $\alpha$ определена
соотношением~\eqref{eq0.1.1A}{\em,} а кратность $N(f, D)$ определена
вторым соотношением в~\eqref{eq1.7A}.}
 \end{lemma}

\medskip
{\it Доказательство теоремы \ref{th3}}. По лемме \ref{thOS4.1}
отображение $f$ в каждой точке $x_0\in D$ является нижним
$Q$-отображением относительно $p$-модуля в каждой точке
$x_0\in\overline{D}$ при $Q(x)=N(f, D)\cdot
K^{\frac{p-n+1}{n-1}}_{I, \alpha}(x, f),$ $\alpha:=\frac{p}{p-n+1}$
(т.е., $p=\frac{\alpha(n-1)}{\alpha-1}>n-1$), где внутренняя
дилатация $K_{I,\alpha}(x, f)$ отображения $f$ в точке $x$ порядка
$\alpha$ определена соотношением (\ref{eq0.1.1A}), а кратность $N(f,
D)$ определена вторым соотношением в (\ref{eq1.7A}). Тогда
необходимое заключение вытекает из теоремы \ref{th2}, а также того
почти очевидного факта, что максимальная кратность $N(f, D)$
замкнутого открытого дискретного отображения $f$ конечна (см.,
напр., \cite[теорема~5.5]{Va$_1$}).~$\Box$

\medskip
{\bf 5. Граничное поведение кольцевых $Q$-отображений относительно
$p$-модуля.} Как мы убедились в предыдущем разделе, ключевым для
исследования классов Орлича--Соболева в плане их граничного
продолжения было то, что такие отображения удовлетворяли некоторым
оценкам относительно модуля семейств поверхностей определённого
порядка (см. лемму \ref{thOS4.1}). Только после этого становится
возможным применение основной теоремы -- теоремы \ref{th2A} для этих
классов. В данном разделе нами также исследуются классы отображений,
удовлетворяющие оценкам уже несколько иного типа. Подчеркнём, что
этим оценкам (по крайней мере, в евклидовом $n$-мерном пространстве)
удовлетворяют почти все известные отображения, начиная от
отображений конформных, квазиконформных, квазирегулярных, и
заканчивая классами отображений конечного искажения, имеющих
довольно общую природу. По этому поводу укажем на публикации
\cite{MRSY}, \cite{MRV$_1$}--\cite{MRV$_2$}, \cite{Re}, \cite{Ri} и
\cite{Pol}.

Всюду ниже ${\Bbb M}^n$ и ${\Bbb M}^n_*$~--- римановы многообразия
размерности $n\geqslant 2$ с геодезическими расстояниями $d$ и
$d_*,$ соответственно, $D, D^{\,\prime}$ -- области, принадлежащие
${\Bbb M}^n$ и ${\Bbb M}_*^n,$ соответственно. Пусть также
$0<r_1<r_2<r_0,$
\begin{equation}\label{eq2I}
A=A(x_0, r_1, r_2)=\{x\in {\Bbb M}^n\,|\,r_1<d(x, x_0)<r_2\},
\end{equation}
$S_i=S(x_0,r_i),$ $i=1,2,$~--- геодезические сферы с центром в точке
$x_0$ и радиусов $r_1$ и $r_2,$ соответственно, а
$\Gamma\left(S_1,\,S_2,\,A\right)$ обозначает семейство всех кривых,
соединяющих $S_1$ и $S_2$ внутри области $A.$ Пусть $p\geqslant 1,$
$D$~--- область в ${\Bbb M}^n,$ $Q\colon {\Bbb M}^n\to [0,
\infty]$~--- измеримая по Лебегу функция, $Q(x)\equiv 0$ при всех
$x\not\in D.$ Отображение $f\colon D\to {\Bbb M}^n$ будем называть
{\it кольцевым $Q$-отоб\-ра\-же\-нием относительно $p$-модуля в
точке $x_0\in
\partial D,$} если для некоторого $r_0=r(x_0)>0,$ такого, что шар $B(x_0, r_0)$ лежит в некоторой
нормальной окрестности точки $x_0,$ произвольных <<сферического>>
кольца $A=A(x_0,r_1,r_2),$ центрированного в точке $x_0,$ радиусов
$r_1$ и $r_2,$ $0<r_1<r_2< r_0=r(x_0),$ и любых континуумов
$E_1\subset \overline{B(x_0, r_1)}\cap D,$ $E_2\subset {\Bbb
M}^n\setminus B(x_0, r_2)\cap D$ отображение $f$ удовлетворяет
соотношению
\begin{equation}\label{eq3*!!A}
M_p\left(f\left(\Gamma\left(E_1,\,E_2,\,D\right)\right)\right)
\leqslant \int\limits_{A} Q(x) \eta^p(d(x, x_0))\ dv(x)
 \end{equation}
для каждой измеримой функции $\eta\colon(r_1,r_2)\to [0,\infty],$
такой, что имеет место соотношение
\begin{equation}\label{eq*3!!}
\int\limits_{r_1}^{r_2}\eta(r)dr\geqslant 1\,.
 \end{equation}

Основной результат настоящего раздела содержится в следующем
утверждении.

\medskip
\begin{theorem}\label{th4}
{\,\sl Пусть $n\geqslant 2,$ области $D$ и $D^{\,\prime}$ имеют
компактные замыкания, $Q:{\Bbb M}^n\rightarrow[0, \infty],$
$Q(x)\equiv 0$ на ${\Bbb M}^n\setminus D,$ $p\geqslant 1,$ область
$D\subset {\Bbb M}^n$ регулярна, а $D^{\,\prime}\subset {\Bbb M}^n$
имеет локально квазиконформную границу, являющуюся сильно достижимой
относительно $p$-модуля. Пусть также отображение $f:D\rightarrow
D^{\,\prime},$ $D^{\,\prime}=f(D),$ является кольцевым
$Q$-отображением относительно $p$-модуля в каждой точке $x_0\in
\partial D,$ кроме того, $f$ является открытым, дискретным и
замкнутым. Тогда $f$ продолжается до непрерывного отображения
$f:\overline{D}_P\rightarrow \overline{D^{\,\prime}}_P,$
$f(\overline{D}_P)=\overline{D^{\,\prime}}_P,$ если выполнено одно
из следующих условий:

1) либо в каждой точке $x_0\in
\partial D$ при некотором $\varepsilon_0=\varepsilon_0(x_0)>0$ и
всех $0<\varepsilon<\varepsilon_0$
\begin{equation}\label{eq10A}
\int\limits_{\varepsilon}^{\varepsilon_0}
\frac{dt}{t^{\frac{n-1}{p-1}}q_{x_0}^{\,\frac{1}{p-1}}(t)}<\infty\,,\qquad
\int\limits_{0}^{\varepsilon_0}
\frac{dt}{t^{\frac{n-1}{p-1}}q_{x_0}^{\,\frac{1}{p-1}}(t)}=\infty\,,
\end{equation}
где $q_{x_0}(r):=\frac{1}{r^{n-1}}\int\limits_{S(x_0,
r)}Q(x)\,d{\mathcal A};$

2) либо $Q\in FMO(x_0)$ в каждой точке $x_0\in \partial
D.$ 
}
\end{theorem}

\medskip
Как и в предыдущем разделе, для доказательства основного результата
-- теоремы \ref{th4} мы докажем сначала некое вспомогательное
утверждение, содержащее в себе утверждение указанной теоремы в
большей степени общности. Следующая лемма для случая гомеоморфизмов
на плоскости доказана в \cite[лемма~5.1]{GRY}. В нашем случае речь
идёт о ситуации римановых многообразий и отображений со значительно
более общими свойствами.

\medskip
\begin{lemma}\label{lem1A}
{\, Пусть $n\geqslant 2,$ $p\geqslant 1,$ область $D\subset {\Bbb
R}^n$ регулярна, а $D^{\,\prime}\subset {\Bbb R}^n$ ограничена и
имеет локально квазиконформную границу, являющуюся сильно достижимой
относительно $p$-модуля. Пусть также отображение $f:D\rightarrow
D^{\,\prime},$ $D^{\,\prime}=f(D),$ является кольцевым
$Q$-отображением относительно $p$-модуля во всех точках
$x_0\in\partial D,$ кроме того, $f$ является открытым, дискретным и
замкнутым. Тогда $f$ продолжается до непрерывного отображения
$f:\overline{D}_P\rightarrow \overline{D^{\,\prime}}_P,$
$f(\overline{D}_P)=\overline{D^{\,\prime}}_P,$ если найдётся
измеримая по Лебегу функция $\psi:(0, \infty)\rightarrow [0,
\infty]$ такая, что
\begin{equation} \label{eq5C}
I(\varepsilon,
\varepsilon_0):=\int\limits_{\varepsilon}^{\varepsilon_0}\psi(t)dt <
\infty
\end{equation}
при всех $\varepsilon\in (0,\varepsilon_0)$ и, кроме того,
$I(\varepsilon, \varepsilon_0)\rightarrow\infty$ при
$\varepsilon\rightarrow 0,$ и при $\varepsilon\rightarrow 0$
\begin{equation} \label{eq4*}
\int\limits_{\varepsilon<d(x_0,
x)<\varepsilon_0}Q(x)\cdot\psi^p(d(x_0, x)) \
dv(x)\,=\,o\left(I^p(\varepsilon, \varepsilon_0)\right)\,.
\end{equation}}
\end{lemma}

\begin{proof}
Так как $D^{\,\prime}$ имеет локально квазиконформную границу, то
$\overline{D^{\,\prime}}_P=\overline{D^{\,\prime}}$ ввиду теоремы
\ref{th1}. В силу метризуемости пространства $\overline{D}_P$
достаточно доказать, что для каждого простого конца $P$ области $D$
предельное множество
$$L=C(f, P):=\left\{y\in{{\Bbb
M}_*^n}:y=\lim\limits_{m\rightarrow\infty}f(x_m),x_m\rightarrow
P,x_m\in D\right\}$$ состоит из единственной точки $y_0\in\partial
D^{\,\prime}$.

Заметим, что $L\ne\varnothing$ в силу компактности множества
$\overline{D^{\,\prime}}$, и $L$ является подмножеством $\partial
D^{\,\prime}$ ввиду предложения \ref{pr4}. Предположим, что
существуют, по крайней мере, две точки $y_0$ и $z_0\in L.$ То есть,
найдётся не менее двух последовательностей $x_k, x_k^{\,\prime}\in
D,$ таких, что $x_k\rightarrow P$ и $x^{\,\prime}_k\rightarrow P$
при $k\rightarrow\infty,$ и при этом, $f(x_k)\rightarrow y_0$ и
$f(x^{\,\prime}_k)\rightarrow z_0$ при $k\rightarrow\infty.$ В силу
определения регулярной области и леммы \ref{thabc2} каждый простой
конец $P$ регулярной области $D$ в ${\Bbb M}^n$, $n\geqslant 2,$
содержит цепь разрезов $\sigma_m$, лежащую на сферах $S_m$ с центром
в некоторой точке $x_0\in\partial D$ и геодезическими радиусами
$r_m\rightarrow 0$ при $m\rightarrow\infty$. Пусть $D_k$ -- области,
ассоциированные с разрезами $\sigma_k$, $k=1,2,\ldots$. Не
ограничивая общности рассуждений, переходя к подпоследовательности,
если это необходимо, мы можем считать, что $x_k, x_k^{\,\prime}\in
D_k.$ В самом деле, так как последовательности $x_k$ и
$x_k^{\,\prime}$ сходятся к простому концу $P,$ найдётся номер
$k_1\in {\Bbb N}$ такой, что $x_{k_1}, x_{k_1}^{\,\prime}\in D_1.$
Далее, найдётся номер $k_2\in {\Bbb N},$ $k_2>k_1,$ такой, что
$x_{k_2}, x_{k_2}^{\,\prime}\in D_2.$ И так далее. Вообще, на $m$-м
шаге мы найдём номер $k_m\in {\Bbb N},$ $k_m>k_{m-1},$ такой, что
$x_{k_m}, x_{k_m}^{\,\prime}\in D_m.$ Продолжая этот процесс, мы
получим две последовательности $x_{k_m}$ и $x^{\,\prime}_{k_m},$
принадлежащие области $D_m,$ сходящиеся к $P$ при
$m\rightarrow\infty$ и такие, что $f(x_{k_m})\rightarrow y_0$ и
$f(x^{\,\prime}_{k_m})\rightarrow y_0$ при $m\rightarrow\infty.$
Переобозначая, если это необходимо, $x_{k_m}\mapsto x_m,$ мы
получаем последовательность $x_m$ с требуемыми свойствами.

По определению сильно достижимой границы в точке $y_0\in \partial
D^{\,\prime}$ относительно $p$-модуля, для любой окрестности $U$
этой точки найдутся компакт $C_0^{\,\prime}\subset D^{\,\prime},$
окрестность $V$ точки $y_0,$ $V\subset U,$ и число $\delta>0$ такие,
что
\begin{equation}\label{eq1E}
M_p(\Gamma(C_0^{\,\prime}, F, D^{\,\prime}))\ge \delta
>0
\end{equation} для произвольного континуума
$F,$ пересекающего $\partial U$ и $\partial V.$ Так как отображение
$f$ -- замкнутое, ввиду предложения \ref{pr4} для множества
$C_0:=f^{\,-1}(C_0^{\,\prime})$ выполнено условие $C_0\cap \partial
D=\varnothing.$ Поскольку
$I(P)=\bigcap\limits_{m=1}\limits^{\infty}\overline{D_m}\subset
\partial D$ (см. предложение~\ref{thabc1}), то не ограничивая общности рассуждений, можно считать, что
$C_0\cap\overline{D_k}=\varnothing$ для каждого $k\in {\Bbb N.}$
Соединим точки $x_k$ и $x_k^{\,\prime}$ кривой $\gamma_k,$ лежащей в
$D_k.$ Заметим, что $f(x_k)\in V$ и $f(x_k^{\,\prime})\in D\setminus
\overline{U}$ при всех достаточно больших $k\in {\Bbb N}.$ В таком
случае, найдётся номер $k_0\in {\Bbb N},$ такой, что согласно
(\ref{eq1E})
\begin{equation}\label{eq2A}
M_p(\Gamma(C_0^{\,\prime}, |f(\gamma_k)|, D^{\,\prime}))\ge \delta
>0
\end{equation}
при всех $k\ge k_0\in {\Bbb N}.$
При каждом фиксированном $k\in {\Bbb N},$ $k\ge k_0,$ рассмотрим
семейство $\Gamma_k^{\,\prime}$ (полных) поднятий $\alpha:[a,
b]\rightarrow D$ семейства $\Gamma\left(C_0^{\,\prime},
|f(\gamma_k)|, D^{\,\prime}\right)$ с началом во множестве
$|\gamma_k|,$ т.е., $f\circ\alpha=\beta,$ $\beta\in
\Gamma\left(C_0^{\,\prime}, |f(\gamma_k)|, D^{\,\prime}\right)$ и
$\alpha(a)\in |\gamma_k|.$ Заметим, что по определению
$\overline{\beta}(b)\in C_0^{\,\prime},$ так что
$\overline{\alpha}(b)\in C_0$ по определению множества $C_0.$
Значит, $\alpha\in \Gamma(|\gamma_k|, C_0, D).$ Погрузим компакт
$C_0$ в некоторый континуум $C_1,$ всё ещё полностью лежащий в
области $D$ (см.~\cite[лемма~1]{Smol}). Можно снова считать, что
$C_1\cap\overline{D_k}=\varnothing,$ $k=1,2,\ldots.$ Заметим, что
$\Gamma(|\gamma_k|, C_0, D)>\Gamma(\sigma_k, C_1, D),$ при этом,
$|\gamma_k|$ и $C_0$ -- континуумы в $D,$ а $\sigma_k$ -- разрез
соответствующий области $D_k.$ Поэтому к семейству кривых
$\Gamma(\sigma_k, C_1, D)$ можно применить определение кольцевого
$Q$-отображения (\ref{eq3*!!A}). Как уже было отмечено выше,
$\sigma_k\in S(x_0, r_k)$ для некоторой точки $x_0\in
\partial D$ и некоторой последовательности $r_k>0,$ $r_k\rightarrow
0$ при $k\rightarrow\infty.$ Здесь, не ограничивая общности
рассуждений, можно считать, что ${\rm dist}\,(x_0,
C_1)>\varepsilon_0.$ Кроме того, заметим, что функция
$$\eta_k(t)=\left\{
\begin{array}{rr}
\psi(t)/I(r_k, \varepsilon_0), &   t\in (r_k,
\varepsilon_0),\\
0,  &  t\in {\Bbb R}\setminus (r_k, \varepsilon_0)\,,
\end{array}
\right. $$
где $I(\varepsilon,
\varepsilon_0):=\int\limits_{\varepsilon}^{\varepsilon_0}\psi(t)dt,$
удовлетворяет условию нормировки вида (\ref{eq*3!!}). По доказанному
$\Gamma_k^{\,\prime}\subset\Gamma(|\gamma_k|, C_0, D),$ так что
$M_p(f(\Gamma_k^{\,\prime}))\leqslant M_p(f(\Gamma(|\gamma_k|, C_0,
D))).$ Поэтому, в силу определения кольцевого $Q$-отоб\-ра\-же\-ния
в граничной точке относительно $p$-модуля, а также ввиду условий
(\ref{eq5C})--(\ref{eq4*}),
\begin{equation}\label{eq11*}
M_p(f(\Gamma_k^{\,\prime}))=M_p(f(\Gamma(|\gamma_k|, C_0,
D)))\leqslant M_p(f(\Gamma(\sigma_k, C_1, D))\leqslant \Delta(k)\,,
\end{equation}
где $\Delta(k)\rightarrow 0$ при $k\rightarrow \infty.$ Однако,
$f(\Gamma_k^{\,\prime})=\Gamma(C_0^{\,\prime}, |f(\gamma_k)|,
D^{\,\prime}),$ поэтому из (\ref{eq11*}) получим, что при
$k\rightarrow \infty$
\begin{equation}\label{eq3E}
M_p(\Gamma(C_0^{\,\prime}, |f(\gamma_k)|, D^{\,\prime}))=
M_p\left(f(\Gamma_k^{\,\prime})\right)\leqslant \Delta(k)\rightarrow
0\,.
\end{equation}
Однако, соотношение (\ref{eq3E}) противоречит неравенству
(\ref{eq2A}), что и доказывает лемму.~$\Box$
\end{proof}

\medskip
{\it Доказательство теоремы \ref{th4}} сводится к лемме \ref{lem1A}
на основании подбора функций $\psi$ из этой леммы в подходящем для
нас виде (см. по этому поводу \cite[доказательство теорем~1.1 и
2.1]{IS$_1$}).~$\Box$

\medskip
{\bf 6. Гомеоморфное продолжение отображений на границу. Граничное
поведение обратных отображений}. Основным объектом настоящего
раздела является продолжение отображений $f:D\rightarrow
D^{\,\prime}$ до гомеоморфизма $f:\overline{D}_P\rightarrow
\overline{D^{\,\prime}}_P.$ Подчеркнём, что теоремы \ref{th2A} и
\ref{th4} не гарантируют, вообще говоря, такого продолжения.
Ситуация меняется в случае, если отображённая область $D^{\,\prime}$
имеет более <<хорошие>> свойства. Как будет показано ниже, в
качестве таких свойств достаточно взять область со слабо плоской
границей. (По этому поводу следует также упомянуть случай простых
концов, относящийся к пространству ${\Bbb R}^n,$ см. \cite[теоремы~3
и 4]{KR}. Кроме того, аналогичные результаты были получены в
терминах поточечного граничного продолжения на многообразиях и
метрических пространствах, см., напр., \cite[теорема~13.5]{MRSY},
\cite[теоремы~6.2 и 6.3]{ARS} и \cite[теорема~5]{Smol}).

\medskip
Следующие два утверждения доказаны в работе \cite{KR} в пространстве
${\Bbb R}^n$ для несколько иных классов отображений (см. лемму 4 и
теорему 1), а также на плоскости в случае тех же классов (см. лемму
6.1 и теорему 6.1 в \cite{GRY}). Доказательство этих утверждений не
содержит существенных отличий упомянутых случаев, однако, ради
полноты изложения их важно привести полностью.

\medskip
\begin{lemma}\label{l:9.1} {\sl Пусть $n\geqslant 2,$ $D\subset {\Bbb M}^n$ и $D^{\,\prime}\subset {\Bbb M}_*^n$ -- регулярные области
с компактными замыканиями,  $P_1$ и $P_2$ -- различные простые концы
области $D$, и пусть $\sigma_m$, $m=1,2,\ldots$, -- цепь разрезов
простого конца $P_1$ из леммы \ref{thabc2}, лежащая на сферах
$S(z_1,r_m)$, $z_1\in I(P_1)$, с ассоциированными областями $D_m$.
Предположим, что функция $Q:{\Bbb M}^n\rightarrow [0, \infty]$
интегрируема по сферам
\begin{equation*}\label{INTERSECTION}
D(r)=\left\{x\in D: d(x, z_1)=r\right\}=D\cap
S(z_1,r)\end{equation*}
для некоторого множества $E$ чисел $r\in(0,d)$ положительной
линейной меры, где $d=r_{m_0}$ и $m_0$ минимальное из чисел, таких,
что область $D_{m_0}$ не содержит последовательностей точек,
сходящихся к $P_2$. Если $f$ является кольцевым $Q$-гомеоморфизмом
области $D$ на область $D^{\,\prime}$ в точке $z_1$ и $\partial
D^{\,\prime}$ является слабо плоской (т.е., выполнено условие
(\ref{eq3})), то
\begin{equation*}\label{e:9.2}C(f, P_1)\cap C(f, P_2)=\varnothing.\end{equation*}}
\end{lemma}
Заметим, что в силу метризуемости пополнения $\overline{D}_P$
области $D$ простыми концами, см. сделанное выше замечание, число
$m_0$ в лемме \ref{l:9.1} всегда существует.

\medskip
\begin{proof} Выберем $\varepsilon\in(0,d)$ так, что $E_0:=\{r\in
E:r\in(\varepsilon,d)\}$ имеет положительную линейную меру. Такой
выбор возможен в силу полуаддитивности линейной меры и исчерпания
$E=\cup E_m$, где $E_m=\{r\in E:r\in(1/m,d)\}$, $m=1,2,\ldots$.
Полагая $S_1=S(z_1, \varepsilon),$ $S_2=S(z_1, d).$ Заметим, что
функция $\eta(t)=\frac{1}{Itq_{x_0}^{\frac{1}{n-1}}(t)},$
$I=\int\limits_{\varepsilon}^d\frac{dr}{rq_{x_0}^{\frac{1}{n-1}}(r)},$
удовлетворяет соотношению (\ref{eq*3!!}) при $r_1=\varepsilon$ и
$r_2=d,$ поэтому из определения кольцевого $Q$-отображения и ввиду
аналога теоремы Фубини на многообразиях (см.
\cite[соотношение~(2.15), замечание~2.1]{IS}) вытекает, что найдётся
постоянная $C>0$ такая, что
\begin{equation}\label{e:9.3}
M(f(\Gamma(S_1, S_2, D)))\leqslant
\frac{C}{I^{n-1}}<\infty\,.
\end{equation}
Предположим, что $C_1\cap C_2\ne\varnothing$, где $C_i=C(f, P_i)$,
$i=1,2$. По построению существует $m_1>m_0$, такое, что
$\sigma_{m_1}$ лежит на сфере $S(z_1,r_{m_1})$ с
$r_{m_1}<\varepsilon$. Пусть $D_0=D_{m_1}$ и $D_*\subseteq
D\setminus D_{m_0}$ -- область, ассоциированная с цепью разрезов
простого конца $P_2$. Пусть $y_0\in C_1\cap C_2$.

Заметим, что можно выбрать $r_0>0$ так, что $S(y_0,r_0)\cap
f(D_0)\ne\varnothing$ и $S(y_0,r_0)\cap f(D_*)\ne\varnothing$.
Действительно, так как $y_0\in C_1\cap C_2,$ то, в частности,
$y_0\in \overline{f(D_0)},$ откуда следует, что в произвольной
окрестности $U=B(y_0, r_1)$ точки $y_0$ имеется точка $x_1\in B(y_0,
r_1)\cap f(D_0).$ Точно также, $y_0\in \overline{f(D_*)}$ и, значит,
в этой же окрестности $U$ найдётся точка $x_2\in B(y_0, r_1)\cap
f(D_*).$ Пусть $r_0:=\min\{d_*(y_0, x_1), d_*(y_0, x_2)\},$ где
$d_*$ -- геодезическое расстояние в ${\Bbb M}^n_*.$ Заметим, что по
построению $f(D_0)\cap B(y_0, r_0)\ne\varnothing\ne f(D_0)\setminus
B(y_0, r_0)$ и $f(D_*)\cap B(y_0, r_0)\ne\varnothing\ne
f(D_*)\setminus B(y_0, r_0).$ Тогда ввиду \cite[теорема~1, гл.~5,
\S\, 46]{Ku} мы имеем, что $S(y_0, r_0)\cap f(D_*)\ne\varnothing\ne
S(y_0, r_0)\cap f(D_0),$ что и требовалось установить. Заметим, что
ввиду аналогичных соображений  $S(y_0, r_*)\cap
f(D_*)\ne\varnothing\ne S(y_0, r_*)\cap f(D_0)$ при любом $r_*\in
(0, r_0).$

Обозначим $\Gamma=\Gamma(\overline{D_0},\overline{D_*}, D)$. Тогда
по принципу минорирования и из (\ref{e:9.3}) следует, что
\begin{equation*}\label{e:9.4}
M(f(\Gamma))\leqslant M(f(\Gamma(S_1, S_2,
D)))<\infty\,.\end{equation*}
Пусть $M_0>M(f(\Gamma))$ -- конечное число. Из условия слабой
плоскости $\partial D^{\,\prime}$ следует, что найдется
$r_*\in(0,r_0)$, такое, что
$$M(\Gamma(E, F, D^{\,\prime}))\geqslant M_0$$
для всех континуумов $E$ и $F$ в $D^{\,\prime}$, пересекающих сферы
$S(y_0,r_0)$ и $S(y_0,r_*)$. Однако, эти сферы могут быть соединены
непрерывными кривыми $c_1$ и $c_2$ в областях $f(D_0)$ и $f(D_*)$ и,
в частности, для этих кривых
\begin{equation*}\label{e:9.4a}
M_0\leqslant M(\Gamma(c_1, c_2, D^{\,\prime}))\leqslant
M(f(\Gamma)).\end{equation*} Полученное противоречие опровергает
предположение, что $C_1\cap C_2\ne\varnothing$.~$\Box$ \end{proof}

\medskip
Из леммы \ref{lem1A} вытекают следующие два важнейшие утверждения
настоящего раздела.

\medskip
\begin{theorem}\label{t:9.5} {\sl Пусть $n\geqslant 2,$ $D\subset {\Bbb M}^n$ и $D^{\,\prime}\subset {\Bbb M}_*^n$ -- регулярные области
с компактными замыканиями.  Если $f$ -- кольцевой $Q$-гомеоморфизм
$D$ на $D^{\,\prime}$ в каждой точке $x_0\in \partial D$ и $Q\in
L(D)$, то $f^{\,-1}$ продолжается до непрерывного отображения
$\overline{D^{\,\prime}}_P$ на $\overline{D}_P$.}
\end{theorem}

\begin{proof} Ввиду аналога теоремы Фубини на многообразиях (см.
\cite[замечание~2.1]{IS}), множество
$$E(x_0)=\left\{r\in(0,d(x_0)):Q|_{D(x_0,r)}\in
L(D(x_0,r))\right\} \quad\forall\ x_0\in\partial D,$$
где $d(x_0)=\sup_{x\in D}d(x,x_0)$ и $D(x_0,r)=D\cap S(x_0,r)$,
имеет положительную линейную меру, поскольку $Q\in L(D)$. Согласно
сделанным во введении замечаниям,  без ограничения общности можем
считать, что область $D^{\,\prime}$ имеет слабо плоскую границу.
(Если это не так, то следует рассмотреть вспомогательное
квазиконформное отображение $g$ области $D^{\,\prime}$ на область
$D^{\,\prime\prime}$ с локально квазиконформной границей. Такое
отображение существует по определению регулярной области. В
частности, $D^{\,\prime\prime}$ имеет слабо плоскую границу по лемме
\ref{lem2}. Отображение $\varphi:=g\circ f$ будет тогда кольцевым
$K\cdot Q$-гомеоморфизмом в области $D,$ где $K\geqslant 1$ --
некоторая постоянная. Установив заключение теоремы для $\varphi,$ мы
установим тем самым и желанное заключение теоремы \ref{t:9.5},
используя соотношение $f^{\,-1}=\varphi^{\,-1}\circ g$ ).

Заметим, что для произвольного $\zeta_0\in \partial D^{\,\prime}$
множество $C(f^{\,-1}, \zeta_0)$ состоит из одной точки $\xi_0\in
E_D,$ где $E_D$ -- пространство простых концов области $D.$ В самом
деле, если мы имеем не менее двух последовательностей
$x_k\stackrel{d_*}{\rightarrow} \zeta_0$ при $k\rightarrow\infty$ и
$y_k \stackrel{d_*}{\rightarrow}\zeta_0$ при $k\rightarrow\infty$
таких, что $f^{\,-1}(x_k)\rightarrow P_1\in E_D$ и
$f^{\,-1}(y_k)\rightarrow P_2\in E_D$ при $k\rightarrow\infty,$
$P_1\ne P_2,$ то $\zeta_0\in C(f, P_1)\cap C(f, P_2),$ что
противоречит утверждению леммы \ref{l:9.1}.

Таким образом, мы имеем продолжение $f^{\,-1}$ на
$\overline{D^{\,\prime}}$ такое, что $C(f^{\,-1}, \partial
D^{\,\prime})\subset \overline{D}_P\setminus D.$ Покажем, что
$C(f^{\,-1}, \partial D^{\,\prime})=\overline{D}_P\setminus D.$
Действительно, если $P_0$ -- простой конец в $D,$ то  найдётся
последовательность $x_m\rightarrow P_0$ при $m\rightarrow \infty.$
Ввиду компактности $\overline{D}$ и $\overline{D^{\,\prime}}$ мы
можем считать, что $x_m\rightarrow x_0\in \partial D$ и
$f(x_m)\stackrel{d_*}{\rightarrow} \zeta_0\in
\partial D^{\,\prime}$ при $m\rightarrow\infty.$ Последнее означает,
что $P_0\in C(f^{\,-1}, \zeta_0).$

Покажем, наконец, что продолженное отображение
$g:\overline{D^{\,\prime}}\rightarrow \overline{D}_P$ непрерывно в
$\overline{D^{\,\prime}}.$ Действительно, пусть $\zeta_m\rightarrow
\zeta_0$ при $m\rightarrow\infty,$ $\zeta_m, \zeta_0\in
\overline{D^{\,\prime}}.$ Если $\zeta_0$ -- внутренняя точка области
$D^{\,\prime},$ желанное утверждение очевидно. Пусть $\zeta_0\in
\partial D^{\,\prime},$ тогда выберем $\zeta_m^*\in D^{\,\prime}$
таким, что $d_*(\zeta_m,\zeta_m^*)<1/m$ и $\rho(g(\zeta_m),
g(\zeta_m^*))<1/m,$ где $\rho$ -- одна из метрик, указанных в
замечании \ref{rem1}. По построению $g(\zeta_m^*)\rightarrow
g(\zeta_0),$ поскольку $\zeta_m^*\stackrel{d_*}{\rightarrow}
\zeta_0$ и, значит, также и $g(\zeta_m)\rightarrow g(\zeta_0)$ при
$m\rightarrow\infty.$~$\Box$
\end{proof}

\medskip
Имеет место также следующая

\medskip
\begin{theorem}\label{th4A}
{\,\sl Пусть $n\geqslant 2,$ области $D$ и $D^{\,\prime}$ регулярны
и имеют компактные замыкания, $Q:{\Bbb M}^n\rightarrow[0, \infty],$
$Q(x)\equiv 0$ на ${\Bbb M}^n\setminus D.$ Пусть также отображение
$f:D\rightarrow D^{\,\prime},$ $D^{\,\prime}=f(D),$ является
кольцевым $Q$-гомеоморфизмом в каждой точке $x_0\in
\partial D.$ Тогда $f$ продолжается до гомеоморфизма
$f:\overline{D}_P\rightarrow \overline{D^{\,\prime}}_P,$
$f(\overline{D}_P)=\overline{D^{\,\prime}}_P,$ если выполнено одно
из следующих условий:

1) либо в каждой точке $x_0\in
\partial D$ при некотором $\varepsilon_0=\varepsilon_0(x_0)>0$ и
всех $0<\varepsilon<\varepsilon_0$
\begin{equation}\label{eq10C}
\int\limits_{\varepsilon}^{\varepsilon_0}
\frac{dt}{tq_{x_0}^{\,\frac{1}{n-1}}(t)}<\infty\,,\qquad
\int\limits_{0}^{\varepsilon_0}
\frac{dt}{tq_{x_0}^{\,\frac{1}{n-1}}(t)}=\infty\,,
\end{equation}
где $q_{x_0}(r):=\frac{1}{r^{n-1}}\int\limits_{S(x_0,
r)}Q(x)\,d{\mathcal A};$

2) либо $Q\in FMO(x_0)$ в каждой точке $x_0\in \partial
D.$ 
}
\end{theorem}

\medskip
\begin{proof}
Заметим, что функция $Q$ удовлетворяет условиям, сформулированным в
лемме \ref{l:9.1}. В самом деле, в случае 1) условие (\ref{eq10C})
гарантирует интегрируемость функции $Q$ на некотором множестве сфер
$S(x_0, r)$ положительной линейной меры относительно $r,$ а
выполнение указанного условия (более того, локальная интегрируемость
функции $Q$) в случае 2) вытекает непосредственно из определения
класса $FMO.$ В таком случае, возможность непрерывного продолжения
$f:\overline{D}_P\rightarrow \overline{D^{\,\prime}}_P$ вытекает из
леммы \ref{th4}, а гомеоморфность указанного продолжения есть
результат леммы \ref{l:9.1}.~$\Box$
\end{proof}

\medskip
{\bf 7. Связь верхних и нижних модульных оценок в граничных точках.}
В предыдущих разделах мы сталкивались с двумя видами оценок модулей
семейств кривых: независимо друг от друга рассматривались оценки
вида (\ref{eq1A}) и (\ref{eq3*!!A}). Вопросы взаимосвязи этих оценок
неоднократно рассматривались ранее (см., напр.,
\cite[предложение~3]{KR} либо \cite[теорема~12.3]{KSS}). В наиболее
общей ситуации, когда порядок модуля семейств кривых --
произвольный, результат об указанной связи может быть сформулирован
следующим образом.

\medskip
\begin{theorem}\label{th4AB}
{\,\sl Пусть $D\subset {\Bbb M}^n,$ $D^{\,\prime}\subset {\Bbb
R}^n,$ $n\geqslant 2,$ область $D^{\,\prime}$ ограничена, $x_0\in
\partial D,$ отображение $f:D\rightarrow D^{\,\prime}$
является нижним $Q$-гомеоморфизмом относительно $p$-модуля в области
$D,$ $Q\in L_{loc}^{\frac{n-1}{p-n+1}}({\Bbb M}^n),$ $Q(x)\equiv 0$
на ${\Bbb M}^n\setminus D,$ $p>n-1$ и $\alpha:=\frac{p}{p-n+1}.$
Тогда $f$ является кольцевым $C\cdot
Q^{\frac{n-1}{p-n+1}}$-гомеоморфизмом в этой же точке, где $C>0$ --
некоторая постоянная.
 }
\end{theorem}

\medskip
\begin{proof} Зафиксируем $\varepsilon_0\in(0,d_0),$ $d_0=\sup\limits_{x\in
D}d(x, x_0).$ Пусть $\varepsilon\in(0, \varepsilon_0)$ и пусть
континуумы $C_1$ и $C_2$ удовлетворяют условиям  $C_1\subset
\overline{B(x_0, \varepsilon)}\cap D$ и $C_2\subset D\setminus
B(x_0, \varepsilon_0).$ Рассмотрим семейство множеств
$\Gamma_{\varepsilon}:=\bigcup\limits_{r\in (\varepsilon,
\varepsilon_0)}\{f(S(x_0, r)\cap D)\}.$ Заметим, что множество
$\sigma_r:=f(S(x_0, r)\cap D)$ замкнуто в $f(D)$ как гомеоморфный
образ замкнутого множества $S(x_0, r)\cap D$ в $D.$ Кроме того,
заметим, что $\sigma_r$ при $r\in (\varepsilon, \varepsilon_0)$
отделяет $f(C_1)$ от $f(C_2)$ в $f(D),$ поскольку
$$f(C_1)\subset f(B(x_0, r)\cap D):=A,\quad
f(C_2)\subset f(D)\setminus \overline{f(B(x_0, r)\cap D)}:=B\,,$$
$A$ и $B$ открыты в $f(D)$ и
$$f(D)=A\cup \sigma_r\cup B\,.$$

\medskip
Пусть $\Sigma_{\varepsilon}$ -- семейство всех множеств, отделяющих
$f(C_1)$ от $f(C_2)$ в $f(D).$ Пусть $\rho^{n-1}\in \widetilde{{\rm
adm}}\bigcup\limits_{r\in (\varepsilon, \varepsilon_0)} f(S(x_0,
r)\cap D)$ в смысле соотношения (\ref{eq13.4.13}), тогда также
$\rho\in {\rm adm}\bigcup\limits_{r\in (\varepsilon, \varepsilon_0)}
f(S(x_0, r)\cap D)$ в смысле соотношения (\ref{eq8.2.6}) при
$k=n-1.$ Следовательно, так как
$\widetilde{M}_{q}(\Sigma_{\varepsilon})\geqslant
M_{q(n-1)}(\Sigma_{\varepsilon})$ при произвольном $q\geqslant 1,$
то
$$\widetilde{M}_{p/(n-1)}(\Sigma_{\varepsilon})\geqslant$$
\begin{equation}\label{eq5AA}
\geqslant \widetilde{M}_{p/(n-1)}\left(\bigcup \limits_{r\in
(\varepsilon, \varepsilon_0)} f(S(x_0, r)\cap D)\right)\geqslant
M_p\left(\bigcup\limits_{r\in (\varepsilon, \varepsilon_0)} f(S(x_0,
r)\cap D)\right)\,.
\end{equation}
Однако, ввиду (\ref{eq3}) и (\ref{eq4}),
\begin{equation}\label{eq1F}
\widetilde{M}_{p/(n-1)}(\Sigma_{\varepsilon})=\frac{1}{(M_{\alpha}(\Gamma(f(C_1),
f(C_2), f(D))))^{1/(\alpha-1)}}\,, \alpha=p/(p-n+1)\,.
\end{equation}
По предложению \ref{lem4A}
$$M_p\left(\bigcup\limits_{r\in (\varepsilon, \varepsilon_0)} f(S(x_0,
r)\cap D)\right)\geqslant
$$
\begin{equation}\label{eq8BA}
\geqslant C_1\cdot\int\limits_{\varepsilon}^{\varepsilon_0}
\frac{dr}{\Vert\,Q\Vert_{s}(r)}=
C_1\cdot\int\limits_{\varepsilon}^{\varepsilon_0} \frac{dt}{
t^{\frac{n-1}{\alpha-1}}\widetilde{q}_{x_0}^{\,\frac{1}{\alpha-1}}(t)}\quad\forall\,\,
i\in {\Bbb N}\,, s=\frac{n-1}{p-n+1}\,,\end{equation} где
$\Vert
Q\Vert_{s}(r)=\left(\int\limits_{D(x_0,r)}Q^{s}(x)\,d{\mathcal{A}}\right)^{\frac{1}{s}}$
-- $L_{s}$-норма функции $Q$ над сферой $S(x_0,r)\cap D,$ а
$\widetilde{q}_{x_0}(r)$ -- её среднее значение над этой сферой.
Тогда из (\ref{eq5AA})--(\ref{eq8BA}) вытекает, что
\begin{equation}\label{eq9C}
M_{\alpha}(\Gamma(f(C_1), f(C_2), f(D)))\leqslant
\frac{C_2}{I^{\alpha-1}}\,,
\end{equation}
где $C_2>0$ -- некоторая постоянная,
$I=\int\limits_{\varepsilon}^{\varepsilon_0}\
\frac{dr}{r^{\frac{n-1}{\alpha-1}}\widetilde{q}_{x_0}^{\frac{1}{\alpha-1}}(r)}.$
Заметим, что $f(\Gamma(C_1,C_2, D))=\Gamma(f(C_1), f(C_2), f(D)),$
так что из (\ref{eq9C}) вытекает, что
$$
M_{\alpha}(f(\Gamma(C_1,C_2, D)))\leqslant
\frac{C_2}{I^{\alpha-1}}\,.
$$
Завершает доказательство применение предложения \ref{pr1A}.~$\Box$
\end{proof}

\medskip
{\bf 8. Равностепенная непрерывность семейств отображений в
замыкании области}. В этом разделе получены некоторые приложения
теорем о граничном продолжении отображений, доказанных выше. Речь
идёт, прежде всего, о кольцевых $Q$-отображениях относительно
$p$-модуля. Ниже будет показано, что семейства таких отображений
равностепенно непрерывны $\overline{D}_P$ при не очень сильных
ограничениях на многообразия и рассматриваемые области их
определения. Для простоты ограничимся случаем, когда все отображения
рассматриваемого семейства являются гомеоморфизмами.

\medskip
Семейство $\frak{F}$ отображений $f\colon X\rightarrow
{X}^{\,\prime}$ называется {\it равностепенно непрерывным в точке}
$x_0 \in X,$ если для любого $\varepsilon>0$ найдётся $\delta>0$
такое, что ${d}^{\,\prime} \left(f(x),f(x_0)\right)<\varepsilon$ для
всех таких $x,$ что $d(x,x_0)<\delta$ и для всех $f\in \frak{F}.$
Говорят, что $\frak{F}$ {\it равностепенно непрерывно}, если
$\frak{F}$ равностепенно непрерывно в каждой  точке $x_0\in X.$
Согласно одной из версий теоремы Арцела--Асколи (см., напр.,
\cite[пункт~20.4]{Va}), если $\left(X,\,d\right)$ --- сепарабельное
метрическое пространство, а $\left(X^{\,\prime},\,
d^{\,\prime}\right)$ --- компактное метрическое пространство, то
семейство $\frak{F}$ отображений $f\colon X\rightarrow
{X}^{\,\prime}$ нормально тогда  и только тогда, когда  $\frak{F}$
равностепенно непрерывно.

\medskip
Понятия регулярности области по Альфорсу, а также сведения,
относящиеся к $(1; p)$-неравенству Пуанкаре, могут быть найдены,
напр., в \cite{AS} (см. также нашу предыдущую публикацию \cite{IS}).
Перед тем, как переходить к изложению основных результатов,
сформулируем и докажем следующее простое утверждение.

\medskip
\begin{proposition}\label{pr2A}
{\sl Пусть $(X, d)$ -- компактное метрическое пространство, и пусть
$x_0\in X$ -- произвольная фиксированная точка. Тогда существует
$R>0:$ $X=B(x_0, R).$}
\end{proposition}

\medskip
\begin{proof}
Предположим противное, а именно, что для каждого $k\in {\Bbb N}$
найдётся элемент $x_k\in X:$ $d(x_0, x_k)>k.$ Так как по условию
пространство $(X, d)$ является компактным, существует
подпоследовательность $x_{k_l}$ последовательности $x_k,$ сходящаяся
к некоторой точке $y_0\in X$ при $l\rightarrow\infty.$ Тогда по
неравенству треугольника $d(x_0, x_{k_l})\leqslant d(x_0,
y_0)+d(y_0, x_{k_l})<d(x_0, y_0)+1$ при всех $l\geqslant l_0$ и
некотором $l_0\in {\Bbb N}.$ Полученное соотношение противоречит
неравенству $d(x_0, x_k)>k,$ $k=1,2,\ldots ,$ что указывает на
неверность исходного предположения. Значит, $B(x_0, R)=X$ при
некотором достаточно большом $R>0.$~$\Box$
\end{proof}

\medskip
Справедливо следующее утверждение (см.~\cite[предложение~4.7]{AS}).

\medskip
\begin{proposition}\label{pr2}
{\sl Пусть $X$ --- $\beta$-регулярное по Альфорсу метрическое
пространство с мерой, в котором выполняется $(1;
\alpha)$-неравенство Пуанкаре, $\beta-1< \alpha\leqslant \beta.$
Тогда для произвольных континуумов $E$ и $F,$ содержащихся в шаре
$B(x_0, R),$ и некоторой постоянной $C>0$ выполняется неравенство
$$M_{\alpha}(\Gamma(E, F, X))\geqslant
\frac{1}{C}\cdot\frac{\min\{{\rm diam}\,E, {\rm
diam}\,F\}}{R^{1+\alpha-\beta}}\,.$$ }
\end{proposition}

\medskip
Как и ранее, положим
$A=A(x_0, r_1, r_2)=\{x\in {\Bbb M}^n\,|\,r_1<d(x, x_0)<r_2\}.$
Имеет место следующее утверждение, обобщающее
\cite[лемма~3.1]{Sev$_3$} в случае не локально связных границ. (Для
пространства ${\Bbb R}^n$ и не только локально связных границ это
утверждение также было получено ранее, см. \cite{SP}).

\medskip
\begin{lemma}\label{lem3B}
{\sl\, Пусть $\alpha\in (n-1, n],$ области $D\subset {\Bbb M}^n$ и
$D^{\,\prime}\subset {\Bbb M}_*^n$ имеют компактные замыкания, кроме
того, область $D$ регулярна и содержит не менее двух различных
простых концов, $D^{\,\prime}$ имеет локально квазиконформную
границу и, одновременно, является пространством $n$-регулярным по
Альфорсу относительно геодезической метрики $d_*$ и меры объёма
$v_*$ в ${\Bbb M}_*^n,$ в котором выполнено $(1;
\alpha)$-неравенство Пуанкаре. Пусть также $P_0$ -- некоторый
простой конец в $E_D,$ а $\sigma_m,$ $m=1,2,\ldots,$ --
соответствующая ему цепь разрезов, лежащих на сферах с центром в
некоторой точке $x_0\in
\partial D$ и радиусов $r_m\rightarrow 0,$ $m\rightarrow\infty.$
Пусть $D_m$ -- соответствующая $P_0$ последовательность
ассоциированных областей, а $C_m$ -- произвольная последовательность
континуумов, принадлежащих $D_m.$ Потребуем, кроме того, чтобы $C(f,
P_1)\cap C(f, P_2)=\varnothing$ для произвольных различных простых
концов $P_1, P_2\in E_D.$

Предположим, $f:D\rightarrow D^{\,\prime}$ -- кольцевой
$Q$-гомеоморфизм относительно $\alpha$-модуля в $\overline{D},$
$f(D)=D^{\,\prime},$ такой что $b_0^{\,\prime}=f(b_0)$ для некоторых
$b_0\in D$ и $b_0^{\,\prime}\in D^{\,\prime}.$ Пусть также найдётся
$\varepsilon_0=\varepsilon(x_0)>0,$
такое, что при некотором $0<p^{\,\prime}<\alpha$ выполнено условие
\begin{equation}\label{eq5***}
\int\limits_{A(x_0, \varepsilon, \varepsilon_0)}
Q(x)\cdot\psi^{\,\alpha}(d(x, x_0))\,dv(x) \leqslant K\cdot
I^{p^{\,\prime}}(\varepsilon, \varepsilon_0)\,,
\end{equation}
где $\psi$ -- некоторая неотрицательная измеримая функция, такая,
что при всех $\varepsilon\in(0, \varepsilon_0)$
\begin{equation}\label{eq7E}
I(\varepsilon,
\varepsilon_0):=\int\limits_{\varepsilon}^{\varepsilon_0}\psi(t)\,dt
< \infty\,,
\end{equation}
при этом, $I(\varepsilon, \varepsilon_0)\rightarrow \infty$ при
$\varepsilon\rightarrow 0.$

Тогда найдутся число $R>0,$ зависящее только от области
$D^{\,\prime}$ и числа
$\widetilde{\varepsilon_0}=\widetilde{\varepsilon_0}(x_0)\in (0,
\varepsilon_0),$ $M_0\in {\Bbb N}$ такие, что
$$d_*(f(C_m))\leqslant C\cdot R^{1+\alpha-n}\cdot K\cdot
I^{p^{\,\prime}-\alpha}(r_m, \varepsilon_0)\cdot \Delta(\sigma,
\widetilde{\varepsilon_0}, \varepsilon_0)\,,\quad m\geqslant M_1\,,
$$
где
\begin{equation}\label{eq1.3}
\Delta(\sigma, \widetilde{\varepsilon_0}, \varepsilon_0)=\left(
1+\frac{\int\limits_{\widetilde{\varepsilon_0}}^{\varepsilon_0}\psi(t)\,dt}
{\int\limits_{\sigma}^{\widetilde{\varepsilon_0}}\psi(t)\,dt}\right)^{\alpha}\,,\end{equation}
а $C$ -- постоянная из предложения \ref{pr2}.}
\end{lemma}

\medskip
\begin{proof} Ввиду теоремы \ref{th2} мы можем считать, что $\overline{D^{\,\prime}}_P=\overline{D^{\,\prime}}.$
Заметим, что ввиду предложения \ref{pr2} граница области
$D^{\,\prime}$ является сильно достижимой относительно
$\alpha$-модуля, так что по лемме \ref{lem1A} и ввиду условия $C(f,
P_1)\cap C(f, P_2)=\varnothing,$ $P_1, P_2\in E_D,$ отображение $f$
продолжается до гомеоморфизма $f:\overline{D}_P\rightarrow
\overline{D^{\,\prime}}.$ Тогда учитывая, что по условию леммы
\ref{lem3B} найдутся не менее двух простых концов $P_1, P_2\in E_D,$
мы заключаем, что граница области $D^{\,\prime}$ содержит не менее
двух точек. В частности, $\partial D^{\,\prime}\ne\varnothing.$
Далее, поскольку $\overline{D^{\,\prime}}$ -- компакт, по
предложению \ref{pr2A} существуют точка $\overline{x_0}\in
D^{\,\prime}$ и число $R=R(D^{\,\prime}):$
$\overline{D^{\,\prime}}\subset B(\overline{x_0}, R).$

Пусть теперь $P_1\in E_D$ -- простой конец, не совпадающий с $P_0,$
где $P_0$ -- фиксированный простой конец из условия леммы. (Такой
простой конец существует ввиду условия леммы~\ref{lem3B}).
Предположим, $G_m,$ $m=1,2,\ldots,$ -- последовательность областей,
соответствующая простому концу $P_1$ и $x_m\in G$ -- произвольная
последовательность точек, такая что $x_m\rightarrow P_1$ при
$m\rightarrow\infty.$ Можно считать, что $x_m\in G_m$ для всякого
$m\in {\Bbb N}.$ Тогда, так как $f$ имеет непрерывное продолжение на
$\overline{D}_P$ ввиду леммы \ref{lem1A}, то $f(x_m)\rightarrow
f(P_1)$ при $m\rightarrow\infty.$ Заметим, что при всех $m\geqslant
m_0$ и некотором $m_0\in {\Bbb N}$ ввиду неравенства треугольника,
применённого по отношению к метрике $d_*,$
$$d_*(f(b_0), f(x_m))=d_*(b_0^{\,\prime}, f(x_m))\geqslant$$
\begin{equation}\label{eq3C}
\geqslant d_*(b_0^{\,\prime}, f(P_1))-d_*(f(x_m), f(P_1))\geqslant
\frac{1}{2}\cdot d_*(b_0^{\,\prime}, \partial
D^{\,\prime}):=\delta\,,
\end{equation}
где $d_*(b_0^{\,\prime}, \partial D^{\,\prime})$ обозначает
геодезическое расстояние от точки $b_0^{\,\prime}$ до $\partial
D^{\,\prime}.$ (В частности, здесь мы использовали, что $\partial
D^{\,\prime}\ne\varnothing.$). Построим последовательность
континуумов $K_m,$ $m=1,2,\ldots,$ следующим образом. Соединим точку
$x_1$ с точкой $b_0$ произвольной кривой в $D,$ которую мы обозначим
через $K_1.$ Далее, соединим точки $x_2$ и $x_1$ кривой
$K_1^{\prime},$ лежащей в $G_1.$ Объединив кривые $K_1$ и
$K_1^{\prime},$ получим кривую $K_2,$ соединяющую точки $b_0$ и
$x_2.$ И так далее. Пусть на некотором шаге мы имеем кривую $K_m,$
соединяющую точки $x_m$ и $b_0.$ Соединим точки $x_{m+1}$ и $x_m$
кривой $K_m^{\,\prime},$ лежащей в $G_m.$ Объединяя между собой
кривые $K_m$ и $K_m^{\,\prime},$ получим кривую $K_{m+1}.$ И так
далее.

Покажем, что найдётся номер $m_1\in {\Bbb N},$ такой что
\begin{equation}\label{eq4B}
D_m\cap K_m=\varnothing\quad\forall\quad m\geqslant m_1\,.
\end{equation}
Предположим, что (\ref{eq4B}) не имеет места, тогда найдутся
возрастающая последовательность номеров $m_k\rightarrow\infty,$
$k\rightarrow\infty,$ и последовательность точек $\xi_k\in
K_{m_k}\cap D_{m_k},$ $k=1,2,\ldots,\,.$ Тогда, с одной стороны,
$\xi_k \rightarrow P_0$ при $k\rightarrow\infty.$

\medskip
Рассмотрим следующую процедуру. Заметим, что возможны два случая:
либо все элементы $\xi_k$ при $k=1,2,\ldots$ принадлежат множеству
$D\setminus G_1,$ либо найдётся номер $k_1$ такой, что $\xi_{k_1}\in
G_1.$ Далее, рассмотрим последовательность $\xi_k,$ $k>k_1.$
Заметим, что возможны два случая: либо $\xi_k$ при $k>k_1$
принадлежат множеству $D\setminus G_2,$ либо найдётся номер
$k_2>k_1$ такой, что $\xi_{k_2}\in G_2.$ И так далее. Предположим,
элемент $\xi_{k_{l-1}}\in G_{l-1}$ построен. Заметим, что возможны
два случая: либо $\xi_k$ при $k>k_{l-1}$ принадлежат множеству
$D\setminus G_l,$ либо найдётся номер $k_l>k_{l-1}$ такой, что
$\xi_{k_l}\in G_l.$ И так далее. Эта процедура может быть как
конечной (оборваться на каком-то $l\in {\Bbb N}$), так и
бесконечной, в зависимости от чего мы имеем две ситуации:

1) либо найдутся номера $n_0\in {\Bbb N}$ и $l_0\in {\Bbb N}$ такие,
что $\xi_k\in D\setminus G_{n_0}$ при всех $k>l_0;$

2) либо для каждого $l\in {\Bbb N}$ найдётся элемент $\xi_{k_l}$
такой, что $\xi_{k_l}\in G_l,$ причём последовательность $k_l$
является возрастающей по $l\in {\Bbb N}.$

\medskip
Рассмотрим каждый из этих случаев и покажем, что в обоих из них мы
приходим к противоречию. Пусть имеет место ситуация 1), тогда
заметим, что все элементы последовательности $\xi_k$ принадлежат
$K_{n_0},$ откуда вытекает существование подпоследовательности
$\xi_{k_r},$ $r=1,2,\ldots,$ сходящейся при $r\rightarrow\infty$ к
некоторой точке $\xi_0\in D.$ Однако, с другой стороны $\xi_k\in
D_{m_k}$ и, значит, $\xi_0\in \bigcap\limits_{m=1}^{\infty}
\overline{D_m}\subset
\partial D$ (см. предложение~\ref{thabc1} по этому поводу).
Полученное противоречие говорит о том, что случай 1) невозможен.
Пусть имеет место случай 2), тогда одновременно $\xi_k\rightarrow
P_0$ и $\xi_k\rightarrow P_1$ при $k\rightarrow\infty.$ В силу
непрерывного продолжения $f$ на $\overline{D}_P$ отсюда имеем, что
$f(\xi_k)\rightarrow f(P_0)$ и $f(\xi_k)\rightarrow f(P_1)$ при
$k\rightarrow\infty,$ откуда $f(P_0)=f(P_1),$ что противоречит
условию $C(f, P_1)\cap C(f, P_2)=\varnothing,$ $P_1\ne P_2.$
Полученное противоречие указывает на справедливость соотношения
(\ref{eq4B}).

Положим теперь $\widetilde{\varepsilon_0}=\min\{\varepsilon_0,
r_{m_1+1}\},$ и пусть $M_0$ -- натуральное число, такое что
$r_m<\widetilde{\varepsilon_0}$ при всех $m\geqslant M_0.$ Заметим,
что ввиду соотношения (\ref{eq4B}) и по определению разрезов
$\sigma_m\subset r_m,$ $\Gamma\left(C_m, K_m, D\right)>\Gamma(S(x_0,
r_m), S(x_0, \widetilde{\varepsilon_0}), D)$ и, значит,
$M_{\alpha}(f(\Gamma\left(C_m, K_m, D\right)))\leqslant
M_{\alpha}(f(\Gamma(S(x_0, r_m), S(x_0, \widetilde{\varepsilon_0}),
D))$ (см. \cite[теорема~6.4]{Va}). Заметим, что функция
$$\eta_{m}(t)= \left\{
\begin{array}{rr}
\psi(t)/I(r_m, \widetilde{\varepsilon_0}), &   t\in (r_m, \widetilde{\varepsilon_0})\\
0,  &  t\not\in (r_m, \widetilde{\varepsilon_0})
\end{array}\,,
\right.$$
%
%
$I(a, b)=\int\limits_a^b\psi(t)\,dt,$ удовлетворяет соотношению вида
(\ref{eq*3!!}) при $r_1:=r_m$ и $r_2:=\widetilde{\varepsilon_0},$
поэтому из последнего соотношения, а также условия (\ref{eq5***}) и
по определению кольцевого $Q$-отображения относительно
$\alpha$-модуля в точке $x_0,$ мы приходим к соотношению
\begin{equation}\label{eq37***}
M_{\alpha}\left(\Gamma\left(f(C_m), f(K_m),
D^{\,\prime}\right)\right)\leqslant K\cdot
I^{p^{\,\prime}-\alpha}(r_m, \varepsilon_0)\cdot\Delta(r_m,
\widetilde{\varepsilon_0}, \varepsilon_0)\,,\quad m\geqslant M_0\,,
\end{equation}
где $\Delta$ определяется из соотношения (\ref{eq1.3}) при
$\sigma=r_m.$ Поскольку из (\ref{eq37***}) вытекает, что
$M_{\alpha}\left(\Gamma\left(f(C_m), f(K_m),
D^{\,\prime}\right)\right)\rightarrow 0$ при $m\rightarrow\infty,$ а
$d_*(f(K_m))\geqslant \delta$ ввиду (\ref{eq3C}), то из предложения
\ref{pr2} вытекает существование некоторого $M_1\geqslant M_0$
такого, что при всех $m\geqslant M_1$
\begin{equation}\label{eq12B}
d_*f(C_m)\leqslant C\cdot R^{1+\alpha-n}\cdot
M_{\alpha}\left(\Gamma\left(f(C_m), f(K_m),
D^{\,\prime}\right)\right)\,,
\end{equation}
где $R$ -- радиус шара, содержащего область $D^{\,\prime},$ а $C>0$
-- постоянная из предложения \ref{pr2}.
Тогда из (\ref{eq37***}) и (\ref{eq12B}) вытекает, что
$d_*(f(C_m))\leqslant C\cdot R^{1+\alpha-n}\cdot K\cdot
I^{p^{\,\prime}-\alpha}(r_m, \varepsilon_0)\cdot\Delta(r_m,
\widetilde{\varepsilon_0}, \varepsilon_0),$ $m\geqslant M_1.$
Лемма доказана.~$\Box$
\end{proof}

\medskip
Для заданных областей $D\subset {\Bbb M}^n,$ $D^{\,\prime} \subset
{\Bbb M}_*^n,$ $n\geqslant 2,$ $n-1<\alpha\leqslant n,$ измеримой по
Лебегу функции $Q(x):{\Bbb M}^n\rightarrow [0, \infty],$ $Q(x)\equiv
0$ на ${\Bbb M}^n\setminus D,$ $b_0\in D,$ $b_0^{\,\prime}\in
D^{\,\prime},$ обозначим через $\frak{G}_{\alpha, b_0,
b_0^{\,\prime}, Q}\left(D, D^{\,\prime}\right)$ семейство всех
кольцевых $Q$-гомеоморфизмов $f:D\rightarrow D^{\,\prime}$
относительно $\alpha$-модуля в $\overline{D},$ таких что
$f(D)=D^{\,\prime},$ $b_0^{\,\prime}=f(b_0).$ В наиболее общей
ситуации основное утверждение настоящего раздела может быть
сформулировано следующим образом.

\medskip
\begin{lemma}\label{lem3A}
{\sl\, Пусть выполнены все условия леммы \ref{lem3B}, кроме того,
предположим, что многообразие ${\Bbb M}_*^n$ связно, и что найдётся
хотя бы один фиксированный невырожденный континуум $K\subset {\Bbb
M}_*^n\setminus D^{\,\prime}.$ Тогда каждое $f\in\frak{G}_{\alpha,
b_0, b_0^{\,\prime}, Q}\left(D, D^{\,\prime}\right)$ продолжается до
гомеоморфизма $f:\overline{D}_P\rightarrow
\overline{D^{\,\prime}}_P,$ при этом семейство таким образом
продолженных отображений является равностепенно непрерывным в
$\overline{D}_P.$ }
\end{lemma}

\begin{proof} Каждое отображение $f$ имеет непрерывное продолжение на
$\overline{D}_P$ в силу леммы \ref{lem1A}. Равностепенная
непрерывность семейства $\frak{G}_{\alpha, b_0, b_0^{\,\prime},
Q}\left(D, D^{\,\prime}\right)$ во внутренних точках области $D$
следует при $\alpha=n$ из \cite[лемма~2.4]{IS} (см. случай
$\alpha\ne n$ в \cite[лемма~2]{SevSkv}).

Осталось показать равностепенную непрерывность семейства
продолженных по непрерывности гомеоморфизмов $\frak{G}_{\alpha, b_0,
b_0^{\,\prime}, Q}\left(\overline{D}_P,
\overline{D^{\,\prime}}_P\right)$ на $E_D.$

Предположим противное, а именно, что семейство отображений
$\frak{G}_{\alpha, b_0, b_0^{\,\prime}, Q}\left(\overline{D}_P,
\overline{D^{\,\prime}}_P\right)$ не является равностепенно
непрерывным в некоторой точке $P_0\in E_D.$ Тогда найдутся число
$a>0,$ последовательность $P_k\in \overline{D}_P,$ $k=1,2,\ldots$ и
элементы $f_k\in\frak{G}_{b_0, b_0^{\,\prime},
Q}\left(\overline{D}_P, \overline{D^{\,\prime}}_P\right)$ такие, что
$d(P_k, P_0)<1/k$ и
\begin{equation}\label{eq6B}
d_*(f_k(P_k), f_k(P_0))\geqslant a\quad\forall\quad k=1,2,\ldots,\,.
\end{equation}
Ввиду возможности непрерывного продолжения каждого $f_k$ на границу
$D$ в терминах простых концов, для всякого $k\in {\Bbb N}$ найдётся
элемент $x_k\in D$ такой, что $d(x_k, P_k)<1/k$ и $d_*(f_k(x_k),
f_k(P_k))<1/k.$ Тогда из (\ref{eq6B}) вытекает, что
\begin{equation}\label{eq7D}
d_*(f_k(x_k), f_k(P_0))\geqslant a/2\quad\forall\quad
k=1,2,\ldots,\,.
\end{equation}
Аналогично, в силу непрерывного продолжения отображения $f_k$ в
$\overline{D}_P$ найдётся последовательность $x_k^{\,\prime}\in D,$
$x_k^{\,\prime}\rightarrow P_0$ при $k\rightarrow \infty$ такая, что
$d_*(f_k(x_k^{\,\prime}), f_k(P_0))<1/k$ при $k=1,2,\ldots\,.$ Тогда
из (\ref{eq7D}) вытекает, что
\begin{equation}\label{eq8BC}
d_*(f_k(x_k), f_k(x_k^{\,\prime}))\geqslant a/4\quad\forall\quad
k=1,2,\ldots\,,\,.
\end{equation}
Пусть $\sigma_m,$ $m=1,2,\ldots,$ -- соответствующая $P_0$ цепь
разрезов, лежащих на сферах с центром в некоторой точке $x_0\in
\partial D$ и радиусов $r_m\rightarrow 0,$ $m\rightarrow\infty.$
Пусть $D_m$ -- соответствующая $P_0$ последовательность
ассоциированных областей. Не ограничивая общности рассуждений, можно
считать, что $x_k$ и $x_k^{\,\prime}$ принадлежат области $D_k.$
Соединим точки $x_k$ и $x_k^{\,\prime}$ кривой $C_k$ лежащей в
$D_k.$ Тогда по лемме \ref{lem3B} мы получим, что
$d_*(f(C_k))\rightarrow 0$ при $k\rightarrow \infty,$ что
противоречит неравенству (\ref{eq8BC}). Полученное противоречие
указывает на то, что исходное предположение об отсутствии
равностепенной непрерывности семейства $\frak{G}_{\alpha, b_0,
b_0^{\,\prime}, Q}\left(\overline{D}_P,
\overline{D^{\,\prime}}_P\right)$ было неверным.~$\Box$
\end{proof}

\medskip
Из леммы \ref{lem3A}, аналогично \cite[доказательство теорем~1.1 и
2.1]{IS$_1$}, получаем следующее утверждение.

\medskip
\begin{theorem}\label{th8}{\sl\,
Пусть $\alpha\in (n-1, n],$ области $D\subset {\Bbb M}^n$ и
$D^{\,\prime}\subset {\Bbb M}_*^n$ имеют компактные замыкания, кроме
того, область $D$ регулярна и содержит не менее двух различных
простых концов, $D^{\,\prime}$ имеет локально квазиконформную
границу и, одновременно, является пространством $n$-регулярным по
Альфорсу относительно геодезической метрики $d_*$ и меры объёма
$v_*$ в ${\Bbb M}_*^n,$ в котором выполнено $(1;
\alpha)$-неравенство Пуанкаре. Потребуем, кроме того, чтобы $C(f,
P_1)\cap C(f, P_2)=\varnothing$ для произвольных различных простых
концов $P_1, P_2\in E_D,$ чтобы многообразие ${\Bbb M}_*^n$ было
связно, и чтобы нашёлся хотя бы один фиксированный невырожденный
континуум $K\subset {\Bbb M}_*^n\setminus D^{\,\prime}.$

Предположим, что выполнено одно из следующих условий:

1) либо в каждой точке $x_0\in \overline{D}$ при некотором
$\varepsilon_0=\varepsilon_0(x_0)>0$ и всех
$0<\varepsilon<\varepsilon_0$ выполнены условия типа (\ref{eq10A});
2) либо $Q\in FMO(\overline{D}).$
Тогда каждое $f\in\frak{G}_{\alpha, b_0, b_0^{\,\prime}, Q}\left(D,
D^{\,\prime}\right)$ продолжается до гомеоморфизма
$f:\overline{D}_P\rightarrow \overline{D^{\,\prime}}_P,$ при этом
семейство таким образом продолженных отображений является
равностепенно непрерывным в $\overline{D}_P.$}
\end{theorem}

\medskip\medskip
\begin{remark}\label{rem3}
Отметим, что условие $C(f, P_1)\cap C(f, P_2)=\varnothing$ для
произвольных различных простых концов $P_1, P_2\in E_D,$
присутствующее в формулировках лемм \ref{lem3B}, \ref{lem3A} и
теоремы \ref{th8}, автоматически выполнено при $p=n$ ввиду леммы
\ref{l:9.1}.~$\Box$
\end{remark}

\medskip
Доказательство следующей леммы может быть найдено, напр., в
\cite[теорема~4.3]{IS$_1$}.

\medskip
\begin{lemma}{}\label{thOS4.1A} Пусть $D$~--- область в ${\Bbb M}^n,$
$n\geqslant 3,$ $\varphi\colon(0,\infty)\to (0,\infty)$~---
неубывающая функция{\em,} удовлетворяющая условию~\eqref{eqOS3.0a}.
Если $p>n-1,$ то каждое открытое дискретное отображение $f\colon
D\to {\Bbb M}_*^n$ с конечным искажением класса
$W^{1,\varphi}_{loc},$ такое{\em,} что $N(f, D)<\infty,$ является
нижним $Q$-отображением относительно $p$-модуля в каждой точке
$x_0\in\overline{D}$ при
 $$
Q(x)=N(f, D) K^{\frac{p-n+1}{n-1}}_{I, \alpha}(x, f),\quad
\alpha:=\frac{p}{p-n+1}\,,
 $$
где внутренняя дилатация $K_{I,\alpha}(x, f)$ отображения $f$ в
точке $x$ порядка $\alpha$ определена
соотношением~\eqref{eq0.1.1A}{\em,} а кратность $N(f, D)$ определена
вторым соотношением в~\eqref{eq1.7A}.
 \end{lemma}

\medskip
Для областей $D\subset {\Bbb M}^n,$ $D^{\,\prime}\subset {\Bbb
M}_*^n,$ $b_0\in D,$ $b_0^{\,\prime}\in D^{\,\prime}$ и произвольной
измеримой по Лебегу функции $Q(x): {\Bbb M}^n\rightarrow [0,
\infty],$ $Q(x)\equiv 0$ при $x\not\in D,$ обозначим символом
$\frak{F}_{\alpha, b_0, b_0^{\,\prime}, \varphi, Q}(D,
D^{\,\prime})$ семейство всех гомеоморфизмов $f:D\rightarrow
D^{\,\prime}$ класса $W_{loc}^{1, \varphi}$ в $D,$
$f(D)=D^{\,\prime},$ таких что $K_{I, \alpha}(x, f)\leqslant Q(x)$
при почти всех $x\in D$ и $f(b_0)=b_0^{\,\prime}.$ Справедливо
следующее утверждение.

\medskip
\begin{theorem}\label{th7A}{\sl\,
Пусть $\alpha\in (n-1, n],$ области $D\subset {\Bbb M}^n$ и
$D^{\,\prime}\subset {\Bbb M}_*^n$ имеют компактные замыкания, кроме
того, область $D$ регулярна и содержит не менее двух различных
простых концов, $D^{\,\prime}$ имеет локально квазиконформную
границу и, одновременно, является пространством $n$-регулярным по
Альфорсу относительно геодезической метрики $d_*$ и меры объёма
$v_*$ в ${\Bbb M}_*^n,$ в котором выполнено $(1;
\alpha)$-неравенство Пуанкаре. Потребуем, кроме того, чтобы $C(f,
P_1)\cap C(f, P_2)=\varnothing$ для произвольных различных простых
концов $P_1, P_2\in E_D,$ чтобы многообразие ${\Bbb M}_*^n$ было
связно, и чтобы нашёлся хотя бы один фиксированный невырожденный
континуум $K\subset {\Bbb M}_*^n\setminus D^{\,\prime}.$
Предположим, $Q\in L_{loc}^1({\Bbb M}^n),$ заданная неубывающая
функция $\varphi:[0,\infty)\rightarrow[0,\infty)$ удовлетворяет
условию
\begin{equation}\label{eqOS3.0aa}
\int\limits_{1}^{\infty}\left(\frac{t}{\varphi(t)}\right)^
{\frac{1}{n-2}}dt<\infty\,,
\end{equation}
и что для каждого $x_0\in \overline{D}$ выполнено одно из следующих
условий:

1) либо $Q\in FMO(\overline{D});$

2) либо в каждой точке $x_0\in \overline{D}$ при некотором
$\varepsilon_0=\varepsilon_0(x_0)>0$ и всех
$0<\varepsilon<\varepsilon_0$
\begin{equation}\label{eq1H}
\int\limits_{\varepsilon}^{\varepsilon_0}
\frac{dt}{t^{\frac{n-1}{\alpha-1}}q_{x_0}^{\,\frac{1}{\alpha-1}}(t)}<\infty\,,\qquad
\int\limits_{0}^{\varepsilon_0}
\frac{dt}{t^{\frac{n-1}{\alpha-1}}q_{x_0}^{\,\frac{1}{\alpha-1}}(t)}=\infty\,,
\end{equation}
где
$q_{x_0}(r):=\frac{1}{r^{n-1}}\int\limits_{|x-x_0|=r}Q(x)\,d{\mathcal
H}^{n-1}.$
Тогда каждый элемент $f\in \frak{F}_{\alpha, b_0, b_0^{\,\prime},
\varphi, Q}(D, D^{\,\prime})$ продолжается до непрерывного
отображения $\overline f\colon\overline D_P\rightarrow\overline
{D^{\,\prime}}_P$, при этом, семейство отображений $\frak{F}_{b_0,
b_0^{\,\prime}, \varphi, Q}(\overline{D}_P,
\overline{D^{\,\prime}}_P),$ состоящее из всех продолженных таким
образом отображений, является равностепенно непрерывным, а значит, и
нормальным  в $\overline{D}_P$.}
\end{theorem}

\medskip
\begin{proof}
Пусть $p$ -- число, определяющее из условия
$\alpha=\frac{p}{p-n+1},$ т.е., $p:=\alpha(n-1)/(\alpha-1))$. Тогда
$p>n-1$ и по лемме \ref{thOS4.1A} каждое отображение
$\frak{F}_{\alpha, b_0, b_0^{\,\prime}, \varphi, Q}(D,
D^{\,\prime})$ является нижним $B$-отображением относительно
$p$-модуля при $B(x)=Q^{\frac{p-n+1}{n-1}}(x, f)$ в $\overline{D}.$
По теореме \ref{th4AB} отображение $f$ является кольцевым
$B^{\frac{n-1}{p-n+1}}(x)$-ото\-бра\-же\-нием в $\overline{D}$
относительно $\alpha$-модуля. Другими словами, поскольку
$B^{\frac{n-1}{p-n+1}}(x)=Q(x),$ то $f$ является кольцевым
$Q(x)$-отображением в $\overline{D}$ относительно $\alpha$-модуля,
т.е., $f\in\frak{G}_{\alpha, b_0, b_0^{\,\prime}, Q}\left(D,
D^{\,\prime}\right).$ Так как $Q(x)$ удовлетворяет условиям 1) и 2)
теоремы \ref{th8}, то желанное утверждение непосредственно вытекает
из этой теоремы.~$\Box$
\end{proof}

\medskip
\begin{remark}\label{rem4}
Если в условиях теорем \ref{th8} и \ref{th7A} потребовать условие
локальной связности области $D$ на границе вместо её регулярности,
то можно показать, что в этом случае мы имеем дело с поточечным
граничным соответствием $\overline
f\colon\overline{D}\rightarrow\overline {D^{\,\prime}},$ при этом,
имеет место равностепенная непрерывность соответствующих семейств,
понимаемую как р.н. между метрическими пространствами
$(\overline{D}, d)$ и $(\overline{D^{\prime}}, d_*).$
\end{remark}

КОНТАКТНАЯ ИНФОРМАЦИЯ

\medskip
\noindent{{\bf Денис Петрович Ильютко} \\
МГУ имени М.\,В.\,Ломоносова \\
кафедра дифференциальной геометрии и приложений, мехмат факультет,\\
Ленинские горы, ГЗ МГУ, ГСП-1, г.~Москва, Россия, 119991\\
тел. +7 495 939 39 40, e-mail: ilyutko@yandex.ru}

\medskip
\noindent{{\bf Евгений Александрович Севостьянов} \\
Житомирский государственный университет им.\ И.~Франко\\
кафедра математического анализа, ул. Большая Бердичевская, 40 \\
г.~Житомир, Украина, 10 008 \\ тел. +38 066 959 50 34 (моб.),
e-mail: esevostyanov2009@mail.ru}


\begin{thebibliography}{99}
{\small

\bibitem{ABBS} {\it Adamowicz~T., Bj\"{o}rn~A., Bj\"{o}rn~J.,
Shanmugalingam~N.} Nageswari Prime ends for domains in metric spaces
// Adv. Math. -- 2013. -- V.\,\textbf{238}. -- P.~459-–505.

\bibitem{A} {\it Adamowicz~T.} Prime ends in metric spaces and boundary extensions of
mappings, www. arxiv. org, arXiv:1608.02393.


\bibitem{BGMV} {\it Bishop~C.J., Gutlyanski\u\i~V.Ya., Martio~O. and
Vuorinen~M.} On conformal dilatation in space // Int. J. Math. Math.
Sci. -- 2003. -- V.~\textbf{22}. -- P.~1397--1420.


\bibitem{Cr$_2$} {\it Cristea~M.}  Open discrete mappings having local $ACL^n$
inverses // Complex Variables and Elliptic Equations. -- 2010. --
V.~\textbf{55}, no.~1--3. -- P.~61--90.


\bibitem{Gol$_2$} {\it Golberg~A.} Differential properties of
$(\alpha, Q)$-ho\-me\-o\-mor\-phisms,  Further Progress in Analysis,
World Scientific Publ., 2009, 218--228.

\bibitem{GRSY} {\it Gutlyanskii~V.Ya., Ryazanov~V.I., Srebro~U., Yakubov~E.}
The Beltrami Equation: A Geometric Approach. -- New York etc.:
Springer, 2012.


\bibitem{MRSY} {\it Martio~O., Ryazanov~V., Srebro~U. and Yakubov
~E.} Moduli in Modern Mapping Theory. -- New York: Springer Science
+ Business Media, LLC, 2009.


\bibitem{MRV$_1$} {\it Martio~O., Rickman~S.,
V\"{a}is\"{a}l\"{a}~J.} Definitions for quasiregular mappings //
Ann. Acad. Sci. Fenn. Ser. A I. Math. -- 1969. -- V.~\textbf{448}.
-- P.~1--40.

\bibitem{MRV$_2$} {\it Martio~O., Rickman~S.,
V\"{a}is\"{a}l\"{a}~J.} Distortion and singularities of quasiregular
mappings // Ann. Acad. Sci. Fenn. Ser. A1. -- 1970. --
V.~\textbf{465}. -- P.~1--13.



\bibitem{Re} {\it Решетняк~Ю.~Г.} Пространственные отображения с ограниченным
искажением. -- Новосибирск: Наука, 1982.

\bibitem{Ri} {\it Rickman~S.} Quasiregular mappings. -- Berlin etc.: Springer-Verlag,
1993.

\bibitem{Vu} {\it Vuorinen~M.} Conformal Geometry and Quasiregular Mappings, Lecture Notes in
Math., 1319. -- Berlin etc.: Springer--Verlag, 1988.


\bibitem{KR} {\it Ковтонюк Д.А., Рязанов В.И.} Простые концы и
классы Орлича--Соболева // Алгебра и анализ. -- 2015. -- Т.
\textbf{27}, № 5. -- С. 81--116.

\bibitem{GRY} {\it Gutlyanskii~V., Ryazanov~V.,
Yakubov~E.} The Beltrami equations and prime ends // Український
математичний вiсник. -- 2015. -- \textbf{12}, № 1. -- С. 27–-66.

\bibitem{KRSS} {\it Ковтонюк~Д.А., Рязанов~В.И., Салимов~Р.Р.,
Севостьянов~Е.А.} К теории классов Орлича--Соболева // Алгебра и
анализ. -- 2013. -- \textbf{25}, № 6. -- С.~50--102.

\bibitem{KSS}  {\it Ковтонюк Д., Салимов Р., Севостьянов Е.} К теории отображений
классов Соболева и Орлича-Соболева (под редакцией В.И. Рязанова). --
Киев: Наукова думка, 2013.

\bibitem{Sev$_3$} {\it Севостьянов~Е.А.} О равностепенной
непрерывности гомеоморфизмов с неограниченной характеристикой //
Математические труды. -- 2012. -- \textbf{15}, № 1. -- С. 178--204.

\bibitem{ARS}  {\it Афанасьева Е.С., Рязанов В.И., Салимов Р.Р.}
Об отображениях в классах Орлича–Соболева на римановых многообразиях
// Укр. мат. вестник. -- 2011. -- \textbf{8}, no. 3. -- С.~319--342.

\bibitem{Af$_1$}  {\it Смоловая~Е.С.} Граничное поведение
кольцевых $Q$-гомеоморфизмов в метрических пространствах // Укр.
мат. журн. -- 2012. -- \textbf{62}, no. 5. -- С.~682--689.

\bibitem{Af$_2$} {\it Афанасьева~Е.С.} Граничное поведение
кольцевых $Q$-гомеоморфизмов на римановых многообразиях // Укр. мат.
журн. -- 2011. -- \textbf{63}, № 10. --  С.~1299-1313.


\bibitem{IS}  {\it Ильютко~Д.П., Севостьянов~Е.А.} Об открытых
дискретных отображениях с неограниченной характеристикой на
римановых многообразиях // Мат. Сборник. -- 2016. -- \textbf{207},
№~4. -- С.~65--112.


\bibitem{Na} {\it N\"akki R.} Prime ends and quasiconformal mappings
// J. Anal. Math. -- 1979. -- V. \textbf{35}. -- P. 13-40.

\bibitem{Fu} {\it Fuglede~B.}  Extremal length and functional
completion // Acta Math. -- 1957. -- \textbf{98.} -- P.~171--219.

\bibitem{Wh} {\it Whyburn G.T.} Analytic topology. -- American Mathematical
Society, Rhode Island. -- 1942.

\bibitem{Va} {\it V\"{a}is\"{a}l\"{a} J.} Lectures on $n$-Dimensional Quasiconformal
Mappings. Lecture Notes in Math., V. 229. -- Berlin etc.:
Springer--Verlag, 1971.

\bibitem{Ku} {\it Куратовский~К.} Топология, Т.~2. -- М.: Мир, 1969.

\bibitem{Lee} {\it Lee~J.\,M.} Riemannian Manifolds: An
Introduction to Curvature. -- Springer: New York, 1997.

\bibitem{IS$_1$}  {\it Ильютко~Д.П., Севостьянов~Е.А.} О граничном поведении открытых дискретных отображений на римановых
многообразиях // Мат. Сборник, 51 стр. (принята к публикации).

\bibitem{Vu$_1$} {\it Vuorinen M.} Exceptional sets and boundary behavior of quasiregular
mappings in $n$-space // Ann. Acad. Sci. Fenn. Ser. A 1. Math.
Dissertationes. -- 1976. -- \textbf{11}. -- P. 1--44.

\bibitem{Sev$_2$} {\it Севостьянов Е.А.}  О граничном продолжении и равностепенной непрерывности семейств
отображений в терминах простых концов // Алгебра и анализ (в
печати); опубликована в виде электронного препринта by {\it
Sevost'yanov E.,} On boundary behavior of mappings in terms of prime
ends, www.arxiv.org, arXiv:1602.00660.

\bibitem{Zi} {\it Ziemer~W.P.} Extremal length and conformal
capacity // Trans. Amer. Math. Soc. -- 1967. -- \textbf{126}, no. 3.
-- P.~460--473.

\bibitem{Zi$_1$}  {\it Ziemer~W.P.} Extremal length and
$p$-capacity // Michigan Math. J. -- 1969. --  \textbf{16}. --
P.~43--51.

\bibitem{Shl}  {\it Шлык~В.А.} О равенстве $p$-емкости и
$p$-модуля // Сиб. матем. журн. -- 1993. -- V. \textbf{34}, № 6. --
С. 216-–221.

\bibitem{KR$_1$} {\it Kovtonuyk~D. and Ryazanov~V.} New modulus
estimates in Orlicz-Sobolev classes // Annals of the University of
Bucharest (mathematical series). -- 2014. --  \textbf{5 (LXIII)}. --
P.~131--135.

\bibitem{Va$_1$} {\it V\"{a}is\"{a}l\"{a}~J.} Discrete open mappings on manifolds // Ann. Acad. Sci. Fenn. Ser. A
I. -- 1966. -- \textbf{392}. -- P.~1--10.

\bibitem{Pol} {\it Полецкий~Е.\,А.} Метод модулей для
негомеоморфных квазиконформных отображений // Мат. сб. -- 1970. --
\textbf{83(125)}, № 2(10). -- С.~261--272.


\bibitem{Smol} {\it Смоловая Е.С.} Граничное поведение кольцевых
$Q$-го\-ме\-о\-мор\-физ\-мов в метрических пространствах // Укр.
матем. ж. -- 2010. -- \textbf{62} (2010), № 5. -- С.~682--689.


\bibitem{AS} {\it Adamowicz~T. and Shanmugalingam~N.}
Non-conformal Loewner type estimates for modulus of curve families
// Ann. Acad. Sci. Fenn. Math. -- 2010. -- V. \textbf{35}. -- P. 609–-626.


\bibitem{SP} {\it Севостьянов Е.А., Петров Е.А.} О равностепенной непрерывности гомеоморфизмов классов Соболева
и Орлича-Соболева в замыкании области // Укр. мат. ж., 14 с. (в
печати)


\bibitem{SevSkv} {\it Севостьянов Е.А., Скворцов С.А.} О равностепенной непрерывности обобщённых квазиизометрий на римановых
многообразиях // Математични Студии. -- 2016. -- Т. \textbf{45}, №
2. -- С. 159--169.}

\end{thebibliography}
\end{document}